\documentclass[a4paper, twoside, 11pt, english]{article}

\usepackage[T1]{fontenc}
\usepackage[utf8]{inputenc}
\usepackage{etoolbox}
\usepackage{amsmath,amsthm,amssymb,stmaryrd}
\usepackage{mathtools}
\usepackage[ttscale=.875]{libertine}
\usepackage[libertine]{newtxmath}
\usepackage[numbers,sort&compress]{natbib}
\usepackage{comment}
\usepackage{ifthen}
\usepackage{algpseudocode}
\usepackage{algorithm2e}
\usepackage{multicol}
\usepackage{multirow}
\usepackage{fancyhdr, titlesec, url, enumerate, microtype,setspace}
\usepackage{tikz-qtree,varwidth}
\usepackage{youngtab}
\usepackage{ytableau}
\usepackage[english]{babel}
\usepackage{hyperref}
\usepackage[all,knot,poly]{xy}
\usepackage{tikz, tikz-3dplot, pgfplots}
\usepackage[normalem]{ulem}
\usepackage{graphicx} 
\usepackage{float}
\usetikzlibrary{decorations.pathreplacing}
\usetikzlibrary[positioning, patterns]
\usetikzlibrary{arrows.meta}

\newcommand{\caseconfluence}[1]{%
\[
\scalebox{0.9}{$#1$}
\]
}

\newcommand{\npath}[1]{\Lambda_{#1}}

\usetikzlibrary{calc}
\CompileMatrices
\titleformat{\section}[block]
  {\scshape\filcenter\Large}{\thesection.}{.5em}{}
\titleformat{\subsection}[runin]
  {\bfseries}{\thesubsection.}{.5em}{}[.]
\titleformat{\subsubsection}[runin]
  {\bfseries}{\thesubsubsection.}{.5em}{}[.]
\titlespacing{\subsubsection}{0pt}{10pt}{.5em}

\newtheoremstyle{ntheorem}
  {\topsep}
  {\topsep}
  {\itshape}
  {0pt}
  {\bfseries}
  {.}
  {.5em}
  {\thmname{#1}\thmnumber{ #2}\thmnote{ (#3)}}

\theoremstyle{ntheorem}
\newtheorem{theorem}[subsubsection]{Theorem}
\newtheorem{proposition}[subsubsection]{Proposition}
\newtheorem{lemma}[subsubsection]{Lemma}

\theoremstyle{definition}

\newtheorem{example}[subsubsection]{Example}
\newtheorem{remark}[subsubsection]{Remark}

\fancypagestyle{plain}{%
  \fancyhf{}%
  \fancyfoot[LE,RO]{\thepage}%
}

\newcount\hh
\newcount\mm
\mm=\time
\hh=\time
\divide\hh by 60
\divide\mm by 60
\multiply\mm by 60
\mm=-\mm
\advance\mm by \time
\newcommand{\hhmm}{\number\hh:\ifnum\mm<10 0\fi\number\mm}

\UseTips
\SelectTips{eu}{11}

\newdir{ >}{{}*!/-10pt/@{>}}
\newdir{ -}{{}*!/-10pt/@{}}
\newdir{> }{{}*!/+10pt/@{>}}

\makeatletter

\xyletcsnamecsname@{dir4{}}{dir{}}
\xydefcsname@{dir4{-}}{\line@ \quadruple@\xydashh@}
\xydefcsname@{dir4{.}}{\point@ \quadruple@\xydashh@}
\xydefcsname@{dir4{~}}{\squiggle@ \quadruple@\xybsqlh@}
\xydefcsname@{dir4{>}}{\Tttip@}
\xydefcsname@{dir4{<}}{\reverseDirection@\Tttip@}

\xydef@\quadruple@#1{%
  \edef\Drop@@{%
    \dimen@=#1\relax
    \dimen@=.5\dimen@
    \A@=-\sinDirection\dimen@
    \B@=\cosDirection\dimen@
    \setboxz@h{%
      \setbox2=\hbox{\kern3\A@\raise3\B@\copy\z@}%
      \dp2=\z@ \ht2=\z@ \wd2=\z@ \box2
      \setbox2=\hbox{\kern\A@\raise\B@\copy\z@}%
      \dp2=\z@ \ht2=\z@ \wd2=\z@ \box2
      \setbox2=\hbox{\kern-\A@\raise-\B@\copy\z@}%
      \dp2=\z@ \ht2=\z@ \wd2=\z@ \box2
      \setbox2=\hbox{\kern-3\A@\raise-3\B@\noexpand\boxz@}%
      \dp2=\z@ \ht2=\z@ \wd2=\z@ \box2
    }%
    \ht\z@=\z@ \dp\z@=\z@ \wd\z@=\z@ \noexpand\styledboxz@
  }%
}

\xydef@\Tttip@{%
  \kern2pt \vrule height2pt depth2pt width\z@
  \Tttip@@ \kern2pt \egroup
  \U@c=0pt \D@c=0pt \L@c=0pt \R@c=0pt \Edge@c={\circleEdge}%
  \def\Leftness@{.5}\def\Upness@{.5}%
  \def\Drop@@{\styledboxz@}%
  \def\Connect@@{\straight@{\dottedSpread@\jot}}%
}

\xydef@\Tttip@@{%
  \dimen@=.25\dimen@
  \B@=\cosDirection\dimen@
  \setboxz@h\bgroup
  \reverseDirection@\line@
  \wdz@=\z@ \ht\z@=\z@ \dp\z@=\z@
  {\vDirection@(1,-1)\xydashl@ \xyatipfont\char\DirectionChar}%
  {\vDirection@(1,+1)\xydashl@ \xybtipfont\char\DirectionChar}%
}

\gdef\ar@form{%
  \ifx \space@\next \expandafter\DN@\space{\xyFN@\ar@form}%
  \else\ifx ^\next \DN@ ^{\xyFN@\ar@style}\edef\arvariant@@{\string^}%
  \else\ifx _\next \DN@ _{\xyFN@\ar@style}\edef\arvariant@@{\string_}%
  \else\ifx 0\next \DN@ 0{\xyFN@\ar@style}\def\arvariant@@{0}%
  \else\ifx 1\next \DN@ 1{\xyFN@\ar@style}\def\arvariant@@{1}%
  \else\ifx 2\next \DN@ 2{\xyFN@\ar@style}\def\arvariant@@{2}%
  \else\ifx 3\next \DN@ 3{\xyFN@\ar@style}\def\arvariant@@{3}%
  \else\ifx 4\next \DN@ 4{\xyFN@\ar@style}\def\arvariant@@{4}%
  \else\ifx \bgroup\next \let\next@=\ar@style
  \else\ifx [\next \DN@[##1]{\ar@modifiers{[##1]}}
  \else\ifx *\next \DN@ *{\ar@modifiers}%
  \else\addLT@\ifx\next \let\next@=\ar@slide
  \else\ifx /\next \let\next@=\ar@curveslash
  \else\ifx (\next \let\next@=\ar@curveinout 
  \else\addRQ@\ifx\next \addRQ@\DN@{\ar@curve@}%
  \else\addLQ@\ifx\next \addLQ@\DN@{\xyFN@\ar@curve}%
  \else\addDASH@\ifx\next \addDASH@\DN@{\defarstem@-\xyFN@\ar@}%
  \else\addEQ@\ifx\next \addEQ@\DN@{\def\arvariant@@{2}\defarstem@-\xyFN@\ar@}%
  \else\addDOT@\ifx\next \addDOT@\DN@{\defarstem@.\xyFN@\ar@}%
  \else\ifx :\next \DN@:{\def\arvariant@@{2}\defarstem@.\xyFN@\ar@}%
  \else\ifx ~\next \DN@~{\defarstem@~\xyFN@\ar@}%
  \else\ifx !\next \DN@!{\dasharstem@\xyFN@\ar@}%
  \else\ifx ?\next \DN@?{\ar@upsidedown\xyFN@\ar@}%
  \else \let\next@=\ar@error
  \fi\fi\fi\fi\fi\fi\fi\fi\fi\fi\fi\fi
  \fi\fi\fi\fi\fi\fi\fi\fi\fi\fi\fi \next@
}

\makeatother

\newcommand{\dfl}{\Rightarrow}
\newcommand{\dfll}{\Longrightarrow}
\newcommand{\odfl}[1]{\overset{\displaystyle #1}{\dfll}}
\newcommand{\tfl}{\Rrightarrow}
\newcommand{\qfl}{\xymatrix@1@C=10pt{\ar@4[r]&}}

\newcommand{\tck}[1]{#1^{\top}}

\newcommand{\supp}{\operatorname{supp}}
\newcommand{\Styl}{\mathrm{Styl}}

\newcommand{\lex}{\mathrm{lex}}
\newcommand{\slex}{\mathrm{slex}}

\renewcommand{\phi}{\varphi}
\renewcommand{\epsilon}{\varepsilon}
\renewcommand{\leq}{\leqslant}
\renewcommand{\geq}{\geqslant}

\newcommand{\Rr}{\mathcal{R}}

\providecommand{\R}{}
\renewcommand{\R}{\mathbf{R}}

\newcommand{\catego}[1]{\mathbf{\mathsf{#1}}}

\newcommand{\Cat}{\catego{Cat}}
\newcommand{\Act}{\catego{Act}}

\newcommand{\insr}[1]{%
  \;\raisebox{0.1em}{\rotatebox[origin=c]{180}{$\rightsquigarrow$}}_{#1}\;}

\makeatletter
\newcommand{\blfootnote}{\xdef\@thefnmark{}\@footnotetext}
\makeatother

\begin{document}
\thispagestyle{empty}

\begin{center}

\begin{doublespace}
\begin{huge}
{\scshape Coherent presentations for stylic monoids \\ via row rewriting}
\end{huge}

\bigskip
\hrule height 1.5pt
\bigskip

\begin{Large}
{\scshape Nohra Hage}
\end{Large}

\vspace{0.5cm}
\end{doublespace}

\vspace{1cm}

\begin{small}
\begin{minipage}{14cm}
\noindent\textbf{Abstract -- }
We study stylic monoids from the viewpoint of higher-dimensional rewriting. A coherent presentation records not only the defining relations of a monoid, but also the relations among them. Using the $N$-tableau model, we construct a finite coherent presentation for every stylic monoid of finite rank. We start from the known finite convergent row presentation. By adjoining one generating confluence for each critical branching,
Squier's coherence theorem yields a finite coherent extension. We then simplify the resulted coherent presentation by two successive homotopical reductions adapted to the row combinatorics of $N$-tableaux. Finally, coherent Tietze transformations and a further homotopical reduction transfer the construction from the auxiliary row generators to the standard letter generators. This gives an effectively computable finite homotopy basis for the standard letter presentation. The construction is illustrated in detail in rank~$2$.

\medskip


\medskip

\smallskip
\noindent\textbf{M.S.C. 2020 -- Primary:}
18N30, 68Q42.
\textbf{Secondary:}
20M05, 20M10, 05E16.
\end{minipage}
\end{small}

\begin{small}\begin{minipage}{12cm}
\renewcommand{\contentsname}{}
\setcounter{tocdepth}{1}
\tableofcontents
\end{minipage}
\end{small}
\end{center}

\tikzset{every tree node/.style={minimum width=1em,draw,circle},
         blank/.style={draw=none},
         edge from parent/.style=
         {draw,edge from parent path={(\tikzparentnode) -- (\tikzchildnode)}},
         level distance=0.8cm}
\vspace{0.5cm}

\section{Introduction}
The \emph{plactic monoid} is a classical object in algebraic combinatorics. 
Introduced through Knuth's relations on words, it is closely related to the Robinson--Schensted correspondence and to semistandard Young tableaux~\cite{Knuth70,LascouxSchutzenberger81,Fulton97}. It is defined as the quotient of the free monoid on a totally ordered alphabet by the \emph{Knuth relations}, and its elements are naturally represented by tableaux.
The plactic monoid also admits an interpretation through Kashiwara's
theory of crystal bases, which leads to plactic monoids associated with
other Cartan types~\cite{Kashiwara91,Lecouvey02,Lecouvey03}.
It also belongs to a broader family of
\emph{plactic-like monoids}, including the hypoplactic, sylvester,
Baxter, and Chinese monoids~\cite{CainGrayMalheiro15b,CainMalheiro19}. 
The interplay between words, tableaux, crystal
structures, and rewriting has made the plactic monoid and its variants
fundamental examples in algebraic combinatorics and rewriting theory
\cite{CainGrayMalheiro15, CainGrayMalheiro19, Hage25, Hage26, Hage2021Super,Hage15,
HageMalbos22,HageMalbos17,Hage2022RSK}.
For ranks greater than~$3$, Kubat and Okniński showed that the plactic
algebra of type~A does not admit a finite Gröbner--Shirshov basis with respect to the
natural degree-lexicographic order on the usual generators~\cite{KubatOkninski14}. By replacing the usual generators with column
generators, Cain, Gray, and Malheiro constructed a finite convergent
presentation whose rewriting rules encode Schensted's insertion algorithm. They used this presentation to establish biautomatic structures for plactic monoids~\cite{CainGrayMalheiro15}. Since it is adapted to the tableau structure, the column presentation also provides a natural starting point for higher-dimensional constructions. Starting from it, the author and Malbos constructed coherent presentations for plactic monoids of type~A and reduced them by homotopical methods~\cite{HageMalbos17}.

More recently, Abram and Reutenauer introduced the
\emph{stylic monoid} through the natural action of the plactic monoid on
the finite set of columns~\cite{AbramReutenauerStylic}. 
It admits the presentation obtained from the Knuth presentation by adjoining
the \emph{idempotency relations}, and its elements are represented by
$N$-tableaux. The associated right
$N$-insertion algorithm computes these representatives, analogously to
Schensted insertion in the plactic setting. The stylic monoid has since
been studied from several perspectives. Volkov investigated its
identities~\cite{Volkov2022identities}. Aird and Ribeiro studied
its tropical representations and identities~\cite{AirdRibeiro2023tropical}. Abram, Reutenauer, and Saliola studied
quivers of stylic algebras~\cite{AbramReutenauerSaliola2023Quivers}.  Luo, Jin, and Zhang
constructed finite complete rewriting systems and biautomatic
structures for stylic monoids of finite rank~\cite{LuoJinZhangStylic}. Together, these works connect stylic monoids
with semigroup identity theory, representation theory, and rewriting
theory.

The purpose of this paper is to study stylic monoids from the point of view of
higher-dimensional rewriting. More precisely, we construct coherent
presentations for stylic monoids of finite rank. Such presentations encode not
only the defining relations of the monoid, but also the relations among these
relations. They therefore provide a natural framework for studying the
higher-dimensional structure of presentations.
Our approach is based on Squier's theory and on the methods of homotopical
completion and reduction developed in higher-dimensional rewriting~\cite{GaussentGuiraudMalbos15, PolygraphsBook}. By Squier's  coherence theorem, every finite convergent
presentation gives rise to a coherent presentation by adjoining generating confluences associated with its critical branchings. Homotopical reduction
then allows one to reduce the resulting coherent presentation by eliminating
collapsible cells, while preserving a homotopy basis.

These methods have been used to construct coherent presentations for several
families of monoids, including Artin monoids~\cite{GaussentGuiraudMalbos15},
monoids admitting a right-noetherian Garside family
~\cite{CurienDuricGuiraud23}, plactic monoids of several types~\cite{HageMalbos17, Meha2021}, Chinese monoids~\cite{HageMalbos22}, and tableaux
monoids~\cite{HageTableauxMonoid}.
 In the present paper, we follow the same general strategy as in the type~A plactic case~\cite{HageMalbos17}: we
start from a finite convergent presentation adapted to tableaux, construct the
coherent presentation given by Squier's coherence theorem, and reduce it by homotopical
methods. Beyond coherence itself, these constructions provide the first
dimensions of a \emph{polygraphic resolution}, that is, a free
higher-dimensional resolution of the monoid by an acyclic
$(\omega,1)$-polygraph~\cite{GuiraudMalbos12advances,GuiraudMalbos10smf}.
Coherent presentations also give a finite description of \emph{actions of
monoids on categories}~\cite{GaussentGuiraudMalbos15}.

Our concrete starting point is the finite convergent row rewriting system for the stylic monoid constructed by Luo, Jin, and Zhang~\cite{LuoJinZhangStylic}. We regard this rewriting system as a $2$-polygraph. Its generating $1$-cells~$r_u$ are indexed by the nonempty rows~$u$ over the ordered alphabet~$[n]:=\{1,\ldots,n\}$, and its generating $2$-cells are the row rewriting rules determined by right $N$-insertion. The irreducible words in the row generators are precisely the row-generator readings of $N$-tableaux.
Since this $2$-polygraph is finite and convergent, Squier's coherence theorem applies. In Theorem~\ref{T:CoherentPresentation}, we construct the corresponding coherent presentation explicitly by adjoining one generating confluence for each critical branching. These critical branchings are indexed by triples of rows~$(u,v,t)$ for which both adjacent pairs~$(u,v)$ and~$(v,t)$ are reducible. We prove in Proposition~\ref{Prop:boundedBranching} that each branch reaches the common normal form in at most three rewriting steps. In Subsubsection~\ref{SSS:ExplicitGeneratingConfluences}, we give the finite list of possible confluence diagrams. The appendix verifies that all their rewriting steps are well defined and that both branches reach the normal form obtained by right $N$-insertion.

We then reduce the coherent row presentation by homotopical reduction in
Section~\ref{S:ReducedCoherentPresentationStylicMonoid}. The first reduction removes the \emph{splittable generating confluences}.
For the second reduction, suitable Nielsen transformations isolate
selected row relations as targets of collapsible $3$-cells. An explicit
acyclicity procedure then determines a family that can be eliminated
simultaneously. The explicit case-by-case analysis of the
confluence diagrams identifies which row relations can be isolated and which
dependencies may occur.
Finally, we transport the reduced coherent row presentation, through an
intermediate pre-row presentation, to the standard presentation on the
letter generators. We then eliminate all non-letter row generators.
This proves our main result, Theorem~\ref{T:LetterCoherentPresentation}:
for every $n\geq1$, the stylic monoid~$\Styl_n$ admits an effectively
computable finite coherent presentation whose underlying $2$-polygraph is
the standard presentation on~$[n]$ by the oriented Knuth and idempotency
relations.

The paper is organized as follows. Section~\ref{S:RowPresentationStylicMonoid} recalls the background on
$2$-polygraphs, the stylic monoid, $N$-tableaux, and the finite convergent row
presentation. 
Section~\ref{S:CoherentPresentationStylicMonoid} constructs the coherent row
presentation. We describe the critical branchings, give the corresponding
generating confluences, and count the generating cells.
Section~\ref{S:ReducedCoherentPresentationStylicMonoid} carries out the two
homotopical reductions and then transports the reduced coherent row presentation
first to the pre-row presentation and finally to the letter presentation.
Section~\ref{S:Example} gives explicit computations in rank $2$, where the whole
reduction process can be described in detail. 
Appendix~\ref{A:ExplicitGeneratingConfluences} contains the
case-by-case verification of the generating confluence diagrams used in the
construction of the coherent row presentation.

\subsection*{Notation}

Throughout the paper, fix an integer $n\geq 1$, and let $[n]=\{1<2<\cdots<n\}$ be a finite totally ordered alphabet. We denote by $[n]^\ast$ the free monoid generated by $[n]$, with empty word $\varepsilon$ and concatenation as product. For a word $u=x_1x_2\cdots x_q$ in $[n]^\ast$, we denote by $\ell(u)=q$ its \emph{length} and by~$\supp(u)=\{x_1,\dots,x_q\}\subseteq [n]$ its \emph{support}. We set~$\supp(\varepsilon)=\varnothing$. If $u$ is nonempty, we write~$\min(u):=\min(\supp(u))$ and~$\max(u):=\max(\supp(u))$.

The \emph{lexicographic order} on~$[n]^\ast$, denoted by
$<_{\lex}$, is induced by the order on~$[n]$. Thus, for
$u=x_1\cdots x_p$ and $v=y_1\cdots y_q$, one has
$u<_{\lex}v$ if either $u$ is a proper prefix of~$v$, or there exists
$i\leq\min(p,q)$ such that~$
x_1=y_1,\ldots,x_{i-1}=y_{i-1}$, and~$x_i<y_i$.
The \emph{shortlex order}, denoted by~$<_{\slex}$, is defined by~$u<_{\slex}v$ if~$\ell(u)<\ell(v)$ or~$\ell(u)=\ell(v)$ and~$u<_{\lex}v$.

\section{A finite convergent row presentation of the stylic monoid}
\label{S:RowPresentationStylicMonoid}

In this section, we recall the finite convergent row presentation of the
stylic monoid. We first fix the terminology concerning $2$-polygraphs,
rewriting systems, and Tietze transformations. We then recall the realization
of the stylic monoid by $N$-tableaux and right $N$-insertion. Finally, we
introduce the pre-row and row presentations and prove explicitly that they are
Tietze equivalent.

\subsection{Two-dimensional polygraphs}
\label{SS:TwoPolygraphs}
We recall the language of $2$-polygraphs and rewriting systems needed in the
sequel. General references on higher-dimensional rewriting include
~\cite{GuiraudMalbos18,GaussentGuiraudMalbos15,
GuiraudMalbos12advances,PolygraphsBook}.

In this article, monoid presentations are encoded by rewriting systems arising from $2$-polygraphs with a unique $0$-cell, denoted by~$\bullet$.
A \emph{$2$-polygraph}~$\Sigma$ consists of a pair~$
(\Sigma_1,\Sigma_2)$,
where~$\Sigma_1$ is a set of generating $1$-cells and $\Sigma_2$ is a \emph{globular extension} of the free monoid~$\Sigma_1^\ast$, that is, a set of $2$-cells~$\beta\colon u \dfl v$
between $1$-cells $u,v\in \Sigma_1^\ast$.
The $1$-cells $u$ and $v$ are respectively called the \emph{source} and the \emph{target} of~$\beta$, and are denoted by $s_1(\beta)$ and $t_1(\beta)$.
Recall that a $2$-category is a category enriched in categories, whereas a $(2,1)$-category is a category enriched in groupoids.
If two $1$-cells or two $2$-cells $f$ and $g$ are $0$-composable, we denote their $0$-composite by~$fg$.
If they are $1$-composable, we denote their $1$-composite by~$f\star_1 g$.
We denote by $\Sigma_2^\ast$ the $2$-category freely generated by the
$2$-polygraph $\Sigma$, and by $\Sigma_2^\top$ the $(2,1)$-category freely
generated by $\Sigma$, obtained by formally adjoining inverses to the
generating $2$-cells. We call $\Sigma_2^\top$ the \emph{track $2$-category}
generated by $\Sigma$.
The \emph{monoid presented by~$\Sigma$}, denoted by~$\overline{\Sigma}$, is the quotient of the free monoid~$\Sigma_1^\ast$ by the congruence generated by the $2$-cells of~$\Sigma_2$.
A \emph{presentation} of a monoid~$M$ is a $2$-polygraph whose presented monoid is isomorphic to~$M$.

We now recall the elementary transformations that will be used to compare the presentations appearing later.
Let~$\Sigma$ be a $2$-polygraph.
A generating $2$-cell~$\beta$ in~$\Sigma_2$ is said to be
\emph{collapsible} if its target~$t_1(\beta)$ is a $1$-cell of~$\Sigma_1$ and the source~$s_1(\beta)$ does not contain~$t_1(\beta)$.
In that case, the $1$-cell~$t_1(\beta)$ is called \emph{redundant}.
A generating $2$-cell~$\beta\colon u\dfl v$ of~$\Sigma$ is called \emph{redundant} if there
exists a parallel $2$-cell~$\widehat{\beta}\colon u\dfl v$ in the track $2$-category generated by the $2$-polygraph obtained from
$\Sigma$ by deleting~$\beta$. Equivalently, the source and target of~$\beta$
are already connected by a $2$-cell containing neither~$\beta$ nor its
inverse.
An \emph{elementary Tietze transformation} is the $2$-functor induced by one
of the following operations:
\begin{enumerate}[\bf i)]
\item adjunction or elimination of a redundant generating $1$-cell~$x$
together with its collapsible generating $2$-cell~$\beta$;
\item adjunction or elimination of a redundant generating
$2$-cell~$\beta$.
\end{enumerate}
If $\Sigma$ and $\Upsilon$ are $2$-polygraphs, a \emph{Tietze transformation} from~$\Sigma$ to~$\Upsilon$ is a $2$-functor
\[
F\colon \Sigma_2^\top \longrightarrow \Upsilon_2^\top
\]
that decomposes as a finite sequence of elementary Tietze transformations.
Two $2$-polygraphs are said to be \emph{Tietze equivalent} if they are related
by a finite sequence of elementary Tietze transformations. 
Tietze equivalent $2$-polygraphs present isomorphic monoids.
A useful particular case is the following Nielsen transformation.
Let~$\gamma\in\Sigma_2$ be a generating $2$-cell, and let
$\gamma_1$ and~$\gamma_2$ be $2$-cells of~$\Sigma_2^\top$ that
contain neither~$\gamma$ nor its inverse. Suppose that the composite~$
\delta
=
\gamma_1\star_1\gamma\star_1\gamma_2$
is defined. One may adjoin a generating $2$-cell~$\beta\colon s_1(\delta)\dfl t_1(\delta)$
parallel to~$\delta$. The new cell~$\beta$ is redundant, and its
elimination functor sends~$\beta$ to~$\delta$.
Conversely, after adjoining~$\beta$, the generating $2$-cell~$\gamma$
is redundant, since it is parallel to~$\gamma_1^{-1}\star_1\beta\star_1\gamma_2^{-1}$,
which contains neither~$\gamma$ nor its inverse.  Thus one may first
adjoin~$\beta$ and then eliminate~$\gamma$. This two-step Tietze
transformation is called a \emph{Nielsen transformation}.

We finally recall the standard notions of rewriting, branching, and convergence for $2$-polygraphs.
A \emph{rewriting step} of a $2$-polygraph~$\Sigma$ is a $2$-cell of~$\Sigma_2^\ast$ of the form~$w\,\beta\,w'$,
where $\beta$ is a $2$-cell of~$\Sigma_2$ and~$w,w'$ are $1$-cells of~$\Sigma_1^\ast$.
A \emph{rewriting sequence} of~$\Sigma$ is a finite or infinite
composable sequence of rewriting steps.
If $u,v\in\Sigma_1^\ast$, we write~$u\dfl^\ast v$ when there exists a $2$-cell~$f:u\dfl v$ in~$\Sigma_2^\ast$, that is, a finite
composite of rewriting steps from~$u$ to~$v$. We allow the empty composite, so~$u\dfl^\ast u$.
A $1$-cell~$u$ of~$\Sigma_1^\ast$ is a \emph{normal form} if there is no rewriting step with source~$u$.
The $2$-polygraph~$\Sigma$ is said to \emph{terminate} if it admits no infinite rewriting sequence.
A \emph{branching} of~$\Sigma$ is a pair $(f,g)$ of $2$-cells of~$\Sigma_2^\ast$ such that~$s_1(f)=s_1(g)$.
A branching is said to be \emph{local} if both $f$ and $g$ are rewriting steps.
A local branching is \emph{aspherical} if it is of the form~$(f,f)$
for some rewriting step~$f$. It is \emph{Peiffer} if it is of the
form~$(fv,ug)$, where~$f$ and~$g$ are rewriting steps satisfying
$s_1(f)=u$ and~$s_1(g)=v$. All other local branchings are called
\emph{overlapping}.
An overlapping local branching is \emph{critical} if it is minimal with respect to the order generated by the relations
\[
(f,g)\leq (wfw',wgw')
\]
for all local branchings~$(f,g)$ and all $1$-cells~$w,w'$ of~$\Sigma_1^\ast$.
A branching~$(f,g)$ is \emph{confluent} if there exist $2$-cells $f'$ and $g'$ in~$\Sigma_2^\ast$ such that~$s_1(f')=t_1(f)$, $s_1(g')=t_1(g)$, and~$t_1(f')=t_1(g')$.
The $2$-polygraph~$\Sigma$ is \emph{confluent} if all of its branchings are confluent.
Finally,~$\Sigma$ is \emph{convergent} if it is terminating and confluent.
In that case, every $1$-cell~$u$ of~$\Sigma_1^\ast$ admits a unique normal form, denoted by~$\widehat u$.
 
For every finite convergent $2$-polygraph~$\Sigma$ considered below, we fix
a total order on~$\Sigma_2$. At each reducible $1$-cell, we apply a rewriting
step whose occurrence begins at the leftmost possible position. If several
generating $2$-cells are applicable at that position, we choose the least one
with respect to the fixed total order. Since~$\Sigma$ is convergent, this
procedure terminates at the unique normal form~$\widehat u$. 
The resulting normalization path is the $2$-cell
\[
\npath{u}\colon u\dfl\widehat u
\]
in~$\Sigma_2^\ast$. It is called the \emph{leftmost normalization path} of~$u$. If~$u$ is
already in normal form, we set~$\npath{u}:=1_u$.

\subsection{The stylic monoid, $N$-tableaux, and right insertion}
\label{SS:StylicMonoid}

The \emph{stylic monoid of rank~$n$}, denoted by~$\Styl_n$, is presented by
the $2$-polygraph~$\Sigma^{\mathrm{styl}}(n)$, whose generating $1$-cells
are the letters of~$[n]$ and whose generating $2$-cells, oriented in the
decreasing direction with respect to the shortlex order fixed above, are~\cite{AbramReutenauerStylic}:
\[
\begin{array}{rl}
\eta_x\colon& xx\dfl x
\qquad
\text{for }x\in[n],\\
\kappa_{x,y,z}\colon& zxy\dfl xzy
\qquad
\text{for }x\leq y<z,\\
\lambda_{x,y,z}\colon& yzx\dfl yxz
\qquad
\text{for }x<y\leq z.
\end{array}
\]

We now recall from~\cite{AbramReutenauerStylic} the \emph{tableau model} for the
stylic monoid, described in terms of $N$-tableaux and right $N$-insertion.
A \emph{row} is a nonempty word~$u=x_1x_2\cdots x_k$
in~$[n]^\ast$ such that~$x_1<x_2<\cdots<x_k$.
Since a row is strictly increasing, it can be identified with its support.
We denote by
\[
\Rr_n
:=
\left\{
u\in[n]^\ast
\ \middle|\
u\text{ is a nonempty row}
\right\}
\]
the set of nonempty rows on~$[n]$.
The set~$\Rr_n$ is finite, and~$|\Rr_n|=2^n-1$.
If $u$ and $v$ are rows, we write~$u\triangleleft v$
whenever~$\supp(u)\subseteq \supp(v)$ and~$\min(v)<\min(u)$.
An \emph{$N$-tableau} is a collection of left-justified boxes filled with letters of~$[n]$
such that each row (resp. column) is strictly increasing from left to right (resp. from bottom to top), and each row, viewed as a subset of~$[n]$, is contained in the row below it.
Thus, if the rows are denoted from top to bottom by~$u_k,u_{k-1},\dots,u_1$,
then~$T$ is an $N$-tableau if and only if~$u_{i+1}\triangleleft u_i$, for all~$1\le i\le k-1$.
In that case, we denote by~$\R(T)=u_k u_{k-1}\cdots u_1$
the \emph{row reading} of~$T$, obtained by reading the rows from top to bottom.
For instance, the $N$-tableau

\begin{equation}
\label{eq:Ntableau}
T=
\begin{ytableau}
4 & 5 & \none & \none & \none \\
2 & 4 & 5 & \none & \none \\
1 & 2 & 3 & 4 & 5
\end{ytableau}
\end{equation}
has row reading~$\R(T)=45\,245\,12345$.

We now recall the right $N$-insertion procedure.
We begin with the insertion of a letter into a row. Let~$u$ be a row and let
$x\in[n]$. Since the row~$u$ is identified with the subset
$\supp(u)\subseteq[n]$, inserting~$x$ into~$u$ replaces $u$ by the strictly
increasing row corresponding to $\supp(u)\cup\{x\}$. If there exists a smallest
letter of~$u$ which is strictly greater than~$x$, we denote it by~$y$ and call it
the \emph{bumped letter}. If no letter of~$u$ is strictly greater than~$x$, then
no letter is bumped. Equivalently, no letter is bumped if and only if~$x$ is
greater than or equal to every letter of~$u$.
Notice that the bumped letter, when it exists, does not disappear from the row.
Moreover, if $x\in\supp(u)$, then the row itself is unchanged, although a letter
may still be bumped.

We now extend this construction to arbitrary $N$-tableaux. Let~$T$ be an
$N$-tableau whose rows, from top to bottom, are~$u_k,u_{k-1},\ldots,u_1$, and
let $x\in[n]$. The \emph{right $N$-insertion} of~$x$ into~$T$,  is defined recursively as follows. First, insert~$x$ into the
bottom row~$u_1$. If no letter is bumped, the resulting tableau is obtained by
replacing~$u_1$ with its new row. If a letter~$y$ is bumped from~$u_1$, then one
inserts~$y$ into the next row~$u_2$. More generally, whenever a letter is bumped
from a row, it is inserted into the row immediately above. If a letter is bumped
from the top row, it creates a new top row consisting of that single letter.
By~\cite[Proposition~6.1]{AbramReutenauerStylic}, repeating this procedure
row by row produces an $N$-tableau, denoted by~$T\insr{}x$.

For a word~$u=x_1x_2\cdots x_p$ in~$[n]^\ast$, we define~$N(u)$ by inserting
the letters of~$u$ successively from left to right, starting from the empty
tableau~$N(\varepsilon)$. Inserting a letter into the empty tableau creates the
one-row tableau consisting of that letter. Thus
\[
N(u)
=
((\cdots((N(\varepsilon)\insr{}x_1)\insr{}x_2)\cdots)\insr{}x_p).
\]
By~\cite[Theorem~7.1]{AbramReutenauerStylic}, the map
$u\longmapsto N(u)$ induces a bijection between stylic classes and
$N$-tableaux. In particular, a word~$u$ and the row reading~$\R(N(u))$
belong to the same stylic class. Hence, for every $N$-tableau~$T$ and every
letter~$x\in[n]$, one has~$T\insr{}x=N(\R(T)x)$.

\begin{example}
Consider the $N$-tableau~\eqref{eq:Ntableau} and insert the letter~$2$.
In the bottom row $12345$, the letter $2$ is already present, so the row is
unchanged, but $3$ is bumped. Inserting $3$ into the row $245$ gives $2345$ and
bumps $4$. Inserting $4$ into the row $45$ leaves the row unchanged and bumps
$5$, which creates a new top row. Hence
\[
T\insr{}2=
\begin{ytableau}
5 & \none & \none & \none & \none \\
4 & 5 & \none & \none & \none \\
2 & 3 & 4 & 5 & \none \\
1 & 2 & 3 & 4 & 5
\end{ytableau}.
\]
\end{example}

We record several elementary properties of right $N$-insertion that will be
used in the subsequent constructions and proofs.

\begin{lemma}
\label{L:BottomRowUnion}
Let $u$ and $v$ be rows. Then the bottom row of $N(uv)$ is the strictly
increasing row corresponding to $\supp(u)\cup\supp(v)$.
\end{lemma}

\begin{proof}
By definition, $N(uv)$ is obtained by inserting the letters of $v$ from left to
right into the row $u$. At each step, the current bottom row is replaced by the
strictly increasing row whose support is the union of its previous support with
the inserted letter. If a letter is bumped, it is passed to the row above, but it
does not disappear from the bottom row. Hence, after all letters of $v$ have
been inserted, the bottom row is exactly the row corresponding to
$\supp(u)\cup\supp(v)$.
\end{proof}

\begin{lemma}
\label{L:TwoRowsProductRows}
Let $u$ and $v$ be rows. Then the tableau $N(uv)$ has at most two rows.
\end{lemma}

\begin{proof}
Write $v=y_1y_2\cdots y_q$, with $y_1<\cdots<y_q$. The tableau $N(uv)$ is
obtained by inserting the letters $y_1,\dots,y_q$ successively into the bottom
row, starting from $u$. Let $R_k$ be the bottom row after inserting~$y_k$, with
$R_0=u$.
If no bump occurs, then $N(uv)$ has one row. Suppose that bumps do occur, and
let $S\subseteq\{1,\dots,q\}$ be the set of indices at which a bump occurs. For
$k\in S$, let
\[
c_k:=\min\{a\in\supp(R_{k-1})\mid a>y_k\}
\]
be the bumped letter.

We claim that the sequence $(c_k)_{k\in S}$ is weakly increasing. Let
$k<k'$ be two consecutive elements of~$S$. Since $v$ is a row, one has
$y_k<y_{k'}$. Between the insertions of $y_k$ and $y_{k'}$, the bottom row does
not lose any letter. It only gains some of the letters
$y_k,y_{k+1},\dots,y_{k'-1}$, all of which are strictly smaller than~$y_{k'}$.
Thus no new letter strictly greater than $y_{k'}$ appears in the bottom row
between these two bumping steps. Hence
\[
\{a\in\supp(R_{k'-1})\mid a>y_{k'}\}
\subseteq
\{a\in\supp(R_{k-1})\mid a>y_k\}.
\]
Taking minima gives $c_k\le c_{k'}$. This proves the claim.

The second row is obtained by inserting the bumped letters $c_k$, for
$k\in S$, in increasing order of the indices. Since this sequence is weakly
increasing, each inserted letter is greater than or equal to every letter
already present in the second row. Hence no bump can occur from the second row.
Therefore no third row is created, and $N(uv)$ has at most two rows.
\end{proof}

\begin{lemma}
\label{L:OneRowCriterion}
Let $u$ and $v$ be rows. Then the following are equivalent:
\begin{enumerate}[{\bf i)}]
\item the tableau $N(uv)$ has one row;
\item no bump occurs during the insertion of the letters of $v$ into $u$;
\item $\min(v)\geq \max(u)$.
\end{enumerate}
When these conditions hold, one has $\R(N(uv))=u\cup v$, where $u\cup v$
denotes the strictly increasing row corresponding to $\supp(u)\cup\supp(v)$.
\end{lemma}

\begin{proof}
We first prove that {\bf ii)} and {\bf iii)} are equivalent. Write
$v=y_1\cdots y_q$, with $y_1<\cdots<y_q$.
Assume that $\min(v)\ge\max(u)$. Let $R_k$ be the bottom row after inserting
$y_1,\dots,y_k$, with $R_0=u$. We prove by induction on $k$ that no bump occurs
at step $k$. For $k=1$, one has
\[
y_1=\min(v)\ge\max(u)=\max(R_0).
\]
Suppose the claim holds up to step $k-1$. Then $R_{k-1}$ has support
$\supp(u)\cup\{y_1,\dots,y_{k-1}\}$. Since
$y_k>y_{k-1}>\cdots>y_1$ and $y_k\ge\min(v)\ge\max(u)$, one has
$y_k\ge\max(R_{k-1})$. Hence no bump occurs at step $k$. Thus {\bf iii)}
implies {\bf ii)}.
Conversely, if $\min(v)<\max(u)$, then $y_1=\min(v)<\max(u)$. Hence some
letter of $u$ is strictly greater than $y_1$, so a bump occurs at the first
step. Thus {\bf ii)} implies {\bf iii)}.
The equivalence between {\bf i)} and {\bf ii)} follows directly from the
definition of right $N$-insertion.
Finally, when these conditions hold, the tableau $N(uv)$ has one row. By
Lemma~\ref{L:BottomRowUnion}, its support is $\supp(u)\cup\supp(v)$. Hence
$\R(N(uv))=u\cup v$.
\end{proof}

\begin{lemma}
\label{L:TriangleleftMonotonicity}
Let $u$, $v$ and $t$ be rows. Assume that $u\triangleleft v$,
$\supp(v)\subseteq\supp(t)$, and $\min(t)\le\min(v)$. Then~$u\triangleleft t$.
\end{lemma}

\begin{proof}
Since $u\triangleleft v$, one has $\supp(u)\subseteq\supp(v)$ and
$\min(v)<\min(u)$. The inclusion $\supp(v)\subseteq\supp(t)$ gives
$\supp(u)\subseteq\supp(t)$. Moreover,
$\min(t)\le\min(v)<\min(u)$. Hence $u\triangleleft t$.
\end{proof}

\begin{lemma}
\label{L:TopRowInclusion}
Let $u$ and $v$ be rows, and suppose that $N(uv)$ has two rows. Write
$\R(N(uv))=cd$, where $c$ and $d$ are respectively the top and bottom rows.
Then $\supp(c)\subseteq\supp(u)$.
\end{lemma}

\begin{proof}
By Lemma~\ref{L:BottomRowUnion}, the bottom row $d$ of $N(uv)$ satisfies~$\supp(d)=\supp(u)\cup\supp(v)$.
The top row $c$ is obtained by inserting, into the row above, the letters bumped
from the bottom row during the successive insertions of the letters of $v$ into
the initial row $u$. It therefore suffices  to prove that no letter coming from
$v$ can be bumped from the bottom row.

Write $v=y_1\cdots y_q$, with $y_1<\cdots<y_q$. During the computation of
$N(uv)$, the letters of $v$ are inserted into the bottom row from left to right,
hence in increasing order.
Suppose that a letter~$y_i$ has already been inserted into the bottom row. Any
later inserted letter is some $y_j$ with $j>i$, hence $y_j>y_i$. But a bump
selects the smallest letter of the current row that is strictly greater than the
inserted letter. Therefore inserting~$y_j$ cannot bump the smaller letter~$y_i$.
Thus no letter inserted from $v$ is ever bumped from the bottom row. Hence every
letter bumped from the bottom row belongs to the initial support~$\supp(u)$.
Since the top row $c$ is formed from these bumped letters, one obtains~$\supp(c)\subseteq\supp(u)$.
\end{proof}

\begin{remark}
\label{R:FiniteCompletionLetterPresentation}
For each fixed rank $n$, the letter presentation
$\Sigma^{\mathrm{styl}}(n)$ admits a finite convergent completion.
Let $\mathcal N_n$ be the finite set of shortlex-minimal representatives of
the stylic classes, including the empty word~$\varepsilon$, which represents
the identity.  Write~$\nu(u)$ for the representative of the class of a
word~$u$. For every $w\in\mathcal N_n$ and every $x\in[n]$ such that
$wx\neq\nu(wx)$, adjoin a transition $2$-cell
\[
\theta_{w,x}\colon wx\dfl\nu(wx).
\]
We denote by~$\widehat{\Sigma}^{\mathrm{styl}}(n)$ the resulting
$2$-polygraph.

Since~$\Styl_n$ is finite, the set~$\mathcal N_n$ is finite, and hence
$\widehat{\Sigma}^{\mathrm{styl}}(n)$ is finite. Every generating rule
strictly decreases the shortlex order. Indeed, the idempotent rules decrease
the length, the two families of Knuth-type rules decrease the lexicographic
order at fixed length, and every transition rule decreases shortlex by the
minimality of~$\nu(wx)$. Since shortlex is a well-founded order compatible
with contexts, the $2$-polygraph
$\widehat{\Sigma}^{\mathrm{styl}}(n)$ terminates.

Moreover, every word reduces to its chosen representative. Indeed, if
$u=x_1\cdots x_k$, set $w_0=\varepsilon$ and
$w_i=\nu(x_1\cdots x_i)$,  for~$1\le i\le k$.
Then $w_{i-1}x_i$ and $x_1\cdots x_i$ represent the same element of
$\Styl_n$, so~$\nu(w_{i-1}x_i)=w_i$.
Thus either $w_{i-1}x_i=w_i$, or the transition rule~$\theta_{w_{i-1},x_i}\colon w_{i-1}x_i\dfl w_i$
applies. Applying these reductions successively gives
\[
x_1\cdots x_k
\dfl^\ast
w_1x_2\cdots x_k
\dfl^\ast
w_2x_3\cdots x_k
\dfl^\ast
\cdots
\dfl^\ast
w_k
=
\nu(u).
\]
Every chosen representative is irreducible. Indeed, a rewriting step starting from
such a representative would produce a strictly smaller word in the same
stylic class, contradicting shortlex minimality. Conversely, if~$q$ is
irreducible, the preceding prefix-reduction construction gives
$q\dfl^\ast\nu(q)$. This path must be empty, and hence $q=\nu(q)$.
Therefore every stylic class contains exactly one irreducible word. Together
with termination, this proves confluence. Hence
$\widehat{\Sigma}^{\mathrm{styl}}(n)$ is a finite convergent completion of
the letter presentation.

The finiteness of~$\Styl_n$ in fixed rank yields the finite shortlex completion above. However, its rewriting rules do not directly reflect right $N$-insertion, which makes its relations among relations difficult to describe explicitly. We therefore use the row presentation, whose rewriting rules are adapted
to the combinatorics of $N$-tableaux and are better suited to the construction of a coherent presentation.
\end{remark}

\subsection{The pre-row presentation}
\label{SS:PreRowPresentation}
We now recall from~\cite{LuoJinZhangStylic} the pre-row presentation of
the stylic monoid. Set
\[
\overline{\Rr}_n
:=
\left\{
u\in\Rr_n
\ \middle|\
\ell(u)\geq2
\right\}.
\]
For each row $u\in \Rr_n$, we introduce a $1$-cell $r_u$.
In particular, for every letter $x\in [n]$, the symbol~$r_x$ denotes the row generator
corresponding to the row of length~$1$. If~$T$ is an $N$-tableau with~$\R(T)= u_k u_{k-1} \cdots u_1$, we
denote by~$\Rr(T):=r_{u_k}r_{u_{k-1}}\cdots r_{u_1}$ the \emph{row-generator reading} of~$T$.

For every row~$u=x_1x_2\cdots x_k\in \overline{\Rr}_n$,
we adjoin a $2$-cell
\[
\gamma_u\colon r_{x_1}r_{x_2}\cdots r_{x_k}\dfl r_u.
\]
These cells simply say that a strictly increasing word may be compressed into the single
generator corresponding to that row.

We denote by $\Sigma^{\mathrm{prerow}}(n)$ the $2$-polygraph whose set of $1$-cells is
\[
\Sigma^{\mathrm{prerow}}_1(n)=\{r_u\mid u\in \Rr_n\}
\]
and whose set of $2$-cells~$\Sigma^{\mathrm{prerow}}_2(n)$ consists of the \emph{lifted stylic relations}:
\[
\begin{array}{rl}
\eta_x^r\colon& r_xr_x\dfl r_x
\qquad
\text{for }x\in [n],\\
\kappa_{x,y,z}^r\colon& r_zr_xr_y\dfl r_xr_zr_y
\qquad
\text{for }x\le y<z,\\
\lambda_{x,y,z}^r\colon& r_yr_zr_x\dfl r_yr_xr_z
\qquad
\text{for }x<y\le z,
\end{array}
\]
together with the \emph{compression $2$-cells}~$\gamma_u$ for each $u$ in~$\overline{\Rr}_n$.
We call this the \emph{pre-row presentation} of the stylic monoid.

\begin{proposition}
\label{P:PrerowPresentsStylic} 
The $2$-polygraph $\Sigma^{\mathrm{prerow}}(n)$ presents the stylic monoid $\Styl_n$. 
\end{proposition}

\begin{proof}
The lifted stylic presentation on the generators $r_x$, with $x\in[n]$, is
identified with the letter presentation $\Sigma^{\mathrm{styl}}(n)$, and hence
presents $\Styl_n$. The pre-row presentation is obtained from it by adjoining
only redundant row generators, each together with its collapsible compression
cell. These elementary Tietze transformations do not change the presented
monoid. Therefore $\Sigma^{\mathrm{prerow}}(n)$ presents $\Styl_n$.
\end{proof}

\subsubsection{Canonical comparison paths in the pre-row presentation}
\label{SSS:CanonicalPreRowPaths}

Fix the following order on the signed lifted stylic relations:
\[
\eta^r
\prec
\kappa^r
\prec
\lambda^r
\prec
(\eta^r)^{-1}
\prec
(\kappa^r)^{-1}
\prec
(\lambda^r)^{-1}.
\]
Within each family, order the cells lexicographically by their parameters.
Elementary rewriting steps are ordered first by their occurrence position,
from left to right, and then by the preceding order.

Let $U$ and $V$ be words in the generators
$\{r_x\mid x\in[n]\}$ representing the same element of~$\Styl_n$.
Consider the \emph{signed rewriting graph} whose vertices are words in the
letter generators and whose edges are applications of lifted stylic
relations or their inverses. Order the outgoing elementary steps at each
vertex by their occurrence position and then by the fixed order on the
signed relations.
Starting from~$U$, enumerate finite rewriting paths breadth first:
paths are considered first by length and, among paths of the same
length, lexicographically by their sequences of elementary steps.
Since the signed rewriting graph is locally finite, only finitely many
paths of any fixed length start at~$U$. Moreover, since~$U$ and~$V$
represent the same element of~$\Styl_n$, there exists a finite signed
rewriting path from~$U$ to~$V$. Consequently, the breadth-first search
eventually reaches~$V$.
Denote by~$\Pi_{U,V}\colon U\dfl V$ the first path reaching~$V$. By construction, $\Pi_{U,V}$ has minimum
length among all paths from~$U$ to~$V$, and it is the
lexicographically least path among those of minimum length. In
particular,~$\Pi_{U,V}$ is well defined and effectively computable.

For a row-generator word~$w=r_{u_1}\cdots r_{u_k}$, with~$u_i=x_{i,1}\cdots x_{i,p_i}$, define
\[
\operatorname{flat}(w)
=
r_{x_{1,1}}\cdots r_{x_{1,p_1}}
\cdots
r_{x_{k,1}}\cdots r_{x_{k,p_k}}.
\]
For each $u\in\Rr_n$, let
\[
e_u\colon r_u\dfl\operatorname{flat}(r_u),
\qquad
e_u=
\begin{cases}
\gamma_u^{-1}, & \ell(u)\geq 2,\\
1_{r_u},       & \ell(u)=1.
\end{cases}
\]
Set~$E_w:=e_{u_1}\cdots e_{u_k}
\colon
w\dfl\operatorname{flat}(w)$,
where the product is the $0$-composite of the displayed $2$-cells.

For row-generator words $s$ and $t$ representing the same stylic
element, define
\begin{equation}
\label{Eq:CanonicalPreRowComparison}
C_{s,t}
:=
E_s
\star_1
\Pi_{\operatorname{flat}(s),\operatorname{flat}(t)}
\star_1
E_t^{-1}
\colon
s\dfl t.
\end{equation}
Thus $C_{s,t}$ expands $s$, follows the canonical lifted stylic path,
and then recompresses to $t$.
In particular, for $u,v\in\Rr_n$ with $u\not\triangleleft v$, set
\begin{equation}
\label{Eq:CanonicalPreRowRowCell}
P_{u,v}
:=
C_{r_ur_v,\Rr(N(uv))}.
\end{equation}
This path is well defined because $uv$ and $\R(N(uv))$ represent the
same element of $\Styl_n$. It provides the specified replacement for
the row $2$-cell $\alpha_{u,v}$ introduced below.

\subsection{The row presentation}
\label{SS:RowPresentation}

We now recall from~\cite{LuoJinZhangStylic} the row presentation of the stylic
monoid. Its set of generating $1$-cells is
\[
\Sigma^{\mathrm{row}}_1(n)
=
\{r_u\mid u\in\Rr_n\}.
\]
Let $u,v\in\Rr_n$. By Lemma~\ref{L:TwoRowsProductRows}, the tableau $N(uv)$ has
at most two rows. For every pair of rows $u,v$ such that $u\not\triangleleft v$,
we introduce a generating $2$-cell~$\alpha_{u,v}\colon r_u r_v \dfl \Rr(N(uv))$.
Thus the set of generating $2$-cells is
\[
\Sigma^{\mathrm{row}}_2(n)
=
\{\alpha_{u,v}\mid u,v\in\Rr_n,\ u\not\triangleleft v\}.
\]
The target of $\alpha_{u,v}$ is either $r_{w'}$, if $N(uv)$ has one row $w'$,
or $r_w r_{w'}$, if $N(uv)$ has top row $w$ and bottom row $w'$.
The resulting $2$-polygraph~$\Sigma^{\mathrm{row}}(n)=
\bigl(\Sigma^{\mathrm{row}}_1(n),\Sigma^{\mathrm{row}}_2(n)\bigr)$
is called the \emph{row presentation} of the stylic monoid.

The condition $u\not\triangleleft v$ means precisely that the word $r_ur_v$ is
not already in row normal form. Indeed, if $u\triangleleft v$, then $r_ur_v$ is
the row-generator reading of a two-row $N$-tableau, and no rewriting step is
needed. Hence the irreducible words are exactly the words~$
r_{u_k}\cdots r_{u_1}$ such that~$u_{i+1}\triangleleft u_i$, for all~$1\le i\le k-1$.
Equivalently, the irreducible words are precisely the row-generator readings of
$N$-tableaux.

\begin{proposition}[{\cite[Theorem~3.5]{LuoJinZhangStylic}}]
\label{P:RowPresentationConvergent}
The $2$-polygraph $\Sigma^{\mathrm{row}}(n)$ is finite and convergent. Its
irreducible $1$-cells are exactly the row-generator readings of $N$-tableaux
on~$[n]$. In particular, $\Sigma^{\mathrm{row}}(n)$ presents the stylic
monoid~$\Styl_n$.
\end{proposition}

We now prove explicitly that the pre-row presentation and the row presentation
are Tietze equivalent. This equivalence is implicit in~\cite{LuoJinZhangStylic}
at the level of rewriting systems. We record it here in the language of
$2$-polygraphs, since it will be used later in the construction.

\begin{proposition}
\label{P:PrerowRowTietze}
The $2$-polygraphs $\Sigma^{\mathrm{prerow}}(n)$ and
$\Sigma^{\mathrm{row}}(n)$ are Tietze equivalent.
\end{proposition}

\begin{proof}
By Proposition~\ref{P:PrerowPresentsStylic}, the $2$-polygraph
$\Sigma^{\mathrm{prerow}}(n)$ presents the stylic monoid $\Styl_n$.

We first adjoin the row $2$-cells. Let $u,v\in\Rr_n$ with
$u\not\triangleleft v$. 
The row $2$-cell~$\alpha_{u,v}$ preserves the stylic class, since the
words~$uv$ and~$\R(N(uv))$ represent the same element of~$\Styl_n$.  Moreover, the canonical path~$P_{u,v}$
defined in~\eqref{Eq:CanonicalPreRowRowCell} has the same source and target
as~$\alpha_{u,v}$ and uses only cells of the pre-row presentation.
Therefore each~$\alpha_{u,v}$ may be adjoined as a redundant $2$-cell,
using~$P_{u,v}$ as its specified replacement.

We now show that, after adjoining all row $2$-cells, the original pre-row
$2$-cells become redundant. First, the idempotent relation~$\eta_x^r$ is parallel to the row
$2$-cell~$\alpha_{x,x}$, and is therefore redundant.
For the lifted Knuth-type relations, the following diagrams show that they are
parallel, in the track $2$-category generated by the row cells, to composites
of row $2$-cells and their inverses. In the strict case $x<y<z$, one has
\[
\xymatrix @R=0.3em @C=2em @!C {
        & r_{yz}r_x
        \ar@2[r]^-{\alpha_{yz,x}}
        & r_y r_{xyz}
        \\
r_y r_z r_x
        \ar@2@/^/[ur]^-{\alpha_{y,z}1_{r_x}}
        \ar@2@/_/[dr]_-{\lambda_{x,y,z}^r}
        \\
        & r_y r_x r_z
        \ar@2[r]_-{1_{r_y}\alpha_{x,z}}
        & r_y r_{xz}
        \ar@2[uu]_-{\alpha_{y,xz}}
}
\qquad
\qquad
\xymatrix @R=0.3em @C=2em @!C {
        & r_z r_{xy}
        \ar@2[r]^-{\alpha_{z,xy}}
        & r_z r_{xyz}
        \\
r_z r_x r_y
        \ar@2@/^/[ur]^-{1_{r_z}\alpha_{x,y}}
        \ar@2@/_/[dr]_-{\kappa_{x,y,z}^r}
        \\
        & r_x r_z r_y
        \ar@2[r]_-{\alpha_{x,z}1_{r_y}}
        & r_{xz} r_y
        \ar@2[uu]_-{\alpha_{xz,y}}
}
\]
In the degenerate cases, one similarly has
\[
\xymatrix @R=0.3em @C=2em @!C {
        & r_y r_x
        \ar@2[r]^-{\alpha_{y,x}}
        & r_y r_{xy}
        \\
r_y r_y r_x
        \ar@2@/^/[ur]^-{\alpha_{y,y}1_{r_x}}
        \ar@2@/_/[dr]_-{\lambda_{x,y,y}^r}
        \\
        & r_y r_x r_y
        \ar@2[uur]_-{1_{r_y}\alpha_{x,y}}
}
\qquad
\qquad
\xymatrix @R=0.3em @C=2em @!C {
        &
        r_z r_x
        \ar@2[r]^-{\alpha_{z,x}}
        &
        r_z r_{xz}
        \\
r_z r_x r_x
        \ar@2@/^/[ur]^-{1_{r_z}\alpha_{x,x}}
        \ar@2@/_/[dr]_-{\kappa_{x,x,z}^r}
        \\
        &
        r_x r_z r_x
        \ar@2[r]_-{\alpha_{x,z}1_{r_x}}
        &
        r_{xz} r_x
        \ar@2[uu]_-{\alpha_{xz,x}}
}
\]
Each displayed diagram gives a parallel $2$-cell in the track $2$-category
generated by the row cells. Hence the corresponding lifted stylic relation is
redundant after the row cells have been adjoined.

It remains to study the compression $2$-cells. If $u=x_1x_2$ has length $2$,
then $\gamma_u$ is parallel to the row $2$-cell~$\alpha_{x_1,x_2}$.
Hence $\gamma_u$ is redundant after adjoining the row $2$-cells.
Now let~$u=x_1\cdots x_p$, with~$p\ge3$, and set~$u^-:=x_1\cdots x_{p-1}$.
Since $u^-x_p=u$ is a row, the row $2$-cell~$\alpha_{u^-,x_p}$ is defined. Hence $\gamma_u$ is parallel to the composite
\[
(\gamma_{u^-}1_{r_{x_p}})
\star_1
\alpha_{u^-,x_p}.
\]
By induction on $\ell(u)$, every compression $2$-cell $\gamma_u$ is therefore
parallel, after adjoining the row $2$-cells, to a composite of row $2$-cells.

Consequently, after adjoining the row $2$-cells, all lifted stylic relations and
all compression cells are redundant and may be eliminated by elementary Tietze transformations. The remaining generating $2$-cells are precisely the row cells
$\alpha_{u,v}$. Hence the resulting $2$-polygraph is
$\Sigma^{\mathrm{row}}(n)$,  showing that~$\Sigma^{\mathrm{prerow}}(n)$ and~$\Sigma^{\mathrm{row}}(n)$ are Tietze equivalent.
\end{proof}

\section{A coherent extension of the row presentation}
\label{S:CoherentPresentationStylicMonoid}

In this section, we construct a coherent extension of the finite convergent row
presentation of the stylic monoid. We first recall the
notion of coherent presentations and Squier's coherence theorem. We then describe the
critical branchings of the row presentation, compute explicit generating confluences,
and adjoin one generating $3$-cell for each of them. This yields a finite coherent presentation of the stylic monoid.

\subsection{Coherent presentations}
\label{SS:CoherentPresentations}
We recall the notion of coherent presentation in the language of higher-dimensional rewriting, following~\cite{GuiraudMalbos18,GaussentGuiraudMalbos15,PolygraphsBook}.

Let~$\Sigma$ be a $2$-polygraph.
A pair~$(f,g)$ of $2$-cells in~$\Sigma_2^\top$ such that~$s_1(f)=s_1(g)$ and~$t_1(f)=t_1(g)$ is called a \emph{$2$-sphere} of~$\Sigma_2^\top$.
A \emph{$(3,1)$-polygraph} is a pair $(\Sigma,\Sigma_3)$ consisting
of a $2$-polygraph~$\Sigma$ together with a cellular extension
$\Sigma_3$ of the $(2,1)$-category~$\Sigma_2^\top$. Thus
$\Sigma_3$ is a family of generating $3$-cells of the form
\[
A\colon f\tfl g,
\]
where $(f,g)$ is a $2$-sphere of~$\Sigma_2^\top$.
The $2$-cells $f$ and $g$ are respectively called the \emph{source} and the
\emph{target} of~$A$, and are denoted by~$s_2(A)$ and~$t_2(A)$.
Such a $3$-cell may be represented by one of the following globular diagrams:
\[
\xymatrix @C=8em {
\bullet
	\ar @/^4ex/ [r] ^-{u} _-{}="src"
	\ar @/_4ex/ [r] _-{v} ^-{}="tgt"
	\ar@2 "src"!<-15pt,-10pt>;"tgt"!<-15pt,10pt> _-*+{f} ^-{}="srcA"
	\ar@2 "src"!<+15pt,-10pt>;"tgt"!<+15pt,10pt> ^-*+{g} _-{}="tgtA"
	\ar@3 "srcA"!<8pt,0pt> ; "tgtA"!<-8pt,0pt> ^-{A}
&
\bullet
}
\qquad\text{or equivalently}\qquad
\xymatrix@!C@C=4.5em{
u
	\ar@2@/^3ex/ [r] ^{f} _{}="src"
	\ar@2@/_3ex/ [r] _{g} ^{}="tgt"
&
v
\ar@3 "src"!<0pt,-7pt>;"tgt"!<0pt,7pt> ^-{A}
}
\]
where~$\bullet$ denotes the unique $0$-cell of~$\Sigma$.

We denote by~$\tck{\Sigma}_3$ the free $(3,1)$-category generated by
$(\Sigma,\Sigma_3)$. In this category, all $2$-cells and $3$-cells are
invertible. 
We denote the $2$-composition of $3$-cells by~$\star_2$ and use
$\star_1$ for their $1$-composition, including whiskering by
$2$-cell contexts. For a $2$-cell~$f$, we denote by~$1_f$ the
identity $3$-cell on~$f$.
The cellular extension~$\Sigma_3$ is called a
\emph{homotopy basis} of~$\Sigma_2^\top$ if every $2$-sphere
$(f,g)$ in~$\Sigma_2^\top$ admits a filling, that is, if there exists
a $3$-cell~$A$ in~$\tck{\Sigma}_3$ such that~$s_2(A)=f$ and
$t_2(A)=g$.
An \emph{extended presentation} of a monoid~$M$ is a
$(3,1)$-polygraph whose underlying $2$-polygraph presents~$M$.
It is a \emph{coherent presentation} if~$\Sigma_3$ is a homotopy
basis of~$\Sigma_2^\top$.

\emph{Squier's coherence theorem} asserts that every convergent
$2$-polygraph presenting a monoid~$M$ can be extended to a coherent
presentation of~$M$ by adjoining, for each critical branching~$(f,g)$,
a generating $3$-cell between the two rewriting paths obtained by
completing~$f$ and~$g$ to a common target
\cite{Squier94,GuiraudMalbos18}. More precisely, one adds a generating
$3$-cell of the form
\[
\xymatrix @R=0.8em @C=2.5em @!C {
        & v
        \ar@2@/^/[dr]^-{f'}
        \ar@3[]!<0pt,-10pt>;[dd]!<0pt,10pt>
          ^-*+{A_{f,g}}
        \\
        u
        \ar@2@/^/[ur]^-{f}
        \ar@2@/_/[dr]_-{g}
        && w
        \\
        & v'
        \ar@2@/_/[ur]_-{g'}
}
\]
where $f'$ and $g'$ are rewriting paths with common target~$w$.
The generating $3$-cells $A_{f,g}$ obtained in this way are called the
\emph{generating confluences}  associated
with the critical branchings~$(f,g)$.

A generating $3$-cell~$A$ is called \emph{redundant} if, after removing~$A$, the generated $(3,1)$-category contains a parallel composite
$3$-cell. Such a generating cell may be adjoined or eliminated by a
coherent Tietze transformation.
Let~$(\Sigma,\Sigma_3)$ be a coherent presentation, and let
\[
\mathcal T\colon
\Sigma\rightsquigarrow\Upsilon
\]
be a finite Tietze sequence between the underlying $2$-polygraphs. At
each adjunction or elimination of a redundant generating $2$-cell
$\beta\colon u\dfl v$, choose a parallel path~$\widehat\beta\colon u\dfl v$
in the track $2$-category generated without~$\beta$.
The sequence~$\mathcal T$ lifts step by step to coherent
presentations. The adjunction or elimination of a redundant generating
$1$-cell together with its collapsible $2$-cell already defines a
coherent Tietze transformation. When a redundant generating $2$-cell
$\beta$ is adjoined, one also adjoins~$B_\beta\colon\widehat\beta\tfl~\beta$.
Since~$\widehat\beta$ does not contain~$\beta$, the pair
$(\beta,B_\beta)$ is collapsible.
Conversely, suppose that~$\beta$ is to be eliminated. Coherence gives
a filling~$H_\beta\colon\widehat\beta\tfl\beta$ of the $2$-sphere~$(\widehat\beta,\beta)$. One first adjoins a
generating $3$-cell~$B_\beta\colon\widehat\beta\tfl~\beta$ parallel to~$H_\beta$, and then eliminates the collapsible pair~$(\beta,B_\beta)$. Thus~$H_\beta$ shows that~$B_\beta$ is redundant,
while~$B_\beta$ makes~$\beta$ collapsible.
After this elimination, the boundaries of the remaining generating
$3$-cells are transported by the induced track $2$-functor~$\Phi_\beta$. This functor sends~$\beta$ to~$\widehat\beta$ and fixes every other generating $2$-cell. Hence a generating $3$-cell~$A\colon f\tfl g$ is transported to~$A^\beta\colon
\Phi_\beta(f)\tfl\Phi_\beta(g)$.
Iterating this construction yields a cellular extension~$\Upsilon_3$
and a finite Tietze transformation
\[
\widetilde{\mathcal T}\colon
(\Sigma,\Sigma_3)
\rightsquigarrow
(\Upsilon,\Upsilon_3).
\]
Tietze transformations preserve the presented monoid and coherence
\cite[Theorem~2.1.3]{GaussentGuiraudMalbos15}. Hence
$(\Upsilon,\Upsilon_3)$ is a coherent presentation of the same monoid.
If~$(\Sigma,\Sigma_3)$ and~$\mathcal T$ are finite, then
$(\Upsilon,\Upsilon_3)$ is finite. We call~$\widetilde{\mathcal T}$ a
\emph{coherent lift} of~$\mathcal T$.

\subsection{A coherent row presentation of the stylic monoid}
\label{SS:CoherentPresentationStylicMonoid}

The goal of this subsection is to extend the finite convergent row presentation
$\Sigma^{\mathrm{row}}(n)$ to a coherent presentation by adjoining one generating
$3$-cell for each critical branching of the generating $2$-cells~$\alpha_{u,v}$.

\begin{lemma}
\label{L:CriticalBranchingsRow}
The critical branchings of the row presentation $\Sigma^{\mathrm{row}}(n)$ are
exactly the branchings
\[
\bigl(\alpha_{u,v}1_{r_t},\,1_{r_u}\alpha_{v,t}\bigr)
\]
for triples $u,v,t\in \Rr_n$ with $u\not\triangleleft v$ and $v\not\triangleleft t$.
\end{lemma}

\begin{proof}
Since every row rewriting rule has source of length~$2$, a minimal overlapping
branching is obtained exactly by overlapping two consecutive rules in a word of
length~$3$. Hence its source is $r_ur_vr_t$, and the branching has the form~$\bigl(\alpha_{u,v}1_{r_t},\,1_{r_u}\alpha_{v,t}\bigr)$, with $u\not\triangleleft v$ and $v\not\triangleleft t$. Conversely, every such
branching is overlapping and has no nontrivial context, hence is critical.
\end{proof}

\begin{lemma}
\label{L:UsefulReduciblePairsCriticalTriple}
Let $(u,v,t)\in\Rr_n^3$ be a critical triple. Write $\R(N(uv))=ab$ with
$a\in\Rr_n\cup\{\varepsilon\}$ and $b\in\Rr_n$, and write
$\R(N(vt))=cd$ with $c\in\Rr_n\cup\{\varepsilon\}$ and $d\in\Rr_n$.
Then:
\begin{enumerate}[\bf i)]
\item one has $b\not\triangleleft t$;
\item if $c\neq\varepsilon$, then $u\not\triangleleft c$.
\end{enumerate}
In particular, the generating $2$-cells $\alpha_{b,t}$ and, when $c\neq\varepsilon$,
$\alpha_{u,c}$ are applicable.
\end{lemma}

\begin{proof}
By Lemma~\ref{L:BottomRowUnion}, the bottom row $b$ of $N(uv)$ satisfies
$\supp(b)=\supp(u)\cup\supp(v)$. In particular,
$\supp(v)\subseteq\supp(b)$ and $\min(b)\le\min(v)$.
If $b\triangleleft t$, then $\supp(b)\subseteq\supp(t)$ and
$\min(t)<\min(b)$. Hence $\supp(v)\subseteq\supp(t)$ and
$\min(t)<\min(v)$, which gives $v\triangleleft t$, contradicting the
criticality of $(u,v,t)$. Therefore $b\not\triangleleft t$.
Suppose now that $c\neq\varepsilon$. Then $N(vt)$ has two rows, with top row
$c$ and bottom row $d$. By Lemma~\ref{L:TopRowInclusion}, one has
$\supp(c)\subseteq\supp(v)$. In particular, $\min(v)\le\min(c)$.
If $u\triangleleft c$, then $\supp(u)\subseteq\supp(c)\subseteq\supp(v)$ and
$\min(v)\le\min(c)<\min(u)$. Thus $u\triangleleft v$, again contradicting the
criticality of $(u,v,t)$. Therefore~$u\not\triangleleft~c$.
\end{proof}

Retain the notation of Lemma~\ref{L:UsefulReduciblePairsCriticalTriple}.
The critical branching associated with~$(u,v,t)$ is
\[
r_u r_v r_t
\odfl{\alpha_{u,v}1_{r_t}}
r_a r_b r_t,
\qquad
r_u r_v r_t
\odfl{1_{r_u}\alpha_{v,t}}
r_u r_c r_d,
\]
where the factors~$r_a$ and~$r_c$ are omitted when
$a=\varepsilon$ and~$c=\varepsilon$, respectively.

For any triple of rows~$(u,v,t)$, set~$\operatorname{NF}(u,v,t):=\Rr(N(uvt))$.
This word is irreducible, since it is the row-generator reading of an
$N$-tableau, and it represents the same stylic element as~$r_ur_vr_t$.
By convergence of~$\Sigma^{\mathrm{row}}(n)$, it is therefore the unique
normal form of~$r_ur_vr_t$. In particular, the two branches above have
common normal form~$\operatorname{NF}(u,v,t)$.

We first establish the bound on the number of rewriting steps, which will 
determine the possible shapes of the confluence diagrams in the case analysis.

\begin{proposition}
\label{Prop:boundedBranching}
Let $(u,v,t)\in \Rr_n^3$ be a critical triple. Then, after either of the two initial branching steps~$\alpha_{u,v}1_{r_t}$ or~$1_{r_u}\alpha_{v,t}$,
the resulting word reaches the normal form $\operatorname{NF}(u,v,t)$ in at
most two further rewriting steps. Moreover:
\begin{enumerate}[\bf i)]
\item The normal form has at most three rows.
\item If $N(uv)$ has one row then the left branch reaches the normal 
      form in exactly one further step.
\item If $N(vt)$ has one row then the right branch reaches the normal 
      form in at most one further step.
\item If both $N(uv)$ and $N(vt)$ have one row then the normal form 
      has one row.
\end{enumerate}
\end{proposition}

\begin{proof}
Since $\Sigma^{\mathrm{row}}(n)$ is convergent, every irreducible word reached
from $r_u r_v r_t$ is necessarily the normal form
$\operatorname{NF}(u,v,t)$.

We first study the left branch.
Write $\R(N(uv))=ab$ with $a\in\Rr_n\cup\{\varepsilon\}$ and $b\in\Rr_n$.
After the first step $\alpha_{u,v}1_{r_t}$, the left branch reaches 
$r_a r_b r_t$ with $r_a$ omitted if $a=\varepsilon$.
Since $a\triangleleft b$ when $a\neq\varepsilon$, only the pair~$(b,t)$ 
can be reducible.
By Lemma~\ref{L:UsefulReduciblePairsCriticalTriple}, one has
$b\not\triangleleft t$. Hence~$\alpha_{b,t}$ is necessarily applicable. Write~$\R(N(bt))=ef$,
with $e\in\Rr_n\cup\{\varepsilon\}$ and $f\in\Rr_n$, and apply the $2$-cell
$\alpha_{b,t}$. Thus, after  one further rewriting step, the left
branch reaches a word of the form $r_a r_e r_f$, where $r_a$ and $r_e$ are
omitted whenever $a=\varepsilon$ or $e=\varepsilon$. Moreover, when
$e\neq\varepsilon$, one has $e\triangleleft f$.

If $a=\varepsilon$, then the word $r_e r_f$ is already in normal form, with
$r_e$ omitted if $e=\varepsilon$, because either $e=\varepsilon$ or
$e\triangleleft f$. This proves~{\bf ii)}.

Assume now that $a\neq\varepsilon$.
If $e=\varepsilon$, then $\R(N(bt))=f$. Since $a\triangleleft b$, one has
$\supp(a)\subseteq\supp(b)$ and $\min(b)<\min(a)$. By
Lemma~\ref{L:BottomRowUnion}, $\supp(b)\subseteq\supp(f)$, and clearly
$\min(f)\leq \min(b)$. Hence $\supp(a)\subseteq\supp(f)$ and
$\min(f)<\min(a)$, so $a\triangleleft f$. Thus $r_a r_f$ is already in
normal form.

If $e\neq\varepsilon$ and $a\triangleleft e$, the word $r_a r_e r_f$ is
already in normal form.
If $e\neq\varepsilon$ and $a\not\triangleleft e$, one further step
$\alpha_{a,e}1_{r_f}$ is needed. Write $\R(N(ae))=pq$, with
$p\in\Rr_n\cup\{\varepsilon\}$ and $q\in\Rr_n$.
Since $\R(N(ae))=pq$, one has $p\triangleleft q$ whenever
$p\neq\varepsilon$. We claim that $q\triangleleft f$.
By Lemma~\ref{L:BottomRowUnion}, $\supp(q)=\supp(a)\cup\supp(e)$.
Since $a\triangleleft b$, we have
$\supp(a)\subseteq\supp(b)\subseteq\supp(f)$.
Since $e\triangleleft f$, we have $\supp(e)\subseteq\supp(f)$.
Hence $\supp(q)\subseteq\supp(f)$.
Moreover, $\min(f)\le\min(b)<\min(a)$ and $\min(f)<\min(e)$, so
$\min(f)<\min(q)$. Therefore $q\triangleleft f$.
Thus the word $r_p r_q r_f$, with $r_p$ omitted if $p=\varepsilon$, is in
normal form. Hence the left branch reaches the normal form in at most two
further steps.

Let us now study the right branch.
Write $\R(N(vt))=cd$ with $c\in\Rr_n\cup\{\varepsilon\}$ and $d\in\Rr_n$.
After the first step $1_{r_u}\alpha_{v,t}$, the right branch reaches
$r_u r_c r_d$ with $r_c$ omitted if $c=\varepsilon$.

\begin{enumerate}[\bf A)]
\item Suppose that~$c=\varepsilon$.
If $u\triangleleft d$, then $r_u r_d$ is already in normal form.
If $u\not\triangleleft d$, write $\R(N(ud))=pq$, with
$p\in\Rr_n\cup\{\varepsilon\}$ and $q\in\Rr_n$.
Then one further step $\alpha_{u,d}$ yields $r_p r_q$.
Since $p\triangleleft q$ whenever $p\neq\varepsilon$, the target is in normal
form. This gives at most one further step, proving~{\bf iii)}.

\item Suppose that $c\neq\varepsilon$.
Since $c\triangleleft d$, only the pair $(u,c)$ can be reducible.
By Lemma~\ref{L:UsefulReduciblePairsCriticalTriple}, one has
$u\not\triangleleft c$. Hence the generating $2$-cell
$\alpha_{u,c}$ is necessarily applicable.
Write~$\R(N(uc))=ij$,
with $i\in\Rr_n\cup\{\varepsilon\}$ and $j\in\Rr_n$, and apply the $2$-cell
$\alpha_{u,c}$. Thus, after at most one further rewriting step, the right
branch reaches a word of the form $r_i r_j r_d$, where $r_i$ is omitted if
$i=\varepsilon$. Moreover, when $i\neq\varepsilon$, one has
$i\triangleleft j$.

 Suppose that~$i=\varepsilon$.
If $j\triangleleft d$, the word~$r_j r_d$ is already in normal form.
If $j\not\triangleleft d$, write $\R(N(jd))=pq$, with
$p\in\Rr_n\cup\{\varepsilon\}$ and $q\in\Rr_n$.
Then one further step $\alpha_{j,d}$ yields $r_p r_q$.
Since $p\triangleleft q$ whenever $p\neq\varepsilon$, the target is in normal
form. Thus in this case the right branch reaches the normal form in at most
two further steps.

Suppose that~$i\neq\varepsilon$.
Since $i\triangleleft j$, the word $r_i r_j r_d$ is already in normal form
whenever $j\triangleleft d$.
Assume now that $j\not\triangleleft d$. Then one further step
$1_{r_i}\alpha_{j,d}$ is needed. Write~$\R(N(jd))=pq$,
with $p\in\Rr_n\cup\{\varepsilon\}$ and $q\in\Rr_n$.
We prove that $p\neq\varepsilon$ and $i\triangleleft p$, so that
$r_i r_p r_q$ is in normal form.

\begin{enumerate}[\bf a)]
\item Prove that $p\neq\varepsilon$.
By Lemma~\ref{L:UsefulReduciblePairsCriticalTriple},~$\alpha_{u,c}$ has already been applied. Since
$\R(N(uc))=~ij$, Lemma~\ref{L:BottomRowUnion} gives
\[
\supp(j)=\supp(u)\cup\supp(c).
\]
In particular, $\supp(c)\subseteq~\supp(j)$, so $\max(j)\ge~\min(c)$.
Since $c\triangleleft d$, one has $\min(d)<\min(c)$. Hence
$\min(d)<\max(j)$. By Lemma~\ref{L:OneRowCriterion}, the tableau $N(jd)$
has two rows, and therefore $p\neq\varepsilon$.

\item Prove now that $i\triangleleft p$.
We first show that $\supp(i)\subseteq\supp(p)$.
Consider $y\in\supp(i)$.
Since $i$ is the top row of $N(uc)$, every letter of $i$ is one of the
letters bumped from the bottom row during the insertion of the letters of
$c$ into $u$. Thus the letter $y$ was bumped from $u$ by some
$x\in\supp(c)$ during the computation of $N(uc)$. Since the letters of
$c$ are inserted in increasing order, one has
\[
y=\min\{z\in\supp(u)\mid z>x\}.
\]
Set~$x'=\max\{z\in\supp(d)\mid z<y\}$.
This is well defined because $x\in\supp(c)\subseteq\supp(d)$ and~$x<y$.
If there existed $z\in\supp(j)$ with $x'<z<y$, then $z\notin\supp(d)$ by
maximality of~$x'$. Since~$\supp(c)\subseteq\supp(d)$, this implies
$z\notin\supp(c)$. But $\supp(j)=\supp(u)\cup\supp(c)$, so
$z\in\supp(u)$. Moreover, $x\le x'<z<y$, contradicting the minimality of
$y$. Therefore there is no letter of $\supp(j)$ strictly between $x'$ and~$y$.
The bumped letter $y$ remains in the bottom row during the computation of
$N(uc)$, so $y\in\supp(j)$. If $y$ has already been bumped before the
insertion of $x'$ during the computation of $N(jd)$, then
$y\in\supp(p)$. Otherwise, since the letters of $d$ are inserted in
increasing order, all letters inserted before $x'$ are smaller than $x'$ and
therefore do not create any new letter strictly between $x'$ and $y$.
Hence, when $x'$ is inserted into the current bottom row, the letter $y$ is
bumped. Thus $y\in\supp(p)$, showing that~$\supp(i)\subseteq\supp(p)$.
It remains to prove that $\min(p)<\min(i)$.
Every letter of $i$ was bumped by some letter of~$c$, so
$\min(i)>\min(c)$. Since $\supp(c)\subseteq\supp(j)$, the letter $\min(c)$
belongs to $j$. Because $\min(d)<\min(c)$, inserting $\min(d)$ into $j$
bumps a letter less than or equal to $\min(c)$ into $\supp(p)$. Hence~$\min(p)\le\min(c)<\min(i)$.
Therefore $i\triangleleft p$.
Thus $r_i r_p r_q$ is in normal form.
\end{enumerate}
\end{enumerate}

Hence, the right branch reaches the normal form in at most two further steps.

In the worst case, the irreducible word reached from either branch has three
row generators, namely~$r_p r_q r_f$ on the left or $r_i r_p r_q$ on the
right. Since this irreducible word is $\operatorname{NF}(u,v,t)$, the normal
form has at most three rows. This proves~{\bf i)}.

Finally, suppose that both $N(uv)$ and $N(vt)$ have one row. Then
$a=\varepsilon$ and $c=\varepsilon$. By Lemma~\ref{L:OneRowCriterion}, one has
$\min(v)\ge\max(u)$ and $\min(t)\ge\max(v)$. Hence the insertion of the letters
of $v$ into $u$ creates no bump, and then the insertion of the letters of $t$
into the resulting row also creates no bump. Thus~$N(uvt)$ has one row. Hence
the normal form $\operatorname{NF}(u,v,t)$ has one row, proving~{\bf iv)}.
\end{proof}

\subsubsection{Explicit generating confluences}
\label{SSS:ExplicitGeneratingConfluences}

For every critical triple $(u,v,t)$, we denote by $\mathcal X_{u,v,t}$ the
generating $3$-cell filling the corresponding confluence diagram. By
Proposition~\ref{Prop:boundedBranching}, each branch reaches the common normal
form $\operatorname{NF}(u,v,t)$ after at most two further rewriting steps.
The following tables summarize the notation and the normal forms appearing in
the diagrams. Whenever a row symbol is equal to~$\varepsilon$, the
corresponding row generator is omitted.
The proof that these diagrams are well defined, cover all possible
cases, and end at~$\operatorname{NF}(u,v,t)$ is given in
Appendix~\ref{A:ExplicitGeneratingConfluences}.

\medskip

\noindent
\textbf{Cases $1$ and $2$.}

\[
\begin{array}{|c|c|c|}
\hline
\text{Case} & \text{Notation and assumptions} & \operatorname{NF}(u,v,t) \\ \hline
1
&
\R(N(uv))=b,\ \R(N(vt))=d,\ m=u\cup v\cup t
&
r_m
\\[0.4em]
2\mathrm a
&
\R(N(uv))=b,\ \R(N(vt))=cd,\ \R(N(uc))=g,\ g\triangleleft d
&
r_g r_d
\\[0.4em]
2\mathrm b
&
\R(N(uv))=b,\ \R(N(vt))=cd,\ \R(N(uc))=g,\ g\not\triangleleft d,\ \R(N(gd))=pq
&
r_p r_q\\\hline
\end{array}
\]

\medskip

\noindent
\textbf{Case $3$.}
In this case, $\R(N(uv))=ab$ with $a\neq\varepsilon$,
$\R(N(vt))=d$, and $\R(N(bt))=ef$ with
$e\in\Rr_n\cup\{\varepsilon\}$.

\[
\begin{array}{|c|c|c|}
\hline
\text{Case} & \text{Additional assumptions and notation}
& \operatorname{NF}(u,v,t) \\ \hline
3\mathrm a1
&
e=\varepsilon,\ \R(N(bt))=f,\ u\triangleleft d
\ \text{; hence } a=u,\ f=d
&
r_u r_d
\\[0.4em]
3\mathrm a2
&
e=\varepsilon,\ \R(N(bt))=f,\ u\not\triangleleft d,\ \R(N(ud))=af
&
r_a r_f
\\[0.4em]
3\mathrm b1
&
e\neq\varepsilon,\ \R(N(ae))=g,\ u\triangleleft d
\ \text{; hence } g=u,\ f=d
&
r_u r_d
\\[0.4em]
3\mathrm b2
&
e\neq\varepsilon,\ \R(N(ae))=g,\ u\not\triangleleft d,\ \R(N(ud))=gf
&
r_g r_f\\ \hline
\end{array}
\]

\medskip

\noindent
\textbf{Case $4$.}
In this case, $\R(N(uv))=ab$, $\R(N(vt))=cd$ and~$\R(N(bt))=ef$ with~$a\neq\varepsilon$,~$c\neq\varepsilon$ and~$e\neq\varepsilon$.

\[
\begin{array}{|c|c|c|}
\hline
\text{Case} & \text{Additional assumptions and notation} & \operatorname{NF}(u,v,t) \\ \hline
4\mathrm a1
&
\R(N(uc))=g,\ g\triangleleft d,\ \R(N(ae))=g',\ g'=g,\ f=d
&
r_g r_d
\\[0.4em]
4\mathrm a2
&
\R(N(uc))=g,\ g\not\triangleleft d,\ \R(N(gd))=pf,\ \R(N(ae))=p',\ p'=p
&
r_p r_f
\\[0.4em]
4\mathrm b1
&
\R(N(uc))=ij,\ a\triangleleft e,\ j\triangleleft d,\ i=a,\ j=e,\ d=f
&
r_a r_e r_f
\\[0.4em]
4\mathrm b2
&
\R(N(uc))=ij,\ a\triangleleft e,\ j\not\triangleleft d,\ \R(N(jd))=sf,\ i=a,\ s=e
&
r_a r_e r_f
\\[0.4em]
4\mathrm b3
&
\R(N(uc))=ij,\ a\not\triangleleft e,\ j\triangleleft d,\ \R(N(ae))=pq,\ p=i,\ q=j,\ f=d
&
r_i r_j r_d
\\[0.4em]
4\mathrm b4
&
\R(N(uc))=ij,\ a\not\triangleleft e,\ j\not\triangleleft d,\ \R(N(ae))=pq,\ \R(N(jd))=sf,\ p=i,\ q=s
&
r_i r_s r_f\\ \hline
\end{array}
\]

\medskip

We now display the corresponding generating confluences. The superscript records the corresponding case and will be omitted
after the case analysis.

\medskip

\noindent
\begin{minipage}[t]{0.32\textwidth}
\centering
\textbf{Case $1$.}

\caseconfluence{
\xymatrix @R=0.75em @C=1.8em @!C {
        & r_b r_t
         \ar@2@/^/[dr]^-{\alpha_{b,t}}
        \ar@3[]!<0pt,-8pt>;[dd]!<0pt,8pt>^-*+{\mathcal X^{(1)}_{u,v,t}}
        \\
r_u r_v r_t
        \ar@2@/^/[ur]^-{\alpha_{u,v}1_{r_t}}
        \ar@2@/_/[dr]_-{1_{r_u}\alpha_{v,t}}
&& r_m
        \\
        & r_u r_d
        \ar@2@/_/[ur]_-{\alpha_{u,d}}
}
}
\end{minipage}
\hfill
\begin{minipage}[t]{0.32\textwidth}
\centering
\textbf{Case $2a$.}

\caseconfluence{
\xymatrix @R=0.75em @C=1.8em @!C {
        & r_b r_t
          \ar@2@/^/[dr]^-{\alpha_{b,t}}
        \ar@3[]!<0pt,-8pt>;[dd]!<0pt,8pt>^-*+{\mathcal X^{(2a)}_{u,v,t}}
        \\
r_u r_v r_t
        \ar@2@/^/[ur]^-{\alpha_{u,v}1_{r_t}}
        \ar@2@/_/[dr]_-{1_{r_u}\alpha_{v,t}}
& & r_g r_d
\\
        & r_u r_c r_d
          \ar@2@/_/[ur]_-{\alpha_{u,c}1_{r_d}}
}
}
\end{minipage}
\hfill
\begin{minipage}[t]{0.32\textwidth}
\centering
\textbf{Case $2b$.}

\caseconfluence{
\xymatrix @R=0.75em @C=1.8em @!C {
        & r_b r_t
         \ar@2[r]^-{\alpha_{b,t}}
        \ar@3[]!<0pt,-8pt>;[dd]!<0pt,8pt>^-*+{\mathcal X^{(2b)}_{u,v,t}}
    & r_p r_q
        \\
r_u r_v r_t
        \ar@2@/^/[ur]^-{\alpha_{u,v}1_{r_t}}
        \ar@2@/_/[dr]_-{1_{r_u}\alpha_{v,t}}
        \\
        & r_u r_c r_d
        \ar@2[r]_-{\alpha_{u,c}1_{r_d}}
        & r_g r_d
        \ar@2[uu]_-{\alpha_{g,d}}
}
}
\end{minipage}

\medskip

\noindent
\begin{minipage}[t]{0.32\textwidth}
\centering
\textbf{Case $3a1$.}

\caseconfluence{
\xymatrix @R=0.75em @C=1.8em @!C {
        & r_a r_b r_t
         \ar@2@/^/[dd]^-{1_{r_a}\alpha_{b,t}}
\ar@3[]!<-30pt,-10pt>;[dd]!<-30pt,10pt>^-*+{\mathcal X^{(3a1)}_{u,v,t}}        
        \\
r_u r_v r_t
        \ar@2@/^/[ur]^-{\alpha_{u,v}1_{r_t}}
        \ar@2@/_/[dr]_-{1_{r_u}\alpha_{v,t}}
\\
        & r_u r_d
}
}
\end{minipage}
\hfill
\begin{minipage}[t]{0.32\textwidth}
\centering
\textbf{Case $3a2$.}

\caseconfluence{
\xymatrix @R=0.75em @C=1.8em @!C {
        & r_a r_b r_t
         \ar@2@/^/[dr]^-{1_{r_a}\alpha_{b,t}}
        \ar@3[]!<0pt,-8pt>;[dd]!<0pt,8pt>^-*+{\mathcal X^{(3a2)}_{u,v,t}}
        \\
r_u r_v r_t
        \ar@2@/^/[ur]^-{\alpha_{u,v}1_{r_t}}
        \ar@2@/_/[dr]_-{1_{r_u}\alpha_{v,t}}
  && r_a r_f
\\
        & r_u r_d
        \ar@2@/_/[ur]_-{\alpha_{u,d}}
}
}
\end{minipage}
\hfill
\begin{minipage}[t]{0.32\textwidth}
\centering
\textbf{Case $3b1$.}

\caseconfluence{
\xymatrix @R=0.75em @C=1.8em @!C {
        & r_a r_b r_t
        \ar@2[r]^-{1_{r_a}\alpha_{b,t}}
        \ar@3[]!<0pt,-8pt>;[dd]!<0pt,8pt>^-*+{\mathcal X^{(3b1)}_{u,v,t}}
        & r_a r_e r_f
          \ar@2@/^/[ddl]^-{\alpha_{a,e}1_{r_f}}
        \\
r_u r_v r_t
        \ar@2@/^/[ur]^-{\alpha_{u,v}1_{r_t}}
        \ar@2@/_/[dr]_-{1_{r_u}\alpha_{v,t}}
        \\
        & r_u r_d
}
}
\end{minipage}
\medskip

\noindent
\begin{minipage}[t]{0.49\textwidth}
\centering
\textbf{Case $3b2$.}

\caseconfluence{
\xymatrix @R=0.75em @C=1.8em @!C {
        & r_a r_b r_t
        \ar@2[r]^-{1_{r_a}\alpha_{b,t}}
        \ar@3[]!<0pt,-8pt>;[dd]!<0pt,8pt>
          ^-*+{\mathcal X^{(3b2)}_{u,v,t}}
        & r_a r_e r_f
        \ar@2[dd]^-{\alpha_{a,e}1_{r_f}}
        \\
r_u r_v r_t
        \ar@2@/^/[ur]^-{\alpha_{u,v}1_{r_t}}
        \ar@2@/_/[dr]_-{1_{r_u}\alpha_{v,t}}
        \\
        & r_u r_d
        \ar@2[r]_-{\alpha_{u,d}}
        & r_g r_f
}
}
\end{minipage}
\hfill
\begin{minipage}[t]{0.49\textwidth}
\centering
\textbf{Case $4a1$.}

\caseconfluence{
\xymatrix @R=0.75em @C=1.8em @!C {
        & r_a r_b r_t
        \ar@2[r]^-{1_{r_a}\alpha_{b,t}}
        \ar@3[]!<0pt,-8pt>;[dd]!<0pt,8pt>
          ^-*+{\mathcal X^{(4a1)}_{u,v,t}}
        & r_a r_e r_f
        \ar@2[dd]^-{\alpha_{a,e}1_{r_f}}
        \\
r_u r_v r_t
        \ar@2@/^/[ur]^-{\alpha_{u,v}1_{r_t}}
        \ar@2@/_/[dr]_-{1_{r_u}\alpha_{v,t}}
        \\
        & r_u r_c r_d
        \ar@2[r]_-{\alpha_{u,c}1_{r_d}}
        & r_g r_d
}
}
\end{minipage}

\medskip

\noindent
\begin{minipage}[t]{0.49\textwidth}
\centering
\textbf{Case $4a2$.}

\caseconfluence{
\xymatrix @R=0.75em @C=1.8em @!C {
        & r_a r_b r_t
        \ar@2[r]^-{1_{r_a}\alpha_{b,t}}
        \ar@3[]!<16pt,-8pt>;[dd]!<16pt,8pt>
          ^-*+{\mathcal X^{(4a2)}_{u,v,t}}
        & r_a r_e r_f
        \ar@2@/^/[dr]^-{\alpha_{a,e}1_{r_f}}
        \\
r_u r_v r_t
        \ar@2@/^/[ur]^-{\alpha_{u,v}1_{r_t}}
        \ar@2@/_/[dr]_-{1_{r_u}\alpha_{v,t}}
        &&& r_p r_f
        \\
        & r_u r_c r_d
        \ar@2[r]_-{\alpha_{u,c}1_{r_d}}
        & r_g r_d
        \ar@2@/_/[ur]_-{\alpha_{g,d}}
}
}
\end{minipage}
\hfill
\begin{minipage}[t]{0.49\textwidth}
\centering
\textbf{Case $4b1$.}

\caseconfluence{
\xymatrix @R=0.75em @C=1.8em @!C {
        & r_a r_b r_t
        \ar@2@/^/[dr]^-{1_{r_a}\alpha_{b,t}}
        \ar@3[]!<0pt,-8pt>;[dd]!<0pt,8pt>
          ^-*+{\mathcal X^{(4b1)}_{u,v,t}}
        \\
r_u r_v r_t
        \ar@2@/^/[ur]^-{\alpha_{u,v}1_{r_t}}
        \ar@2@/_/[dr]_-{1_{r_u}\alpha_{v,t}}
        && r_a r_e r_f
        \\
        & r_u r_c r_d
        \ar@2@/_/[ur]_-{\alpha_{u,c}1_{r_d}}
}
}
\end{minipage}

\medskip

\noindent
\begin{minipage}[t]{0.49\textwidth}
\centering
\textbf{Case $4b2$.}

\caseconfluence{
\xymatrix @R=0.75em @C=1.8em @!C {
        & r_a r_b r_t
        \ar@2[r]^-{1_{r_a}\alpha_{b,t}}
        \ar@3[]!<0pt,-8pt>;[dd]!<0pt,8pt>
          ^-*+{\mathcal X^{(4b2)}_{u,v,t}}
        & r_a r_e r_f
        \\
r_u r_v r_t
        \ar@2@/^/[ur]^-{\alpha_{u,v}1_{r_t}}
        \ar@2@/_/[dr]_-{1_{r_u}\alpha_{v,t}}
        \\
        & r_u r_c r_d
        \ar@2[r]_-{\alpha_{u,c}1_{r_d}}
        & r_a r_j r_d
        \ar@2[uu]_-{1_{r_a}\alpha_{j,d}}
}
}
\end{minipage}
\hfill
\begin{minipage}[t]{0.49\textwidth}
\centering
\textbf{Case $4b3$.}

\caseconfluence{
\xymatrix @R=0.75em @C=1.8em @!C {
        & r_a r_b r_t
        \ar@2[r]^-{1_{r_a}\alpha_{b,t}}
        \ar@3[]!<16pt,-8pt>;[dd]!<16pt,8pt>
          ^-*+{\mathcal X^{(4b3)}_{u,v,t}}
        & r_a r_e r_f
        \ar@2@/^/[dd]^-{\alpha_{a,e}1_{r_f}}
        \\
r_u r_v r_t
        \ar@2@/^/[ur]^-{\alpha_{u,v}1_{r_t}}
        \ar@2@/_/[dr]_-{1_{r_u}\alpha_{v,t}}
        \\
        & r_u r_c r_d
        \ar@2[r]_-{\alpha_{u,c}1_{r_d}}
        & r_i r_j r_d
}
}
\end{minipage}

\medskip

\noindent
\begin{minipage}[t]{\textwidth}
\centering
\textbf{Case $4b4$.}

\caseconfluence{
\xymatrix @R=0.75em @C=1.8em @!C {
        & r_a r_b r_t
        \ar@2[r]^-{1_{r_a}\alpha_{b,t}}
        \ar@3[]!<16pt,-8pt>;[dd]!<16pt,8pt>
          ^-*+{\mathcal X^{(4b4)}_{u,v,t}}
        & r_a r_e r_f
        \ar@2@/^/[dr]^-{\alpha_{a,e}1_{r_f}}
        \\
r_u r_v r_t
        \ar@2@/^/[ur]^-{\alpha_{u,v}1_{r_t}}
        \ar@2@/_/[dr]_-{1_{r_u}\alpha_{v,t}}
        &&& r_i r_s r_f
        \\
        & r_u r_c r_d
        \ar@2[r]_-{\alpha_{u,c}1_{r_d}}
        & r_i r_j r_d
        \ar@2@/_/[ur]_-{1_{r_i}\alpha_{j,d}}
}
}
\end{minipage}

\subsubsection{The coherent row presentation}

We denote by~$\Sigma^{\mathrm{row,coh}}(n)=\bigl(\Sigma^{\mathrm{row}}(n),\Sigma^{\mathrm{row}}_3(n)\bigr)$ the $(3,1)$-polygraph obtained from $\Sigma^{\mathrm{row}}(n)$ by adjoining,
for each critical triple $(u,v,t)\in\Rr_n^3$, one generating $3$-cell
$\mathcal{X}_{u,v,t}$ between the two confluent rewriting paths constructed
in the preceding case analysis.

\begin{theorem}
\label{T:CoherentPresentation}
The $(3,1)$-polygraph $\Sigma^{\mathrm{row,coh}}(n)$ is a finite coherent
presentation of the stylic monoid~$\Styl_n$.
\end{theorem}

\begin{proof}
By Proposition~\ref{P:RowPresentationConvergent}, the $2$-polygraph~$\Sigma^{\mathrm{row}}(n)$ is finite and convergent, and it presents~$\Styl_n$. 
Since~$\Rr_n$ is finite, there are only finitely many critical
triples $(u,v,t)\in\Rr_n^3$. 
By Lemma~\ref{L:CriticalBranchingsRow}, the
critical branchings of~$\Sigma^{\mathrm{row}}(n)$ are exactly those indexed by
critical triples. The preceding case analysis provides, for each such critical
branching, a generating $3$-cell~$\mathcal X_{u,v,t}$ between chosen confluent
rewriting paths. 
Hence, by Squier's coherence theorem,
$\Sigma^{\mathrm{row,coh}}(n)$ is a finite coherent presentation of~$\Styl_n$.
\end{proof}

\subsubsection{Number of cells in the coherent presentation}
The following proposition gives formulas for the
numbers of generating cells in the coherent row presentation~$\Sigma^{\mathrm{row,coh}}(n)$.

\begin{proposition}
\label{Prop:CountingCells}
For every $n\geq1$, the numbers of generating cells of
$\Sigma^{\mathrm{row,coh}}(n)$ are
\begin{align}
\bigl|\Sigma^{\mathrm{row}}_1(n)\bigr|
&=
2^n-1,
\label{eq:CountRowCellsOne}
\\
\bigl|\Sigma^{\mathrm{row}}_2(n)\bigr|
&=
\frac{2^{2n+1}-2^{n+1}-3^n+1}{2},
\label{eq:CountRowCellsTwo}
\\
\bigl|\Sigma^{\mathrm{row}}_3(n)\bigr|
&=
\frac{
6\cdot8^n-6^{n+1}-5\cdot4^n+3^{n+1}+3\cdot2^n-1
}{6}.
\label{eq:CountRowCellsThree}
\end{align}
\end{proposition}

\begin{proof}
A generating $1$-cell is indexed by a nonempty row on $[n]$, equivalently by a nonempty
subset of~$[n]$. This proves~\eqref{eq:CountRowCellsOne}.

A generating $2$-cell is a cell $\alpha_{u,v}$ for an ordered pair $(u,v)\in \Rr_n^2$
such that $u\not\triangleleft v$. Therefore
\[
|\Sigma^{\mathrm{row}}_2(n)|
=
(2^n-1)^2-\#\{(u,v)\in \Rr_n^2\mid u\triangleleft v\}.
\]
For a fixed row $v$ of length $k$, the rows $u$ such that $u\triangleleft v$ are exactly
the nonempty subsets of $v\setminus\{\min(v)\}$, so there are $2^{k-1}-1$ of them.
Summing over all rows $v$ gives
\[
\#\{(u,v)\in \Rr_n^2\mid u\triangleleft v\}
=
\sum_{k=1}^n \binom{n}{k}(2^{k-1}-1)
=
\frac{3^n-2^{n+1}+1}{2}.
\]
Substituting this value proves~\eqref{eq:CountRowCellsTwo}.

By construction and by Lemma~\ref{L:CriticalBranchingsRow}, the generating
$3$-cells are indexed by triples $(u,v,t)\in \Rr_n^3$ satisfying $u\not\triangleleft v$ and $v\not\triangleleft t$.
Fix a middle row $v$ of length $k$ and minimum $m$. Then the number of rows $u$ such that
$u\not\triangleleft v$ is~$2^n-2^{k-1}$,
whereas the number of rows $t$ such that $v\not\triangleleft t$ is
\[
(2^n-1)-(2^{m-1}-1)2^{\,n-k-m+1}.
\]
Indeed, the condition $v\triangleleft t$ means that $\supp(v)\subseteq \supp(t)$ and
$\min(t)<\min(v)=m$. Since $v$ has length~$k$ and minimum $m$, the row $t$
is obtained by adjoining to $v$ a nonempty subset of $\{1,\dots,m-1\}$ and
an arbitrary subset of the $n-k-m+1$ elements of $\{m+1,\dots,n\}\setminus v$.
Hence the number of such rows $t$ is~$\displaystyle (2^{m-1}-1)2^{\,n-k-m+1}$.
Since the number of rows~$v$ of length~$k$ and~$\min(v)=m$ is~$\displaystyle \binom{n-m}{k-1}$,
it follows that
\[
|\Sigma^{\mathrm{row}}_3(n)|
=
\sum_{m=1}^n \sum_{k=1}^{n-m+1}
\binom{n-m}{k-1}
\bigl(2^n-2^{k-1}\bigr)
\Bigl((2^n-1)-(2^{m-1}-1)2^{\,n-k-m+1}\Bigr).
\]
Setting~$r=n-m$ and~$j=k-1$, this becomes
\[
|\Sigma^{\mathrm{row}}_3(n)|
=
\sum_{r=0}^{n-1}\sum_{j=0}^r
\binom{r}{j}
\bigl(2^n-2^j\bigr)
\Bigl((2^n-1)-(2^{n-r-1}-1)2^{\,r-j}\Bigr).
\]
Using the binomial identities
\[
\sum_{j=0}^r \binom{r}{j}=2^r,
\qquad
\sum_{j=0}^r \binom{r}{j}2^j=3^r,
\qquad
\sum_{j=0}^r \binom{r}{j}2^{-j}=\left(\frac32\right)^r,
\]
one obtains
\[
|\Sigma^{\mathrm{row}}_3(n)|
=
\sum_{r=0}^{n-1}
\left(
2^{2n+r}-2^{n+r-1}-2^{2r}+3^r-2^{2n-r-1}3^r
\right).
\]
Summing these geometric series gives~\eqref{eq:CountRowCellsThree}.
\end{proof}

\begin{example}
For small values of $n$, the numbers of generating cells are given in the following table.
\begin{center}
\begin{tabular}{|c|c|c|c|}
\hline
$n$ & $|\Sigma^{\mathrm{row}}_1(n)|$ & $|\Sigma^{\mathrm{row}}_2(n)|$ & $|\Sigma^{\mathrm{row}}_3(n)|$ \\
\hline
$1$ & $1$   & $1$    & $1$ \\
$2$ & $3$   & $8$    & $21$ \\
$3$ & $7$   & $43$   & $260$ \\
$4$ & $15$  & $200$  & $2635$ \\
$5$ & $31$  & $871$  & $24276$ \\
$6$ & $63$  & $3668$ & $212471$ \\
\hline
\end{tabular}
\end{center}
\end{example}

\section{A reduced coherent presentation of the stylic monoid}
\label{S:ReducedCoherentPresentationStylicMonoid}

Starting from the coherent row presentation constructed in
Section~\ref{S:CoherentPresentationStylicMonoid}, we perform two
homotopical reductions. The first removes the splittable generating
confluences, and the second eliminates a deterministic acyclic family of
row $2$-cells together with corresponding generating $3$-cells. We then
transport the reduced presentation to the pre-row presentation and eliminate
the non-letter row generators, obtaining a finite coherent presentation on
the letter generators.

\subsection{Homotopical reduction procedure}
\label{SS:HomotopicalReductionProcedure}

We recall the homotopical reduction procedure
from~\cite{GaussentGuiraudMalbos15}. Let $\Sigma$ be a
$(3,1)$-polygraph. A \emph{$3$-sphere} of $\tck{\Sigma}_3$ is a pair
$(f,g)$ of parallel $3$-cells of $\tck{\Sigma}_3$, that is,
$s_2(f)=s_2(g)$ and~$t_2(f)=t_2(g)$.
A \emph{collapsible part} of $\Sigma$ is a triple
\[
\Gamma=(\Gamma_2,\Gamma_3,\Gamma_4),
\]
where $\Gamma_2$ is a family of generating $2$-cells of $\Sigma$,
$\Gamma_3$ is a family of generating $3$-cells of $\Sigma$, and
$\Gamma_4$ is a family of $3$-spheres of $\tck{\Sigma}_3$, regarded as
collapsible $4$-cells.
For $k\in\{2,3,4\}$, a cell $\gamma\in\Gamma_k$ is
\emph{collapsible} if $t_{k-1}(\gamma)$ is a generating
$(k-1)$-cell of $\Sigma$ and $s_{k-1}(\gamma)$ belongs to the free
$(3,1)$-category generated by
$\Sigma\setminus\{t_{k-1}(\gamma)\}$.
The target $t_{k-1}(\gamma)$ is then called the
\emph{redundant cell} associated with~$\gamma$.
When $k=4$, an element $\gamma\in\Gamma_4$ is a $3$-sphere~$
\gamma
=
\bigl(
s_3(\gamma),
t_3(\gamma)
\bigr)$.
In this case, $t_3(\gamma)$ is a generating $3$-cell of $\Sigma$, while
$s_3(\gamma)$ does not contain this generating cell.
These data satisfy the following conditions:
\begin{enumerate}[\bf i)]
\item every $\gamma\in\Gamma_k$, for $k\in\{2,3,4\}$, is
collapsible, possibly after a Nielsen transformation as recalled below;
\item no element of $\Gamma_k$ is the redundant target of an element
of $\Gamma_{k+1}$, for $k=2,3$;
\item there exist well-founded orders on the generating $1$-cells,
$2$-cells, and $3$-cells of $\Sigma$ such that, for every
$\gamma\in\Gamma_k$, the target $t_{k-1}(\gamma)$ is strictly greater
than every generating $(k-1)$-cell occurring in~$s_{k-1}(\gamma)$.
\end{enumerate}

The \emph{homotopical reduction associated with~$\Gamma$} is defined on
generating cells by
\[
R_\Gamma(x)
=
\begin{cases}
R_\Gamma\bigl(s_{k-1}(\gamma)\bigr),
&
\begin{array}{l}
x=t_{k-1}(\gamma)
\text{ for some }\gamma\in\Gamma_k,
k\in\{2,3,4\},
\end{array}
\\
1_{R_\Gamma(s_{k-1}(\gamma))},
&
\begin{array}{l}
x=\gamma\in\Gamma_k
\text{ for some }k\in\{2,3\},
\end{array}
\\
x,
&
\text{otherwise}.
\end{cases}
\]
It is extended compatibly to identities, inverses, and compositions.
The well-founded orders above ensure that this recursive definition is
well defined. It gives a Tietze transformation
\[
R_\Gamma\colon
\tck{\Sigma}_3
\longrightarrow
\tck{R_\Gamma(\Sigma)}_3.
\]
The polygraph $R_\Gamma(\Sigma)$ is obtained from $\Sigma$ by removing
the cells in $\Gamma_2\cup\Gamma_3$ and the redundant targets~$t_{k-1}(\gamma)$ associated with the cells
$\gamma\in\Gamma_k$, for $k\in\{2,3,4\}$. An element of $\Gamma_4$
is not itself a generator of $\Sigma$: it is the $3$-sphere that
justifies the elimination of its target generating $3$-cell.
The source and target maps of the remaining generators are replaced by
their compositions with $R_\Gamma$. Thus a surviving generator is kept,
but its boundary is reduced whenever it contains an eliminated cell.
The resulting polygraph is Tietze equivalent to $\Sigma$. We refer
to~\cite[Section~2.3.1]{GaussentGuiraudMalbos15} for the detailed
construction.

The first reduction uses critical triple branchings to construct
collapsible $3$-spheres. A \emph{local triple branching} is a triple
$(f,g,h)$ of rewriting steps with a common source. It is
\emph{aspherical} if two of the three steps coincide, and
\emph{Peiffer} if at least one of the three steps forms a Peiffer
branching with each of the other two. All other local triple branchings
are called \emph{overlapping}.
Local triple branchings are ordered by inclusion of their sources: a
branching is below another if the latter is obtained from it by common
whiskering. A minimal overlapping triple branching is called
\emph{critical}. For a coherent presentation with convergent underlying
$2$-polygraph, confluence and coherence complete every critical triple
branching to a $3$-sphere
\cite[Section~2.3.2]{GaussentGuiraudMalbos15}. 

For later use, let~$f\colon u\dfl u'$ and~$g\colon v\dfl v'$ be
rewriting steps acting on disjoint factors. The two paths~$
(f1_v)\star_1(1_{u'}g)$ and~$
(1_ug)\star_1(f1_{v'})$ are equal by the exchange law. We call the corresponding identity
$3$-cell the \emph{Peiffer comparison} between the two disjoint steps.

Finally, we use Nielsen transformations in dimension~$3$. Let~$A\colon f\Rrightarrow g$
be a generating $3$-cell, and suppose that, for some
$\varepsilon\in\{+1,-1\}$, a $3$-cell expression~$B$ has the form~$B
=
1_h\star_1A^\varepsilon\star_1 1_k$,
where $h$ and~$k$ are composable $2$-cells of $\tck{\Sigma}_3$.
Since these $2$-cells are invertible, $A^\varepsilon$ can be recovered
from~$B$ by whiskering with $h^{-1}$ and~$k^{-1}$. Then $A$ may be
replaced by a generating $3$-cell having the boundary of~$B$.
Indeed, one first adjoins a new generating $3$-cell
$\widetilde A$ parallel to~$B$, together with the collapsible
$3$-sphere~$(B,\widetilde A)$.
Using the inverse $2$-cells, $A^\varepsilon$ can then be expressed in
terms of~$\widetilde A$, so that $A$ becomes redundant and may be
eliminated. This composite of elementary coherent Tietze
transformations is called a
\emph{Nielsen transformation in dimension~$3$}
\cite[Section~2.1.4]{GaussentGuiraudMalbos15}.

\subsection{First homotopical reduction: splittable critical confluences}
\label{SS:FirstReductionSplittable}

We now construct a first homotopical reduction of the coherent row
presentation.  A splittable critical triple is
expanded to a critical triple branching on four adjacent row generators.
Five faces are filled by Squier confluences and Peiffer comparisons, while
the sixth face is a whiskered copy of the generating confluence to be
eliminated.

For the row presentation $\Sigma^{\mathrm{row}}(n)$, the leftmost
normalization paths satisfy the prefix factorization
\[
\npath{uv}
=
(\npath{u}1_v)\star_1\npath{\widehat u v}
\]
for all row-generator words $u$ and~$v$. Indeed, the leftmost strategy first
normalizes the prefix~$u$ and then continues from~$\widehat u v$.
Recall that, for every critical triple $(u,v,t)$, we have~$\operatorname{NF}(u,v,t)
=
\Rr(N(uvt))
=
\widehat{r_ur_vr_t}$.
The corresponding generating Squier confluence is
\[
\mathcal X_{u,v,t}\colon
(\alpha_{u,v}1_{r_t})\star_1
\npath{\Rr(N(uv))r_t}
\Rrightarrow
(1_{r_u}\alpha_{v,t})\star_1
\npath{r_u\Rr(N(vt))}.
\]
The case analysis of
Subsubsection~\ref{SSS:ExplicitGeneratingConfluences} shows that these are
exactly the generating confluences fixed there. 
For a critical triple $(u,v,t)$ and a row-generator word $w$, set
\[
\widehat{\mathcal X}^{\,w}_{u,v,t}
:=
(\mathcal X_{u,v,t}1_w)\star_1
1_{\npath{\operatorname{NF}(u,v,t)w}}.
\]
Thus $\widehat{\mathcal X}^{\,w}_{u,v,t}$ is obtained by placing
$\mathcal X_{u,v,t}$ in the right context $w$ and continuing both boundary
paths by the same leftmost normalization path. When $w$ is empty, we write
$\widehat{\mathcal X}_{u,v,t}=\mathcal X_{u,v,t}$. We use the inverse when
the opposite orientation is required.

\subsubsection{Splittable critical triples}

Consider
\[
\mathfrak C_n
=
\{(u,v,t)\in\Rr_n^3
\mid
u\not\triangleleft v,\ v\not\triangleleft t\}
\]
the set of critical triples. If $u=x_1x_2\cdots x_k$, with $k\ge2$, is a
row, set $\partial u:=x_2\cdots x_k$. A critical triple~$(u,v,t)$ is called \emph{splittable} if~$\ell(u)\ge2$ and~$\partial u\not\triangleleft v$.
We denote the set of splittable critical triples by~$\mathfrak S_n$.
A critical triple which is not splittable is called
\emph{nonsplittable}. The same terminology is used for its generating
confluence. A generating confluence~$\mathcal X_{u,v,t}$ is called
\emph{letter-first} if~$\ell(u)=1$.

\begin{lemma}
\label{L:SplittableAutomaticallyCriticalFinal}
Let $u=x\,\partial u$ be a row of length at least two and let $v$ be a row.
If $\partial u\not\triangleleft v$, then $u\not\triangleleft v$.
\end{lemma}

\begin{proof}
If $u\triangleleft v$, then $\supp(u)\subseteq\supp(v)$ and
$\min(v)<\min(u)=x$. Hence $\supp(\partial u)\subseteq\supp(v)$ and
$\min(v)<x<\min(\partial u)$, so $\partial u\triangleleft v$, a
contradiction.
\end{proof}

Let $\tau=(u,v,t)\in\mathfrak S_n$ and write $u=x\,\partial u$. Then~$W_\tau:=r_xr_{\partial u}r_vr_t$
admits the three rewriting steps~$
A_\tau:=\alpha_{x,\partial u}1_{r_v}1_{r_t}$,~$
B_\tau:=1_{r_x}\alpha_{\partial u,v}1_{r_t}$ and~$
C_\tau:=1_{r_x}1_{r_{\partial u}}\alpha_{v,t}$.
Indeed, $\alpha_{x,\partial u}$ exists because
$x\notin\supp(\partial u)$, and hence
$x\not\triangleleft\partial u$.
The $2$-cell $\alpha_{\partial u,v}$ exists by splittability, while
$\alpha_{v,t}$ exists by criticality.
Moreover, $(A_\tau,B_\tau,C_\tau)$ is a critical triple branching.
Indeed, $A_\tau$ overlaps with~$B_\tau$, $B_\tau$ overlaps
with~$C_\tau$, and~$(A_\tau,C_\tau)$ is a Peiffer branching. Any source
containing all three steps must contain the four consecutive row generators
$r_xr_{\partial u}r_vr_t$. Hence the triple branching has no proper
overlapping subbranching and is minimal.

\begin{remark}
The splittability condition is essential. In general,
$u\not\triangleleft v$ does not imply
$\partial u\not\triangleleft v$, so splitting $u=x\,\partial u$ need not
produce the middle rewriting step $\alpha_{\partial u,v}$.
\end{remark}

\subsubsection{The two transport faces}

Fix $\tau=(u,v,t)$ in~$\mathfrak S_n$ and write
$u=x\,\partial u$. Set
\[
\mathbf M:=\Rr(N(uv)),
\qquad
\mathbf U:=\Rr(N(\partial u\,v)),
\qquad
\mathbf V:=\Rr(N(vt)),
\]
and
\[
\mathbf Z:=\operatorname{NF}(\partial u,v,t),
\qquad
\mathbf N_\tau:=\operatorname{NF}(u,v,t).
\]
Replacing a factor by its row normal form does not change the represented
$N$-tableau. Hence, by convergence, the following leftmost normalization
paths are well defined:
\[
\begin{alignedat}{3}
p_\tau
&:=
\npath{r_x\mathbf U}
\colon r_x\mathbf U\dfl^*\mathbf M,
\qquad&
f_\tau^L
&:=
\npath{\mathbf U r_t}
\colon \mathbf U r_t\dfl^*\mathbf Z,
\qquad&
q_\tau
&:=
\npath{r_x\mathbf Z}
\colon r_x\mathbf Z\dfl^*\mathbf N_\tau,
\\[1mm]
h_\tau
&:=
\npath{\mathbf M r_t}
\colon \mathbf M r_t\dfl^*\mathbf N_\tau,
&
f_\tau^R
&:=
\npath{r_{\partial u}\mathbf V}
\colon r_{\partial u}\mathbf V\dfl^*\mathbf Z,
&
\rho_\tau
&:=
\npath{r_u\mathbf V}
\colon r_u\mathbf V\dfl^*\mathbf N_\tau.
\end{alignedat}
\]
The next lemma constructs the two transport faces.

\begin{lemma}
\label{L:NonsplittableTransportFacesFinal}
Let $\tau=(u,v,t)\in\mathfrak S_n$ and write
$u=x\,\partial u$. With the notation above, the following two
$2$-spheres admit fillings $\Theta_\tau^L$ and $\Theta_\tau^R$:
\begin{equation}
\label{eq:TransportFacesSplittableFinal}
\xymatrix @R=0.8em @C=1em @!C {
        & \mathbf M r_t
          \ar@2@/^/[dr]^-{h_\tau}
          \ar@3[]!<0pt,-10pt>;[dd]!<0pt,10pt>
          ^-*+{\Theta_\tau^L}
        \\
r_x\mathbf U r_t
          \ar@2@/^/[ur]^-{p_\tau1_{r_t}}
          \ar@2@/_/[dr]_-{1_{r_x}f_\tau^L}
        && \mathbf N_\tau
        \\
        & r_x\mathbf Z
          \ar@2@/_/[ur]_-{q_\tau}
}
\qquad
\xymatrix @R=0.8em @C=1em @!C {
        & r_u\mathbf V
          \ar@2@/^/[dr]^-{\rho_\tau}
          \ar@3[]!<0pt,-10pt>;[dd]!<0pt,10pt>
          ^-*+{\Theta_\tau^R}
        \\
r_xr_{\partial u}\mathbf V
          \ar@2@/^/[ur]^-{
            \alpha_{x,\partial u}1_{\mathbf V}}
          \ar@2@/_/[dr]_-{1_{r_x}f_\tau^R}
        && \mathbf N_\tau
        \\
        & r_x\mathbf Z
          \ar@2@/_/[ur]_-{q_\tau}
}
\end{equation}
Both fillings are oriented from the upper path to the lower path.
The upper paths are, respectively,
$\npath{r_x\mathbf U r_t}$ and
$\npath{r_xr_{\partial u}\mathbf V}$.
Moreover, the fillings can be constructed from Peiffer comparisons
and nonsplittable Squier confluences.
\end{lemma}

\begin{proof}
The identification of the upper paths follows from prefix factorization
on the left and from the fact that $(x,\partial u)$ is the leftmost
reducible adjacent pair on the right.
Since $\tau$ is splittable, $(\partial u,v,t)$ is a critical triple.
Proposition~\ref{Prop:boundedBranching} and
Lemma~\ref{L:TwoRowsProductRows} therefore show that the cases considered
below are exhaustive.
In the pasting diagrams, common whiskering contexts and identity
$3$-cells on common terminal normalization paths are omitted.

\noindent
\textbf{Left transport.}
Suppose first that $\mathbf U=r_b$. Since $N(\partial u\,v)$ has one
row, Lemma~\ref{L:OneRowCriterion} gives~$b=\partial u\cup v$ and~$\min(v)\geq\max(\partial u)$.
Since $x<\min(\partial u)$, one has $x\notin\supp(b)$, and hence
$x\not\triangleleft b$. Moreover, by
Lemma~\ref{L:UsefulReduciblePairsCriticalTriple},
$b\not\triangleleft t$. Thus $(x,b,t)$ is a critical triple,
$f_\tau^L=\alpha_{b,t}$, and the required transport face is~$\Theta_\tau^L=\mathcal X_{x,b,t}$.
Suppose now that $\mathbf U=r_ar_b$, with $a\triangleleft b$. By
Lemma~\ref{L:TopRowInclusion},
$\supp(a)\subseteq\supp(\partial u)$, and hence $x<\min(a)$. Therefore
Lemma~\ref{L:OneRowCriterion} gives~$\Rr(N(xa))=r_{xa}$ and~
$\partial(xa)=a$.
Moreover, $b\not\triangleleft t$. Put~$\mathbf E:=\Rr(N(bt))$.
If $\mathbf E=r_f$, then~$\supp(a)\subseteq\supp(b)\subseteq\supp(f)$ and~$\min(f)\leq\min(b)<\min(a)$,
so $a\triangleleft f$. Thus no second step occurs. If
$\mathbf E=r_er_f$, a second step occurs exactly when
$a\not\triangleleft e$. In that case it is~$\alpha_{a,e}1_{r_f}$, and $(x,a,e)$ is a critical triple.
If $xa\triangleleft b$, define~$K_\tau^L$ to be the corresponding
identity $3$-cell. Otherwise, $(xa,b,t)$ is a critical triple, and we
set~$K_\tau^L:=\mathcal X_{xa,b,t}$. This confluence is nonsplittable
because~$\partial(xa)=a\triangleleft b$.
Let $\Pi_\tau^L$ denote the Peiffer comparison exchanging the independent
steps on $(x,a)$ and $(b,t)$.

The possible left transport pastings are listed in the following table:
\[
\begin{array}{|c|c|l|}
\hline
\mathbf U
& \text{condition}
& \text{faces pasted, in order}
\\ \hline
r_b
& \text{always}
& \mathcal X_{x,b,t}
\\
r_ar_b
& f_\tau^L\text{ has no second step}
& K_\tau^L,\ \Pi_\tau^L
\\
r_ar_b
& f_\tau^L\text{ has the second step }
  \alpha_{a,e}1_{r_f}
& K_\tau^L,\ \Pi_\tau^L,\
  \widehat{\mathcal X}^{\,r_f}_{x,a,e}
\\ \hline
\end{array}
\]
For $\mathbf U=r_ar_b$, the maximal case occurs when
$f_\tau^L$ has the second step
$\alpha_{a,e}1_{r_f}$ and $xa\not\triangleleft b$. Its pasting diagram is
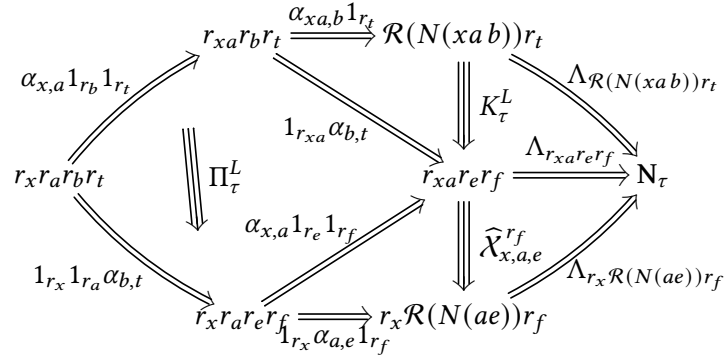
\[
\xymatrix @R=3.3em @C=2.5em {
&
r_{xa}r_br_t
  \ar@2[r]^-{\alpha_{xa,b}1_{r_t}}
  \ar@2[dr]_-{1_{r_{xa}}\alpha_{b,t}}
&
\Rr(N(xa\,b))r_t
  \ar@2@/^/[dr]^-{\npath{\Rr(N(xa\,b))r_t}}
  \ar@3[]!<0pt,-10pt>;[d]!<0pt,10pt>
    ^-*+{K_\tau^L}
&
\\
r_xr_ar_br_t
  \ar@2@/^/[ur]^-{\alpha_{x,a}1_{r_b}1_{r_t}}
    _-{}="pi-upper"
  \ar@2@/_/[dr]_-{1_{r_x}1_{r_a}\alpha_{b,t}}
    ^-{}="pi-lower"
  \ar@3 "pi-upper"!<22pt,-10pt>;
        "pi-lower"!<22pt,10pt>
    ^-*+{\Pi_\tau^L}
&&
r_{xa}r_er_f
  \ar@2[r]^-{\npath{r_{xa}r_er_f}}
  \ar@3[]!<0pt,-10pt>;[d]!<0pt,10pt>
    ^-*+{\widehat{\mathcal X}^{\,r_f}_{x,a,e}}
&
\mathbf N_\tau
\\
&
r_xr_ar_er_f
  \ar@2[ur]^-{\alpha_{x,a}1_{r_e}1_{r_f}}
  \ar@2[r]_-{1_{r_x}\alpha_{a,e}1_{r_f}}
&
r_x\Rr(N(ae))r_f
  \ar@2@/_/[ur]_-{\npath{r_x\Rr(N(ae))r_f}}
&
}
\]
The other cases are obtained as indicated in the table:
$\widehat{\mathcal X}^{\,r_f}_{x,a,e}$ is omitted when
$f_\tau^L$ has no second step, and the identity face $K_\tau^L$ is
contracted when $xa\triangleleft b$. The case $\mathbf U=r_b$ consists
of the single face $\mathcal X_{x,b,t}$. In every case, prefix
factorization identifies the lower boundary with~$(1_{r_x}f_\tau^L)\star_1q_\tau$.
Thus the resulting pasting defines the required $3$-cell
\[
\Theta_\tau^L\colon
\npath{r_x\mathbf U r_t}
\Rrightarrow
(1_{r_x}f_\tau^L)\star_1q_\tau.
\]

\noindent
\textbf{Right transport.}
Suppose first that $\mathbf V=r_d$. By
Proposition~\ref{Prop:boundedBranching}, $f_\tau^R$ is either the identity
or the single step $\alpha_{\partial u,d}$. In the first case the two
transport paths coincide. 
In the second case, $(x,\partial u,d)$ is a critical triple, and the
required face is~$\mathcal X_{x,\partial u,d}$, which is letter-first.
Suppose now that $\mathbf V=r_cr_d$, with $c\triangleleft d$. By
Lemma~\ref{L:UsefulReduciblePairsCriticalTriple},
$\partial u\not\triangleleft c$, so $f_\tau^R$ begins with
$\alpha_{\partial u,c}1_{r_d}$. Hence $(x,\partial u,c)$ is a critical triple.  Put~$\mathbf J:=\Rr(N(\partial u\,c))$.
By prefix factorization, the whiskered confluence~$K_{0,\tau}^R
:=
\widehat{\mathcal X}^{\,r_d}_{x,\partial u,c}$
has boundary
\[
K_{0,\tau}^R\colon
\npath{r_xr_{\partial u}r_cr_d}
\Rrightarrow
(1_{r_x}\alpha_{\partial u,c}1_{r_d})
\star_1
\npath{r_x\mathbf J r_d}.
\]
The generating confluence occurring in $K_{0,\tau}^R$ is letter-first.
If $\mathbf J=r_g$, then Lemma~\ref{L:OneRowCriterion} gives~$g=\partial u\cup c$ and~$\min(c)\geq\max(\partial u)$.
Since $x<\min(\partial u)$, one has $x\notin\supp(g)$, and therefore
$x\not\triangleleft g$.
If $f_\tau^R$ has no second step, then $g\triangleleft d$, and
$K_{0,\tau}^R$ already gives the required transport. If $f_\tau^R$ has
the second step~$\alpha_{g,d}$, then $(x,g,d)$ is critical, and the
required transport is obtained by pasting $K_{0,\tau}^R$ with the
letter-first confluence~$\mathcal X_{x,g,d}$.
Finally, suppose that $\mathbf J=r_ir_j$, with $i\triangleleft j$. By
Lemma~\ref{L:TopRowInclusion},
$\supp(i)\subseteq\supp(\partial u)$, hence~$x<\min(i)$. Thus~$
\Rr(N(xi))=r_{xi}$ and~$\partial(xi)=i$.
If $f_\tau^R$ has no second step, then $j\triangleleft d$, and
$K_{0,\tau}^R$ already gives the required transport.
Suppose now that $f_\tau^R$ has a second step. Then
$j\not\triangleleft d$, and this step is
$1_{r_i}\alpha_{j,d}$. Put~$\mathbf H:=\Rr(N(jd))$
and define~$
K_{1,\tau}^R\colon
\npath{r_{xi}r_jr_d}
\Rrightarrow
(1_{r_{xi}}\alpha_{j,d})\star_1
\npath{r_{xi}\mathbf H}$
to be the identity if~$xi\triangleleft j$, and
$\mathcal X_{xi,j,d}$ otherwise. In the latter case $(xi,j,d)$ is
critical, and the confluence is nonsplittable because~$\partial(xi)=i\triangleleft j$.
Let $\Pi_\tau^R$ denote the Peiffer comparison exchanging the independent
steps on $(x,i)$ and $(j,d)$.

The possible right transport pastings are listed in the following table:
\[
\begin{array}{|c|c|l|}
\hline
\mathbf V
& \text{condition}
& \text{faces pasted, in order}
\\ \hline
r_d
& f_\tau^R=1
& \mathrm{id}
\\
r_d
& f_\tau^R=\alpha_{\partial u,d}
& \mathcal X_{x,\partial u,d}
\\
r_cr_d,\ \mathbf J=r_g
& f_\tau^R\text{ has no second step}
& K_{0,\tau}^R
\\
r_cr_d,\ \mathbf J=r_g
& f_\tau^R\text{ has the second step }\alpha_{g,d}
& K_{0,\tau}^R,\ \mathcal X_{x,g,d}
\\
r_cr_d,\ \mathbf J=r_ir_j
& f_\tau^R\text{ has no second step}
& K_{0,\tau}^R
\\
r_cr_d,\ \mathbf J=r_ir_j
& f_\tau^R\text{ has the second step }
  1_{r_i}\alpha_{j,d}
& K_{0,\tau}^R,\ K_{1,\tau}^R,\ \Pi_\tau^R
\\ \hline
\end{array}
\]
Suppose that $\mathbf V=r_cr_d$ and $\mathbf J=r_g$. The maximal case
occurs when $f_\tau^R$ has the second step $\alpha_{g,d}$, and its
pasting diagram is
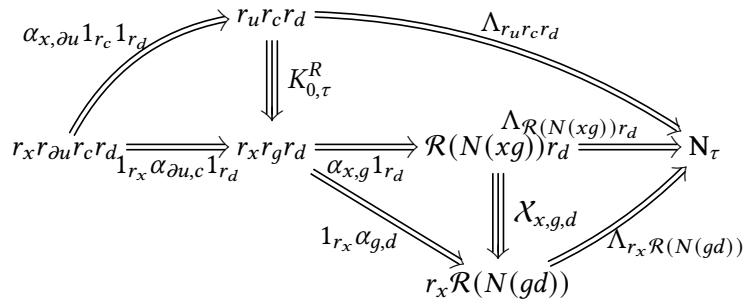
\[
\xymatrix @R=3.2em @C=3.5em {
&
r_ur_cr_d
  \ar@2@/^3ex/[drr]^-{\npath{r_ur_cr_d}}
  \ar@3[]!<0pt,-10pt>;[d]!<0pt,10pt>
    ^-*+{K_{0,\tau}^R}
&&
\\
r_xr_{\partial u}r_cr_d
  \ar@2@/^3ex/[ur]^-{
    \alpha_{x,\partial u}1_{r_c}1_{r_d}}
  \ar@2[r]_-{
    1_{r_x}\alpha_{\partial u,c}1_{r_d}}
&
r_xr_gr_d
  \ar@2[r]_-{\alpha_{x,g}1_{r_d}}
  \ar@2[dr]_-{1_{r_x}\alpha_{g,d}}
&
\Rr(N(xg))r_d
  \ar@2[r]^(.35){
    \npath{\Rr(N(xg))r_d}}
  \ar@3[]!<0pt,-10pt>;[d]!<0pt,10pt>
    ^-*+{\mathcal X_{x,g,d}}
&
\mathbf N_\tau
\\
&&
r_x\Rr(N(gd))
  \ar@2@/_/[ur]_-{\npath{r_x\Rr(N(gd))}}
&
}
\]
If $f_\tau^R$ has no second step, only $K_{0,\tau}^R$ remains:
the $\mathcal X_{x,g,d}$-face and its lower boundary are omitted,
whereas its upper boundary remains part of the target path of
$K_{0,\tau}^R$.
Suppose now that $\mathbf V=r_cr_d$ and $\mathbf J=r_ir_j$. The maximal
case occurs when $f_\tau^R$ has the second step
$1_{r_i}\alpha_{j,d}$ and $xi\not\triangleleft j$. Its pasting diagram is
\[
\xymatrix @R=3.3em @C=3.7em {
r_xr_{\partial u}r_cr_d
  \ar@2[r]^-{
    \alpha_{x,\partial u}1_{r_c}1_{r_d}}
  \ar@2[dr]_-{
    1_{r_x}\alpha_{\partial u,c}1_{r_d}}
&
r_ur_cr_d
  \ar@2@/^5ex/[drrr]^(.25){
    \npath{r_ur_cr_d}}
  \ar@3[]!<0pt,-10pt>;[d]!<0pt,10pt>
    ^-*+{K_{0,\tau}^R}
&&&
\\
&
r_xr_ir_jr_d
  \ar@2[r]^-{
    \alpha_{x,i}1_{r_j}1_{r_d}}
  \ar@2[dr]_-{
    1_{r_x}1_{r_i}\alpha_{j,d}}
&
r_{xi}r_jr_d
  \ar@2[r]^-{
    \alpha_{xi,j}1_{r_d}}
  \ar@2[dr]^-{
    1_{r_{xi}}\alpha_{j,d}}
  \ar@3[]!<0pt,-10pt>;[d]!<0pt,10pt>
    ^-*+{\Pi_\tau^R}
&
\Rr(N(xi\,j))r_d
  \ar@2[r]^(.35){
    \npath{\Rr(N(xi\,j))r_d}}
  \ar@3[]!<0pt,-10pt>;[d]!<0pt,10pt>
    ^-*+{K_{1,\tau}^R}
&
\mathbf N_\tau
\\
&&
r_xr_i\mathbf H
  \ar@2@/_/[r]_-{
    \alpha_{x,i}1_{\mathbf H}}
&
r_{xi}\mathbf H
  \ar@2@/_/[ur]_-{
    \npath{r_{xi}\mathbf H}}
&
}
\]
If $f_\tau^R$ has no second step, only $K_{0,\tau}^R$ remains. If the
second step occurs but $xi\triangleleft j$, the identity face
$K_{1,\tau}^R$ is contracted, while the Peiffer face
$\Pi_\tau^R$ remains. The cases $\mathbf V=r_d$ are given by the first
two lines of the table.
In every case, prefix factorization identifies the lower boundary with~$(1_{r_x}f_\tau^R)\star_1q_\tau$.
Thus the resulting pasting defines the required $3$-cell
\[
\Theta_\tau^R\colon
\npath{r_xr_{\partial u}\mathbf V}
\Rrightarrow
(1_{r_x}f_\tau^R)\star_1q_\tau.
\]

The only Squier confluences used in the two transport fillings that are
not letter-first are
$\mathcal X_{xa,b,t}$ and $\mathcal X_{xi,j,d}$. They are
nonsplittable because~$\partial(xa)=a\triangleleft b$ and~$
\partial(xi)=i\triangleleft j$.
Hence every generating confluence occurring in
$\Theta_\tau^L$ or $\Theta_\tau^R$ is nonsplittable.
\end{proof}

\subsubsection{Elimination of splittable confluences}

Fix $\tau=(u,v,t)\in\mathfrak S_n$ and write $u=x\,\partial u$. The triples~$(x,\partial u,v)$ and~$(\partial u,v,t)$ are critical. Indeed, $x\not\triangleleft\partial u$ because $x\notin\supp(\partial u)$, while $\partial u\not\triangleleft v$ by splittability and $v\not\triangleleft t$ by criticality of $\tau$.

Set
\[
\begin{alignedat}{2}
P_{\tau,0}
&:=
A_\tau\star_1(\alpha_{u,v}1_{r_t})\star_1h_\tau,
\qquad&
P_{\tau,1}
&:=
B_\tau\star_1(p_\tau1_{r_t})\star_1h_\tau,
\\
P_{\tau,2}
&:=
B_\tau\star_1(1_{r_x}f_\tau^L)\star_1q_\tau,
&
P_{\tau,3}
&:=
C_\tau\star_1(1_{r_x}f_\tau^R)\star_1q_\tau,
\\
P_{\tau,4}
&:=
C_\tau\star_1(\alpha_{x,\partial u}1_{\mathbf V})
\star_1\rho_\tau,
&
P_{\tau,5}
&:=
A_\tau\star_1(1_{r_u}\alpha_{v,t})
\star_1\rho_\tau.
\end{alignedat}
\]
Write~$P_\tau:=P_{\tau,0}$ and~$Q_\tau:=P_{\tau,5}$.
These are the boundary paths of the whiskered generating confluence
\[
1_{A_\tau}\star_1\mathcal X_{u,v,t}
\colon
P_\tau\Rrightarrow Q_\tau.
\]

Let~$
\Pi_\tau\colon
C_\tau\star_1(\alpha_{x,\partial u}1_{\mathbf V})
\Rrightarrow
A_\tau\star_1(1_{r_u}\alpha_{v,t})$
be the Peiffer comparison between the disjoint steps on~$(x,\partial u)$ and~$(v,t)$, and define
\[
\begin{alignedat}{3}
F_{1,\tau}
&:=
(\mathcal X_{x,\partial u,v}1_{r_t})
\star_1 1_{h_\tau},
\qquad&
F_{2,\tau}
&:=
1_{B_\tau}\star_1\Theta_\tau^L,
\qquad&
F_{3,\tau}
&:=
(1_{r_x}\mathcal X_{\partial u,v,t})
\star_1 1_{q_\tau},
\\[1mm]
F_{4,\tau}
&:=
1_{C_\tau}\star_1(\Theta_\tau^R)^{-1},
&
F_{5,\tau}
&:=
\Pi_\tau\star_1 1_{\rho_\tau}.
&
&&
\end{alignedat}
\]
By the transport diagrams in
\eqref{eq:TransportFacesSplittableFinal}
and the definitions of
$\mathcal X_{x,\partial u,v}$,
$\mathcal X_{\partial u,v,t}$, and~$\Pi_\tau$, one has
\[
F_{i,\tau}\colon
P_{\tau,i-1}\Rrightarrow P_{\tau,i}
\qquad
(1\leq i\leq5).
\]
Their composite is therefore the $3$-cell
\[
\Psi_\tau
:=
F_{1,\tau}
\star_2F_{2,\tau}
\star_2F_{3,\tau}
\star_2F_{4,\tau}
\star_2F_{5,\tau}
\colon
P_\tau\Rrightarrow Q_\tau.
\]

The diagram below represents the five cells
$F_{1,\tau},\ldots,F_{5,\tau}$ as labels of five faces of the cube,
with their common identity whiskers suppressed. The sixth face is~$1_{A_\tau}\star_1\mathcal X_{u,v,t}$.
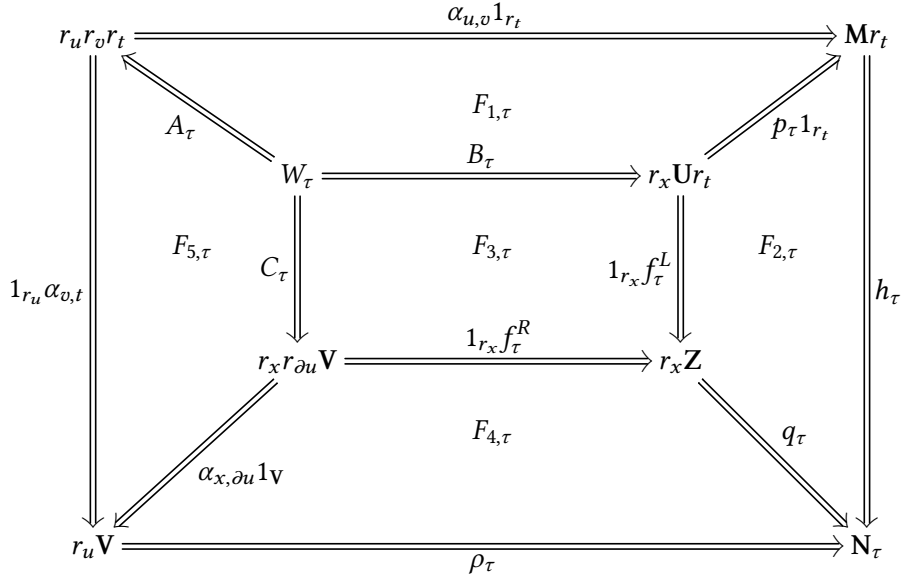
\[
\xymatrix @C=1em @R=1em{
r_ur_vr_t
  \ar@2[rrrrrrrrrr]^-{\alpha_{u,v}1_{r_t}}
  \ar@2[dddddddd]_-{1_{r_u}\alpha_{v,t}}
&&&&&&&&&&
\mathbf M r_t
  \ar@2[dddddddd]^-{h_\tau}
\\
&&&&& { F_{1,\tau}} &&&&&
\\
&&
W_\tau
  \ar@2[uull]^-{A_\tau}
  \ar@2[rrrrrr]^-{B_\tau}
  \ar@2[ddd]_-{C_\tau}
&&&&&&
r_x\mathbf U r_t
  \ar@2[uurr]_-{p_\tau1_{r_t}}
  \ar@2[ddd]_-{1_{r_x}f_\tau^L}
&&
\\
& { F_{5,\tau}}
&&&& { F_{3,\tau}}
&&&& { F_{2,\tau}} &
\\
&&&&&&&&&&
\\
&&
r_xr_{\partial u}\mathbf V
  \ar@2[dddll]^-{\alpha_{x,\partial u}1_{\mathbf V}}
  \ar@2[rrrrrr]^-{1_{r_x}f_\tau^R}
&&&&&&
r_x\mathbf Z
  \ar@2[dddrr]^-{q_\tau}
&&
\\
&&&&& { F_{4,\tau}} &&&&&
\\
&&&&&&&&&&
\\
r_u\mathbf V
  \ar@2[rrrrrrrrrr]_-{\rho_\tau}
&&&&&&&&&&
\mathbf N_\tau
}
\]
Consequently,
\[
\widetilde\Omega_\tau
:=
\bigl(
\Psi_\tau,
1_{A_\tau}\star_1\mathcal X_{u,v,t}
\bigr)
\]
is a $3$-sphere. Since $A_\tau$ is invertible in the track
$2$-category, set~$
S_\tau
:=
1_{A_\tau^{-1}}\star_1\Psi_\tau$.
Prewhiskering $\widetilde\Omega_\tau$ by~$A_\tau^{-1}$ yields the
$3$-sphere~$
\Omega_\tau
:=
\bigl(
S_\tau,
\mathcal X_{u,v,t}
\bigr)$.
Thus $S_\tau$ is parallel to
$\mathcal X_{u,v,t}$.

\begin{proposition}
\label{P:StrictDescentSplittableDependenciesFinal}
Let $\tau=(u,v,t)\in\mathfrak S_n$. If a splittable generating
confluence $\mathcal X_{u',v',t'}$, or its inverse, occurs in
$S_\tau$, then
\[
\ell(u')<\ell(u).
\]
In particular, neither $\mathcal X_{u,v,t}$ nor its inverse occurs
in $S_\tau$.
\end{proposition}

\begin{proof}
The face $F_{1,\tau}$ contains only the letter-first confluence
$\mathcal X_{x,\partial u,v}$. By
Lemma~\ref{L:NonsplittableTransportFacesFinal}, every generating
confluence occurring in $F_{2,\tau}$ or $F_{4,\tau}$ is
nonsplittable, while $F_{5,\tau}$ is Peiffer. Hence a splittable
generating confluence occurring in $\Psi_\tau$ can only come from
$F_{3,\tau}$. Its underlying generator is
$\mathcal X_{\partial u,v,t}$, whose first row has length
\[
\ell(\partial u)=\ell(u)-1<\ell(u).
\]
The identity whisker $1_{A_\tau^{-1}}$ introduces no generating
$3$-cell, so the same conclusion holds in
$S_\tau$.
Thus~$\mathcal X_{u,v,t}$ and its inverse cannot occur in $S_\tau$.
\end{proof}

Set
\[
\Gamma_4^{\mathrm{spl}}(n)
:=
\{\Omega_\tau\mid\tau\in\mathfrak S_n\},
\qquad
\Gamma^{\mathrm{spl}}(n)
:=
\bigl(
\varnothing,
\varnothing,
\Gamma_4^{\mathrm{spl}}(n)
\bigr).
\]

\begin{theorem}
\label{T:FirstSplittableReductionFinal}
For  $n\geq1$, $\Gamma^{\mathrm{spl}}(n)$ is a collapsible part of
$\Sigma^{\mathrm{row,coh}}(n)$. Hence
\[
\Sigma^{\mathrm{row,scoh}}(n)
:=
R_{\Gamma^{\mathrm{spl}}(n)}
\bigl(
\Sigma^{\mathrm{row,coh}}(n)
\bigr)
\]
is a finite coherent presentation of the stylic monoid $\Styl_n$.
Its underlying $2$-polygraph is $\Sigma^{\mathrm{row}}(n)$, and its
generating $3$-cells are
\[
\Sigma^{\mathrm{row,scoh}}_3(n)
=
\left\{
\mathcal X_{u,v,t}
\ \middle|\
(u,v,t)\in
\mathfrak C_n\setminus\mathfrak S_n
\right\}.
\]
\end{theorem}

\begin{proof}
Let $\tau=(u,v,t)\in\mathfrak S_n$. The target of $\Omega_\tau$ is the
generating $3$-cell $\mathcal X_{u,v,t}$, while
Proposition~\ref{P:StrictDescentSplittableDependenciesFinal} shows that
its source $S_\tau$ contains neither $\mathcal X_{u,v,t}$ nor
$\mathcal X_{u,v,t}^{-1}$. Thus condition~\textup{\textbf{i)}} holds.
Condition~\textup{\textbf{ii)}} is automatic because~$
\Gamma_2^{\mathrm{spl}}(n)
=
\Gamma_3^{\mathrm{spl}}(n)
=
\varnothing$.
For condition~\textup{\textbf{iii)}}, choose arbitrary well-founded
orders on the generating $1$-cells and $2$-cells. On the generating
$3$-cells, put every nonsplittable confluence below every splittable
confluence, and order the splittable confluences by increasing length
of the first row of their indexing triples. Refine this order
arbitrarily to a total order. Since all the generating sets are finite,
these orders are well founded.
Let a generating confluence $\mathcal X_{u',v',t'}$, or its inverse,
occur in $S_\tau$.  If
$(u',v',t')$ is nonsplittable, then~$\mathcal X_{u',v',t'}
\prec
\mathcal X_{u,v,t}$.
If it is splittable, then
Proposition~\ref{P:StrictDescentSplittableDependenciesFinal} gives
$\ell(u')<\ell(u)$, and hence~$\mathcal X_{u',v',t'}\prec \mathcal X_{u,v,t}$.
Thus condition~\textup{\textbf{iii)}} holds, and~$\Gamma^{\mathrm{spl}}(n)$ is a collapsible part.

Homotopical reduction therefore removes exactly the generating
$3$-cells indexed by $\mathfrak S_n$, preserves coherence, and leaves
the underlying $2$-polygraph unchanged. Hence the surviving generating
$3$-cells are precisely those indexed by
$\mathfrak C_n\setminus\mathfrak S_n$.
\end{proof}

\subsubsection{Counting the eliminated confluences}

The generating $3$-cells eliminated in the first reduction are indexed by
$\mathfrak S_n$. We now count these cells.

\begin{proposition}
\label{P:CountSplittableTriplesFinal}
Fix $n\geq1$. For a row $v=y_1\cdots y_k$, let $U_n(v)$ be the
number of rows $u$ such that~$\ell(u)\geq2$ and $\partial u\not\triangleleft v$, and let $T_n(v)$
be the number of rows $t$ such that $v\not\triangleleft t$. Then
\begin{align}
U_n(v)
&=
2^n-n-1
-
\sum_{j=2}^k(y_j-1)2^{k-j},
\label{eq:CountSplittableFirstRows}
\\
T_n(v)
&=
2^n-1
-
\bigl(2^{y_1-1}-1\bigr)2^{n-k-y_1+1},
\label{eq:CountSplittableThirdRows}
\end{align}
where the sum in \eqref{eq:CountSplittableFirstRows} is zero when
$k=1$. Moreover,
\begin{equation}
\label{eq:CountSplittableFinal}
|\mathfrak S_n|
=
\sum_{v\in\Rr_n}
U_n(v)T_n(v).
\end{equation}
\end{proposition}

\begin{proof}
Fix a row $v=y_1\cdots y_k$.
We first compute $T_n(v)$. There are $2^n-1$ nonempty rows. The
condition~$v\triangleleft t$ means that $t$ contains all the letters
of $v$ and at least one letter smaller than $y_1$. The required nonempty subset of
$\{1,\ldots,y_1-1\}$ can be chosen in $2^{y_1-1}-1$ ways.
Among the letters larger than $y_1$, exactly
\[
n-y_1-(k-1)=n-k-y_1+1
\]
do not occur in $v$, and an arbitrary subset of these letters may be
chosen. Hence the number of rows $t$ satisfying $v\triangleleft t$ is
\[
\bigl(2^{y_1-1}-1\bigr)2^{n-k-y_1+1}.
\]
Subtracting this number from $2^n-1$ proves
\eqref{eq:CountSplittableThirdRows}.

 We now compute $U_n(v)$. There are $2^n-n-1$ rows of length at least
two. Write each such row uniquely as~$u=x\,\partial u$, where $x$ is
its first letter.  If $\partial u\triangleleft v$ and
$\min(\partial u)=y_j$, then~$y_1=\min(v)<\min(\partial u)=y_j$,
so~$j\geq2$. The remaining letters of $\partial u$ may be chosen
freely among $y_{j+1},\ldots,y_k$, giving $2^{k-j}$ choices.
The first letter $x$ may be any letter smaller than $y_j$, giving
$y_j-1$ choices. Therefore the number of rows $u$ satisfying
$\partial u\triangleleft v$ is
\[
\sum_{j=2}^k
(y_j-1)2^{k-j}.
\]
Subtracting this number from $2^n-n-1$ proves
\eqref{eq:CountSplittableFirstRows}.

By Lemma~\ref{L:SplittableAutomaticallyCriticalFinal}, every row~$u$
counted by $U_n(v)$ also satisfies $u\not\triangleleft v$. Thus, for
a fixed middle row~$v$, the choices of~$u$ and~$t$ are independent,
and $U_n(v)T_n(v)$ is exactly the number of splittable critical
triples with middle row~$v$. Summing over $v\in\Rr_n$ proves
\eqref{eq:CountSplittableFinal}.
\end{proof}

\begin{example}
For $1\leq n\leq6$, the numbers of critical confluences, eliminated
confluences, and surviving generating $3$-cells are as follows:
\[
\begin{array}{|c|r|r|r|}
\hline
n
&
|\mathfrak C_n|
&
|\mathfrak S_n|
&
\bigl|\Sigma^{\mathrm{row,scoh}}_3(n)\bigr|
\\ \hline
1&1&0&1\\
2&21&5&16\\
3&260&111&149\\
4&2635&1506&1129\\
5&24276&16526&7750\\
6&212471&162021&50450\\
\hline
\end{array}
\]
\end{example}

\subsection{Second homotopical reduction: elimination of row relations}
\label{SS:SecondReductionRow}

By Theorem~\ref{T:FirstSplittableReductionFinal}, the generating
confluences that remain after the first reduction are indexed by the
nonsplittable critical triples. Set~$
\mathfrak T_n^{(1)}
:=
\mathfrak C_n\setminus\mathfrak S_n$.
For $\tau=(u,v,t)$ in~$\mathfrak T_n^{(1)}$, write~$
\mathcal X_\tau:=\mathcal X_{u,v,t}$,~$L_\tau:=s_2(\mathcal X_\tau)$ and~$
R_\tau:=t_2(\mathcal X_\tau)$.
Thus~$\mathcal X_\tau\colon L_\tau\Rrightarrow R_\tau$, where~$L_\tau$ and~$R_\tau$ are the two normalization paths fixed in
Subsubsection~\ref{SSS:ExplicitGeneratingConfluences}.

The second reduction concerns the generating row $2$-cells occurring
in empty context in these boundary paths. We first determine the
eligible cells and the dependencies between their eliminations.

\subsubsection{Eligible pairs and dependencies}

Let~$\beta=\alpha_{p,q}\colon
r_pr_q\dfl\Rr(N(pq))$ be a generating row $2$-cell. An occurrence of~$\beta$ in a rewriting
path is said to be \emph{in empty context} if it has no nonempty left
or right whiskering context. Equivalently, the corresponding rewriting
step has source exactly~$r_pr_q$.

Let~$\tau=(u,v,t)$ be in~$\mathfrak T_n^{(1)}$. A pair~$(\tau,\beta)$ is
called \emph{eligible} if:
\begin{enumerate}[\bf i)]
\item $\beta$ occurs in empty context in~$L_\tau$ or~$R_\tau$;
\item after evaluating all row indices, this is the unique occurrence
of~$\beta$ in the two boundary paths~$L_\tau$ and~$R_\tau$, including
occurrences in nonempty context.
\end{enumerate}

The finite set of eligible pairs is denoted by
$\operatorname{Elig}^{(2)}_n$. 

\begin{lemma}
\label{L:EligiblePairGivesCollapsibleCell}
For every~$(\tau,\beta)$ in~$\operatorname{Elig}^{(2)}_n$, a Nielsen
transformation in dimension~$3$ replaces~$\mathcal X_\tau$ by a
generating $3$-cell
\[
A_{\tau,\beta}\colon
S_{\tau,\beta}\Rrightarrow\beta
\]
whose source contains neither~$\beta$ nor its inverse.
\end{lemma}

\begin{proof}
Let~$U$ be the boundary path containing the unique occurrence
of~$\beta$, and let~$V$ be the other boundary path. Since this
occurrence is in empty context, there exist rewriting paths~$f$ and~$g$
such that~$U=f\star_1\beta\star_1g$.
Let~$\mathcal Y_{\tau,\beta}\colon V\Rrightarrow U$
be either~$\mathcal X_\tau$ or its inverse, according to the
orientation of its boundary. Set~$
\widetilde A_{\tau,\beta}
:=
1_{f^{-1}}
\star_1
\mathcal Y_{\tau,\beta}
\star_1
1_{g^{-1}}$
and~$
S_{\tau,\beta}
:=
f^{-1}\star_1V\star_1g^{-1}$.
The source of~$\widetilde A_{\tau,\beta}$ is~$S_{\tau,\beta}$, while
its target is
\[
f^{-1}\star_1U\star_1g^{-1}
=
f^{-1}\star_1f\star_1
\beta
\star_1g\star_1g^{-1}
=
\beta.
\]
Thus~$\widetilde A_{\tau,\beta}\colon
S_{\tau,\beta}\Rrightarrow\beta$.
Adjoin a generating $3$-cell~$A_{\tau,\beta}\colon
S_{\tau,\beta}\Rrightarrow\beta$ parallel to~$\widetilde A_{\tau,\beta}$. Then
$\mathcal Y_{\tau,\beta}$ is parallel to~$1_f\star_1A_{\tau,\beta}\star_1 1_g$.
Hence $\mathcal X_\tau$ is redundant and may be eliminated. This gives
the required Nielsen transformation.

By eligibility, none of~$f$, $g$, and~$V$ contains~$\beta$. Since
$f$, $g$, and~$V$ are rewriting paths, only~$f^{-1}$ and~$g^{-1}$
contain inverse generating cells. Moreover, since neither~$f$ nor~$g$
contains~$\beta$, their inverses do not contain the inverse of~$\beta$.
It follows that~$S_{\tau,\beta}$ contains neither~$\beta$ nor its
inverse.
\end{proof}

A generating row $2$-cell~$\gamma$ is called a \emph{dependency} of
$(\tau,\beta)$ if~$\gamma$ or its inverse occurs in
$S_{\tau,\beta}$, possibly in nonempty context.

\begin{proposition}
\label{P:ExactNielsenSourcesByCase}
The possible eligible pairs and their dependencies are given by the
following table:
\[
\begin{array}{|c|c|l|}
\hline
\text{type}
& \beta
& \text{dependencies in }S_{\tau,\beta}
\\ \hline
1
& \alpha_{b,t}
& \alpha_{u,v},\ \alpha_{v,t},\ \alpha_{u,d}
\\
1
& \alpha_{u,d}
& \alpha_{u,v},\ \alpha_{v,t},\ \alpha_{b,t}
\\
2\mathrm a
& \alpha_{b,t}
& \alpha_{u,v},\ \alpha_{v,t},\ \alpha_{u,c}
\\
2\mathrm b
& \alpha_{b,t}
& \alpha_{u,v},\ \alpha_{v,t},\ \alpha_{u,c},\ \alpha_{g,d}
\\
2\mathrm b
& \alpha_{g,d}
& \alpha_{u,v},\ \alpha_{v,t},\ \alpha_{u,c},\ \alpha_{b,t}
\\
3\mathrm{a2}
& \alpha_{u,d}
& \alpha_{u,v},\ \alpha_{v,t},\ \alpha_{b,t}
\\
3\mathrm{b2}
& \alpha_{u,d}
& \alpha_{u,v},\ \alpha_{v,t},\ \alpha_{b,t},\ \alpha_{a,e}
\\
4\mathrm{a2}
& \alpha_{g,d}
& \alpha_{u,v},\ \alpha_{v,t},\ \alpha_{u,c},\ \alpha_{b,t},\
  \alpha_{a,e}
\\ \hline
\end{array}
\]
No other confluence type contributes an eligible pair.
\end{proposition}

\begin{proof}
By Lemma~\ref{L:EligiblePairGivesCollapsibleCell}, the dependencies
of~$(\tau,\beta)$ are precisely the generating row $2$-cells occurring
in the two boundary paths of~$\mathcal X_\tau$, except for the unique
occurrence of~$\beta$. Inspection of the thirteen confluence diagrams of Subsubsection~\ref{SSS:ExplicitGeneratingConfluences} gives the table.
\end{proof}

All indices in the table are evaluated using the rows attached to
$\tau$. A line of the table applies only when the displayed pair
$(\tau,\beta)$ satisfies the eligibility condition. If two listed
dependencies become equal after evaluation, they define the same
dependency.

\subsubsection{Deterministic matching of eligible pairs}
Order the generating row $2$-cells
lexicographically by
\[
\bigl(\ell(u)+\ell(v),\ell(u),u,v\bigr),
\]
where rows are compared lexicographically and a proper prefix is smaller.
Order critical triples
lexicographically by their three rows, and order eligible pairs
lexicographically by~$(\beta,\tau)$.
View~$\operatorname{Elig}^{(2)}_n$ as the edge set of the bipartite
graph with vertex sets
$\mathfrak T_n^{(1)}$ and~$\Sigma^{\mathrm{row}}_2(n)$. Scan its
edges in the order above and retain an edge~$(\tau,\beta)$ precisely
when neither endpoint belongs to a previously retained edge. This
defines a deterministic matching~$\mathcal M_n^{(2)}
\subseteq
\operatorname{Elig}^{(2)}_n$.
Thus every critical triple and every generating row $2$-cell occurs
in at most one matched pair.

Set
\[
\mathcal B_n^{\mathrm{match}}
:=
\left\{
\beta
\ \middle|\
(\tau,\beta)\in\mathcal M_n^{(2)}
\text{ for some }\tau
\right\}.
\]
For~$\beta$ in~$\mathcal B_n^{\mathrm{match}}$, let~$\tau_\beta$ be its
unique matched triple and set~$
S_\beta
:=
S_{\tau_\beta,\beta}$ and~$
A_\beta
:=
A_{\tau_\beta,\beta}
\colon
S_\beta\Rrightarrow\beta$.

Let~$G_n^{(2)}$ be the directed graph with vertex set
$\mathcal B_n^{\mathrm{match}}$. For
$\beta,\beta'\in\mathcal B_n^{\mathrm{match}}$, it has an edge
\[
\beta'\longrightarrow\beta
\]
whenever~$\beta'$ is a dependency of~$(\tau_\beta,\beta)$. Thus an
edge~$\beta'\to\beta$ means that the source~$S_\beta$ depends
on~$\beta'$. In the order used for homotopical reduction, $\beta'$
must therefore be placed below~$\beta$.
By Lemma~\ref{L:EligiblePairGivesCollapsibleCell}, the source
$S_\beta$ contains neither~$\beta$ nor its inverse. Hence
$G_n^{(2)}$ has no loops. By
Proposition~\ref{P:ExactNielsenSourcesByCase}, its edges are obtained
from the dependency table after evaluating all row indices and
retaining only the matched row cells.
Scan the vertices of~$\mathcal B_n^{\mathrm{match}}$ in the fixed
row-cell order. Starting with the empty set, retain a vertex~$\beta$
precisely when the subgraph induced by the previously retained
vertices together with~$\beta$ is acyclic. Denote the resulting set by~$\mathcal B_n^{\mathrm{acyc}}
\subseteq
\mathcal B_n^{\mathrm{match}}$.
By construction,~$
G_n^{(2)}
\bigl[
\mathcal B_n^{\mathrm{acyc}}
\bigr]$
is a finite directed acyclic graph and therefore admits a topological
order compatible with all its dependencies. Dependencies outside
$\mathcal B_n^{\mathrm{acyc}}$ will later be placed below all selected
row cells.

\begin{remark}
The full graph~$G_n^{(2)}$ need not be acyclic. The deterministic
selection of~$\mathcal B_n^{\mathrm{acyc}}$ discards vertices as
needed so that the induced subgraph is acyclic.
\end{remark}

\subsubsection{The reduced coherent row presentation}

Let~$\Sigma^{\mathrm{row,N}}(n)$ be obtained from
$\Sigma^{\mathrm{row,scoh}}(n)$ by replacing, for every
$\beta\in\mathcal B_n^{\mathrm{acyc}}$,
\[
\mathcal X_{\tau_\beta}
\rightsquigarrow
A_\beta\colon
S_\beta\Rrightarrow\beta.
\]
Since the triples~$\tau_\beta$ are pairwise distinct, these Nielsen
transformations involve distinct generating $3$-cells
$\mathcal X_{\tau_\beta}$. They may therefore be performed successively
in any order. They preserve coherence and leave the underlying
$2$-polygraph unchanged.

Set
\[
\Gamma_3^{(2)}(n)
:=
\left\{
A_\beta
\ \middle|\
\beta\in\mathcal B_n^{\mathrm{acyc}}
\right\},
\qquad
\Gamma^{(2)}(n)
:=
\bigl(
\varnothing,
\Gamma_3^{(2)}(n),
\varnothing
\bigr).
\]

\begin{proposition}
\label{P:SecondAcyclicFamilyCollapsible}
The triple~$\Gamma^{(2)}(n)$ is a collapsible part of
$\Sigma^{\mathrm{row,N}}(n)$.
\end{proposition}

\begin{proof}
Let~$\beta\in\mathcal B_n^{\mathrm{acyc}}$. The target of~$A_\beta$
is the generating row $2$-cell~$\beta$. By
Lemma~\ref{L:EligiblePairGivesCollapsibleCell}, its source~$S_\beta$
contains neither~$\beta$ nor its inverse. Hence~$A_\beta$ is
collapsible, and condition~\textup{\textbf{i)}} holds.
Condition~\textup{\textbf{ii)}} is automatic because~$\Gamma_2^{(2)}(n)=\Gamma_4^{(2)}(n)=\varnothing$.
For condition~\textup{\textbf{iii)}}, only an order on the generating
$2$-cells is relevant. Since~$
G_n^{(2)}
\bigl[
\mathcal B_n^{\mathrm{acyc}}
\bigr]$
is a finite directed acyclic graph, choose a topological order
$\prec_2$ on~$\mathcal B_n^{\mathrm{acyc}}$ such that
\[
\beta'\longrightarrow\beta
\quad\text{ implies }\quad
\beta'\prec_2\beta.
\]
Extend~$\prec_2$ to a total order on all generating row $2$-cells by
placing every cell outside~$\mathcal B_n^{\mathrm{acyc}}$ below every
selected cell.
Let~$\gamma$ be a generating row $2$-cell such that~$\gamma$ or its
inverse occurs in~$S_\beta$. If
$\gamma\in\mathcal B_n^{\mathrm{acyc}}$, then
$\gamma\to\beta$ is an edge of the dependency graph, and hence
$\gamma\prec_2\beta$. If
$\gamma\notin\mathcal B_n^{\mathrm{acyc}}$, the same inequality follows
from the placement of the unselected cells. Thus~$\beta$ is strictly
greater than every generating $2$-cell occurring in~$S_\beta$.
The set of generating row $2$-cells is finite, so this order is well
founded. Choose arbitrary well-founded orders on the generating
$1$-cells and $3$-cells. Condition~\textup{\textbf{iii)}} follows, and
therefore~$\Gamma^{(2)}(n)$ is a collapsible part.
\end{proof}

\begin{theorem}
\label{T:SecondReductionRowRed}
For~$n\geq1$, the homotopical reduction
\[
\Sigma^{\mathrm{row,red}}(n)
:=
R_{\Gamma^{(2)}(n)}
\bigl(\Sigma^{\mathrm{row,N}}(n)\bigr)
\]
is a finite coherent presentation of~$\Styl_n$. Its generating cells are
\[
\Sigma^{\mathrm{row,red}}_2(n)
=
\Sigma^{\mathrm{row}}_2(n)
\setminus\mathcal B_n^{\mathrm{acyc}},
\qquad
\Sigma^{\mathrm{row,red}}_3(n)
=
\left\{
\mathcal X_\tau
\ \middle|\
\tau\in\mathfrak T_n^{(1)}
\setminus
\{\tau_\beta\mid\beta\in\mathcal B_n^{\mathrm{acyc}}\}
\right\},
\]
where the boundaries of the surviving generating $3$-cells are reduced
by~$R_{\Gamma^{(2)}(n)}$.
\end{theorem}

\begin{proof}
By Theorem~\ref{T:FirstSplittableReductionFinal},
$\Sigma^{\mathrm{row,scoh}}(n)$ is a finite coherent presentation of
$\Styl_n$. The Nielsen transformations defining
$\Sigma^{\mathrm{row,N}}(n)$ preserve coherence and leave the
underlying $2$-polygraph unchanged. By
Proposition~\ref{P:SecondAcyclicFamilyCollapsible},
$\Gamma^{(2)}(n)$ is a collapsible part. Homotopical reduction
therefore preserves coherence and the presented monoid. All the sets
involved are finite, so the Nielsen transformations and the
homotopical reduction preserve finiteness. Hence
$\Sigma^{\mathrm{row,red}}(n)$ is a finite coherent presentation of
$\Styl_n$.
\end{proof}

\subsubsection{Cell counts}

We conclude by recording the numbers of generating cells remaining
after the two homotopical reductions.

\begin{proposition}
\label{P:CountsAfterSecondReduction}
For every~$n\geq1$, the numbers of generating cells of
$\Sigma^{\mathrm{row,red}}(n)$ are given by
\[
\bigl|\Sigma^{\mathrm{row,red}}_1(n)\bigr|
=
2^n-1,
\qquad
\bigl|\Sigma^{\mathrm{row,red}}_2(n)\bigr|
=
\frac{
2^{2n+1}-2^{n+1}-3^n+1
}{2}
-
\bigl|\mathcal B_n^{\mathrm{acyc}}\bigr|,
\]
\[
\bigl|\Sigma^{\mathrm{row,red}}_3(n)\bigr|
=
\frac{
6\cdot8^n-6^{n+1}-5\cdot4^n+3^{n+1}
+3\cdot2^n-1
}{6}
-
\bigl|\mathfrak S_n\bigr|
-
\bigl|\mathcal B_n^{\mathrm{acyc}}\bigr|.
\]
\end{proposition}

\begin{proof}
The first reduction removes exactly the generating $3$-cells indexed
by~$\mathfrak S_n$ and leaves the underlying $2$-polygraph unchanged.
The Nielsen transformations preserve the number of generating
$3$-cells.

For every~$\beta\in\mathcal B_n^{\mathrm{acyc}}$, the second reduction
removes the generating $3$-cell~$A_\beta$ and its target row
$2$-cell~$\beta$. Since the matching has distinct $2$-cell targets and
distinct $3$-cell generators, exactly
$\bigl|\mathcal B_n^{\mathrm{acyc}}\bigr|$ generators are removed in
each of dimensions~$2$ and~$3$. No generating $1$-cell is removed.
The formulas now follow from Proposition~\ref{Prop:CountingCells}.
\end{proof}

\subsubsection{Exact finite-rank verification}

For~$1\leq n\leq5$, exhaustive application of the deterministic
constructions above gives:
\[
\begin{array}{|c|r|r|r|r|r|r|}
\hline
n
&
\bigl|\mathfrak T_n^{(1)}\bigr|
&
\bigl|\operatorname{Elig}^{(2)}_n\bigr|
&
\bigl|\mathcal M_n^{(2)}\bigr|
&
\bigl|\mathcal B_n^{\mathrm{acyc}}\bigr|
&
\bigl|\Sigma^{\mathrm{row,red}}_2(n)\bigr|
&
\bigl|\Sigma^{\mathrm{row,red}}_3(n)\bigr|
\\ \hline
1&1&0&0&0&1&1\\
2&16&5&4&3&5&13\\
3&149&66&32&25&18&124\\
4&1129&500&172&143&57&986\\
5&7750&2955&798&701&170&7049\\\hline
\end{array}
\]

The classification of critical confluences was checked through rank~$6$.
The eligible pairs, deterministic matching, dependency graph, and
acyclic selection were computed through rank~$5$.

\subsection{A coherent presentation on the letter generators}
\label{SS:CoherentPresentationLetterGenerators}

We now pass from the reduced row presentation to the letter
presentation. We first transport
$\Sigma^{\mathrm{row,red}}(n)$ to the pre-row presentation. We then
eliminate the non-letter row generators using the compression cells.

\subsubsection{Transport to the pre-row presentation}

By Theorem~\ref{T:SecondReductionRowRed},
$\Sigma^{\mathrm{row,red}}(n)$ is a finite coherent presentation of
$\Styl_n$. Its underlying $2$-polygraph is obtained from~$\Sigma^{\mathrm{row}}(n)$ by eliminating the cells in~$\mathcal B_n^{\mathrm{acyc}}$.
The second homotopical reduction induces a $2$-functor
\[
\mathcal R_n\colon
\Sigma^{\mathrm{row}}(n)_2^\top
\longrightarrow
\Sigma^{\mathrm{row,red}}(n)_2^\top.
\]
It fixes every generating $1$-cell and every surviving row cell and, for
$\beta$ in~$\mathcal B_n^{\mathrm{acyc}}$, satisfies~$\mathcal R_n(\beta)=\mathcal R_n(S_\beta)$.
If a selected cell~$\beta'$ or its inverse occurs in~$S_\beta$, then~$\beta'\to\beta$ is an edge of the dependency graph, so~$\beta'$ is evaluated before~$\beta$ in a topological order. Hence this recursive
definition is well defined and terminates.

Let~$
\delta\colon
s_1(\delta)\dfl t_1(\delta)$
be a generating $2$-cell of
$\Sigma^{\mathrm{prerow}}(n)$. 
Since~$\delta$ preserves the represented stylic element and
$\Sigma^{\mathrm{row}}(n)$ is convergent and presents~$\Styl_n$, one has~$\widehat{s_1(\delta)}=\widehat{t_1(\delta)}$.
Hence the track $2$-category~$\Sigma^{\mathrm{row}}(n)_2^\top$ contains the path
\[
\npath{s_1(\delta)}
\star_1
\npath{t_1(\delta)}^{-1}
\colon
s_1(\delta)\dfl t_1(\delta).
\]
Applying $\mathcal R_n$ gives the path
\[
Q_\delta
:=
\mathcal R_n
\left(
\npath{s_1(\delta)}
\star_1
\npath{t_1(\delta)}^{-1}
\right)
\colon
s_1(\delta)\dfl t_1(\delta)
\]
in the reduced row presentation.
We adjoin every generating cell
$\delta\in\Sigma^{\mathrm{prerow}}_2(n)$ as a redundant $2$-cell,
using~$Q_\delta$ as its replacement path. At the coherent level, we
also adjoin the comparison $3$-cell~$
B_\delta\colon
Q_\delta\Rrightarrow\delta$.

Let~$\Xi_n$ be the underlying $2$-polygraph obtained after all these
adjunctions. Its generating $2$-cells are the surviving row cells
together with all the generating cells of
$\Sigma^{\mathrm{prerow}}(n)$.

We now eliminate the surviving row cells. Let~$
\alpha_{u,v}\colon
r_ur_v\dfl\Rr(N(uv))$
be such a cell. Recall the canonical pre-row path~$
P_{u,v}
:=
C_{r_ur_v,\Rr(N(uv))}$
defined in~\eqref{Eq:CanonicalPreRowRowCell}. 
Since all the pre-row cells have been adjoined, the path $P_{u,v}$ is
defined in $\Xi_n$ and contains no generating cell
$\alpha_{p,q}$ of the row presentation.
The current presentation is coherent. Hence the $2$-sphere
$(P_{u,v},\alpha_{u,v})$ admits a filling~$
H_{u,v}\colon
P_{u,v}\Rrightarrow\alpha_{u,v}$.
Adjoin a generating $3$-cell~$
D_{u,v}\colon
P_{u,v}\Rrightarrow\alpha_{u,v}$
parallel to~$H_{u,v}$. Since $P_{u,v}$ does not contain
$\alpha_{u,v}$, the pair~$(\alpha_{u,v},D_{u,v})$ is collapsible and may be eliminated.
Repeating this construction removes every surviving row cell. The
fillings~$H_{u,v}$ are used only to show that the generating cells~$D_{u,v}$ are redundant, while each~$D_{u,v}$ is eliminated together with~$\alpha_{u,v}$.

The effect of these eliminations is described by the $2$-functor
\[
\mathcal E_n\colon
(\Xi_n)_2^\top
\longrightarrow
\Sigma^{\mathrm{prerow}}(n)_2^\top.
\]
It fixes every generating $1$-cell and every pre-row $2$-cell, and
satisfies~$\mathcal E_n(\alpha_{u,v})=P_{u,v}$ for every surviving row cell~$\alpha_{u,v}$. Its action on identities,
inverses, and compositions is determined by functoriality.

If~$A\colon f\Rrightarrow g$
is a generating $3$-cell of
$\Sigma^{\mathrm{row,red}}(n)$, its transported generator is~$
A^{\mathrm{prerow}}
\colon
\mathcal E_n(f)
\Rrightarrow
\mathcal E_n(g)$.
Similarly, each comparison cell
$B_\delta\colon Q_\delta\Rrightarrow\delta$ becomes~$
B_\delta^{\mathrm{prerow}}
\colon
\mathcal E_n(Q_\delta)
\Rrightarrow
\delta$.

Denote the resulting $(3,1)$-polygraph by
$\Sigma^{\mathrm{prerow,coh}}(n)$. Its underlying $2$-polygraph is
$\Sigma^{\mathrm{prerow}}(n)$, and its generating $3$-cells are
\[
\Sigma^{\mathrm{prerow,coh}}_3(n)
=
\left\{
A^{\mathrm{prerow}}
\ \middle|\
A\in\Sigma^{\mathrm{row,red}}_3(n)
\right\}
\sqcup
\left\{
B_\delta^{\mathrm{prerow}}
\ \middle|\
\delta\in\Sigma^{\mathrm{prerow}}_2(n)
\right\}.
\]

The preceding adjunctions and eliminations define a finite Tietze
sequence
\[
\mathcal T_n^{\mathrm{red}}\colon
\Sigma^{\mathrm{row,red}}(n)
\rightsquigarrow
\Sigma^{\mathrm{prerow}}(n),
\]
and the construction above is its coherent lift.

\begin{proposition}
\label{Prop:CoherentPreRow}
For $n\geq1$, the $(3,1)$-polygraph
$\Sigma^{\mathrm{prerow,coh}}(n)$ is a finite coherent presentation
of the stylic monoid~$\Styl_n$.
\end{proposition}

\begin{proof}
By Theorem~\ref{T:SecondReductionRowRed},
$\Sigma^{\mathrm{row,red}}(n)$ is a finite coherent presentation of~$\Styl_n$. 
The construction above is a coherent lift of the finite Tietze
sequence~$\mathcal T_n^{\mathrm{red}}$. This coherent lift preserves
coherence and the presented monoid. Its underlying $2$-polygraph is~$\Sigma^{\mathrm{prerow}}(n)$, which presents~$\Styl_n$ by Proposition~\ref{P:PrerowPresentsStylic}. Hence~$\Sigma^{\mathrm{prerow,coh}}(n)$ is a coherent presentation of~$\Styl_n$. It is finite because, for fixed~$n$, all the generating
sets involved are finite.
\end{proof}

\subsubsection{Elimination of the non-letter row generators}

For every row $u=x_1\cdots x_k$ in~$\overline{\Rr}_n$, consider the
compression cell~$
\gamma_u\colon
r_{x_1}\cdots r_{x_k}\dfl r_u$.
Set
\[
\Gamma^{\mathrm{let}}_2(n)
:=
\left\{
\gamma_u
\ \middle|\
u\in\overline{\Rr}_n
\right\},
\qquad
\Gamma^{\mathrm{let}}(n)
:=
\bigl(
\Gamma^{\mathrm{let}}_2(n),
\varnothing,
\varnothing
\bigr).
\]

\begin{lemma}
\label{L:LetterCompressionCollapsible}
For  $n\geq1$, the triple
$\Gamma^{\mathrm{let}}(n)$ is a collapsible part of
$\Sigma^{\mathrm{prerow,coh}}(n)$.
\end{lemma}

\begin{proof}
Let $u=x_1\cdots x_k$ be a row in~$\overline{\Rr}_n$. The target of
$\gamma_u$ is the generating $1$-cell~$r_u$, whereas its source
contains only the letter generators
$r_{x_1},\ldots,r_{x_k}$. Hence $r_u$ does not occur in the source of
$\gamma_u$, and condition~\textup{\textbf{i)}} holds.
Condition~\textup{\textbf{ii)}} is automatic because
$\Gamma^{\mathrm{let}}_3(n)=
\Gamma^{\mathrm{let}}_4(n)=\varnothing$.
For condition~\textup{\textbf{iii)}}, order the generating $1$-cells
so that every letter generator lies below every non-letter row
generator, and refine this to a total order. This order is well founded
because the set of row generators is finite. Choose arbitrary
well-founded orders on the generating $2$- and $3$-cells. Then every
generator occurring in the source of~$\gamma_u$ is strictly smaller
than its target~$r_u$. Thus
$\Gamma^{\mathrm{let}}(n)$ is a collapsible part.
\end{proof}

Define
\[
\Sigma^{\mathrm{styl,coh}}(n)
:=
R_{\Gamma^{\mathrm{let}}(n)}
\bigl(
\Sigma^{\mathrm{prerow,coh}}(n)
\bigr).
\]

\begin{theorem}
\label{T:LetterCoherentPresentation}
For every $n\geq1$, the $(3,1)$-polygraph
$\Sigma^{\mathrm{styl,coh}}(n)$ is a finite coherent presentation of
the stylic monoid~$\Styl_n$. After identifying each~$r_x$ with the
letter~$x$, its underlying $2$-polygraph is
$\Sigma^{\mathrm{styl}}(n)$. Its generating $3$-cells form an
effectively computable finite homotopy basis of
$\Sigma^{\mathrm{styl}}(n)_2^\top$.
\end{theorem}

\begin{proof}
By Proposition~\ref{Prop:CoherentPreRow},
$\Sigma^{\mathrm{prerow,coh}}(n)$ is a finite coherent presentation
of~$\Styl_n$. By
Lemma~\ref{L:LetterCompressionCollapsible},
$\Gamma^{\mathrm{let}}(n)$ is a collapsible part. Homotopical
reduction therefore preserves coherence and the presented monoid.
For every non-letter row $u=x_1\cdots x_k$, the reduction satisfies
\[
R_{\Gamma^{\mathrm{let}}(n)}(r_u)
=
r_{x_1}\cdots r_{x_k},
\qquad
R_{\Gamma^{\mathrm{let}}(n)}(\gamma_u)
=
1_{r_{x_1}\cdots r_{x_k}}.
\]
Thus all non-letter row generators and all compression cells are
eliminated. The remaining generating $1$-cells are the letter
generators~$r_x$, and the remaining generating $2$-cells are the
lifted stylic relations. Under the identification $r_x=x$, they become
\[
\eta_x^r\longmapsto\eta_x,
\qquad
\kappa_{x,y,z}^r\longmapsto\kappa_{x,y,z},
\qquad
\lambda_{x,y,z}^r\longmapsto\lambda_{x,y,z}.
\]
Hence the underlying $2$-polygraph is
$\Sigma^{\mathrm{styl}}(n)$.
Since $\Gamma^{\mathrm{let}}_3(n)=\varnothing$, no generating
$3$-cell is eliminated. The generating $3$-cells of~$\Sigma^{\mathrm{prerow,coh}}(n)$ are retained, with their boundaries
transformed by~$R_{\Gamma^{\mathrm{let}}(n)}$. They therefore form a
homotopy basis of~$\Sigma^{\mathrm{styl}}(n)_2^\top$.
This basis is finite and effectively computable. Its boundaries are
obtained from the fixed paths~$Q_\delta$ and~$P_{u,v}$ by applying
$\mathcal E_n$ and
$R_{\Gamma^{\mathrm{let}}(n)}$. The auxiliary fillings~$H_{u,v}$ are
used only to justify the coherent eliminations and do not occur in the
final generators or their boundaries. Since all the generating sets
are finite for fixed~$n$, the result follows.
\end{proof}

\subsubsection{Cell counts in the final presentation}
We now record the numbers of generating cells in the coherent presentation on the letter generators.

\begin{proposition}
\label{P:FinalLetterCellCounts}
For every $n\geq1$, the numbers of generating cells of
$\Sigma^{\mathrm{styl,coh}}(n)$ are
\[
\begin{aligned}
\bigl|\Sigma^{\mathrm{styl,coh}}_1(n)\bigr|
&=n,
\qquad
\bigl|\Sigma^{\mathrm{styl,coh}}_2(n)\bigr|
=\frac{n^3+2n}{3},
\\
\bigl|\Sigma^{\mathrm{styl,coh}}_3(n)\bigr|
&=
\bigl|\Sigma^{\mathrm{row,red}}_3(n)\bigr|
+
2^n-1+\frac{n^3-n}{3}.
\end{aligned}
\]
\end{proposition}

\begin{proof}
The generating $1$-cells are the $n$ letters. The generating $2$-cells
are the $n$ idempotent relations and the two families of Knuth-type
relations. Each of these two families has
$\binom{n+1}{3}$ elements. Hence
\[
\bigl|\Sigma^{\mathrm{styl,coh}}_2(n)\bigr|
=
n+2\binom{n+1}{3}
=
\frac{n^3+2n}{3}.
\]

The transport to the pre-row presentation adds one comparison
$3$-cell~$B_\delta$ for every generating pre-row $2$-cell~$\delta$.
The pre-row presentation has~$\displaystyle
n+2\binom{n+1}{3}$
lifted stylic cells and~$(2^n-n-1)$
compression cells. Therefore
\[
\bigl|\Sigma^{\mathrm{prerow}}_2(n)\bigr|
=
2^n-1+\frac{n^3-n}{3}.
\]

The auxiliary cells~$D_{u,v}$ are eliminated together with the row
cells~$\alpha_{u,v}$. The final reduction has
$\Gamma^{\mathrm{let}}_3(n)=\varnothing$, so it eliminates no
generating $3$-cell. Consequently,
\[
\bigl|\Sigma^{\mathrm{styl,coh}}_3(n)\bigr|
=
\bigl|\Sigma^{\mathrm{row,red}}_3(n)\bigr|
+
\bigl|\Sigma^{\mathrm{prerow}}_2(n)\bigr|,
\]
which gives the stated formula.
\end{proof}

\begin{example}
For~$1\leq n\leq5$, the resulting numbers are
\[
\begin{array}{|c|r|r|r|}
\hline
n
&
\bigl|\Sigma^{\mathrm{styl,coh}}_1(n)\bigr|
&
\bigl|\Sigma^{\mathrm{styl,coh}}_2(n)\bigr|
&
\bigl|\Sigma^{\mathrm{styl,coh}}_3(n)\bigr|
\\ \hline
1&1&1&2\\
2&2&4&18\\
3&3&11&139\\
4&4&24&1021\\
5&5&45&7120\\
\hline
\end{array}
\]
\end{example}

\subsubsection{Categorical actions}
\label{SSS:CategoricalActions}

We finally interpret the coherent letter presentation in terms of
categorical actions.
For $2$-categories $\mathcal A$ and $\mathcal D$, let
$2\operatorname{Rep}(\mathcal A,\mathcal D)$ denote the category of
pseudofunctors from $\mathcal A$ to~$\mathcal D$ and pseudonatural
transformations. We denote by~$2\Cat(\mathcal A,\mathcal D)$ its full subcategory whose objects are
strict $2$-functors. Here~$\Cat$ is the $2$-category of small
categories, functors, and natural transformations.

Let $M$ be a monoid, regarded as a $2$-category with one $0$-cell,
the elements of $M$ as $1$-cells, and only identity $2$-cells. Following
\cite[Subsection~5.1]{GaussentGuiraudMalbos15}, set~$\Act(M):=2\operatorname{Rep}(M,\Cat)$.
The $2$-category associated with the coherent presentation
$\Sigma^{\mathrm{styl,coh}}(n)$ is
\[
\mathcal C_{\mathrm{styl}}^{\mathrm{coh}}(n)
:=
\Sigma^{\mathrm{styl}}(n)_2^\top
\big/
\Sigma^{\mathrm{styl,coh}}_3(n),
\]
where the quotient identifies the two boundary $2$-cells of every
generating $3$-cell.

Since $\Sigma^{\mathrm{styl,coh}}(n)$ is a coherent presentation of
$\Styl_n$, \cite[Theorem~5.1.6]{GaussentGuiraudMalbos15} gives an
equivalence of categories
\[
\Act(\Styl_n)
\simeq
2\Cat\bigl(
\mathcal C_{\mathrm{styl}}^{\mathrm{coh}}(n),
\Cat
\bigr).
\]
Equivalently, the objects on the right-hand side are the strict $2$-functors
\[
\Phi\colon
\Sigma^{\mathrm{styl}}(n)_2^\top
\longrightarrow
\Cat
\]
such that~$
\Phi\bigl(s_2(A)\bigr)
=
\Phi\bigl(t_2(A)\bigr)$
for every generating $3$-cell
$A\in\Sigma^{\mathrm{styl,coh}}_3(n)$.
If $\Phi(\bullet)=\mathbf C$, each letter~$x$ in~$[n]$ determines an
endofunctor
\[
F_x:=\Phi(x)\colon
\mathbf C\longrightarrow\mathbf C.
\]
The defining stylic $2$-cells determine the natural isomorphisms
$\Phi(\eta_x)$, $\Phi(\kappa_{x,y,z})$, and
$\Phi(\lambda_{x,y,z})$, while the generating $3$-cells impose their
coherence diagrams. Thus, up to equivalence, categorical
actions of~$\Styl_n$ are described by a finite family of endofunctors,
natural isomorphisms, and coherence conditions.

\section[Example: the case n=2]
{Example: the case \texorpdfstring{$n=2$}{n=2}}
\label{S:Example}

We conclude with the case of rank~$2$. It illustrates the two
homotopical reductions of the coherent row presentation and the passage
from the row presentation to a coherent presentation on the letter
generators.

\subsection{The row presentation}

For $n=2$, one has~$\Rr_2=\{1,2,12\}$,
and the only pair of rows satisfying the relation~$\triangleleft$ is~$2\triangleleft12$.
Hence the row presentation has the generating $1$-cells
$r_1,r_2,r_{12}$ and the eight generating $2$-cells
\[
r_1r_1\odfl{\alpha_{1,1}}r_1,
\qquad
r_1r_2\odfl{\alpha_{1,2}}r_{12},
\qquad
r_1r_{12}\odfl{\alpha_{1,12}}r_{12},
\qquad
r_2r_1\odfl{\alpha_{2,1}}r_2r_{12},
\]
\[
r_2r_2\odfl{\alpha_{2,2}}r_2,
\qquad
r_{12}r_1\odfl{\alpha_{12,1}}r_2r_{12},
\qquad
r_{12}r_2\odfl{\alpha_{12,2}}r_{12},
\qquad
r_{12}r_{12}\odfl{\alpha_{12,12}}r_2r_{12}.
\]
There are $3^3=27$ triples of rows. A triple is not critical precisely
when~$
(u,v)=(2,12)$ or~$(v,t)=(2,12)$.
Each condition excludes three triples, and they cannot hold
simultaneously. Therefore~$|\mathfrak C_2|=21$.
Thus the coherent row presentation
$\Sigma^{\mathrm{row,coh}}(2)$ has $21$ generating $3$-cells.

\subsection{The first reduction}

Since $12$ is the only row of length at least two and
$\partial(12)=2$, splittability forces~$
u=12$ and~$v\in\{1,2\}$.
If $v=1$, every row $t$ is possible. If $v=2$, criticality excludes
$t=12$. Hence
\[
\mathfrak S_2
=
\left\{
(12,1,1),
(12,1,2),
(12,1,12),
(12,2,1),
(12,2,2)
\right\}.
\]
The first reduction removes the five corresponding generating
confluences. Among the sixteen surviving triples, thirteen are
letter-first: eight have first row~$1$ and five have first row~$2$.
The remaining three are~$(12,12,1)$,~$(12,12,2)$, and~$(12,12,12)$.
Therefore
\[
\bigl|\Sigma^{\mathrm{row,scoh}}_3(2)\bigr|
=
21-5
=
16.
\]
The underlying row $2$-polygraph is unchanged.

\subsection{The second reduction}

Let~$\mathfrak T_2^{(1)}=\mathfrak C_2\setminus\mathfrak S_2$.
The eligible pairs, listed in the deterministic order fixed above, are
\[
\begin{array}{|c|c|c|}
\hline
\tau & \text{type} & \beta \\ \hline
(1,1,2)  & 1             & \alpha_{1,12}\\
(1,2,1)  & 2\mathrm b    & \alpha_{12,1}\\
(1,2,2)  & 1             & \alpha_{12,2}\\
(1,12,1) & 2\mathrm b    & \alpha_{12,12}\\
(1,2,1)  & 2\mathrm b    & \alpha_{12,12}\\ \hline
\end{array}
\]
The triple~$(1,2,1)$ occurs twice because it forms an eligible pair
with both~$\alpha_{12,1}$ and~$\alpha_{12,12}$.
The deterministic matching retains the first four pairs. Its row-cell
vertices are therefore
\[
\mathcal B_2^{\mathrm{match}}
=
\left\{
\alpha_{1,12},
\alpha_{12,1},
\alpha_{12,2},
\alpha_{12,12}
\right\}.
\]
The dependency graph has exactly the three edges
\[
\alpha_{1,12}\longrightarrow\alpha_{12,12},
\qquad
\alpha_{12,1}\longrightarrow\alpha_{12,12},
\qquad
\alpha_{12,12}\longrightarrow\alpha_{12,1}.
\]
In particular, it contains the directed cycle
\[
\alpha_{12,1}
\longrightarrow
\alpha_{12,12}
\longrightarrow
\alpha_{12,1}.
\]
In the fixed greedy order, $\alpha_{12,1}$ is selected before
$\alpha_{12,12}$. The latter is therefore rejected, and
\[
\mathcal B_2^{\mathrm{acyc}}
=
\left\{
\alpha_{1,12},
\alpha_{12,1},
\alpha_{12,2}
\right\}.
\]
The second reduction removes these three row $2$-cells together with
their matched Nielsen-transformed generating $3$-cells. Hence
$\Sigma^{\mathrm{row,red}}(2)$ has the five generating $2$-cells
\[
r_1r_1\odfl{\alpha_{1,1}}r_1,
\qquad
r_1r_2\odfl{\alpha_{1,2}}r_{12},
\qquad
r_2r_1\odfl{\alpha_{2,1}}r_2r_{12},
\qquad
r_2r_2\odfl{\alpha_{2,2}}r_2,
\qquad
r_{12}r_{12}
\odfl{\alpha_{12,12}}
r_2r_{12}.
\]
Moreover,
\[
\bigl|\Sigma^{\mathrm{row,red}}_3(2)\bigr|
=
16-3
=
13.
\]

\subsection{Passage to the pre-row presentation}

The pre-row presentation has the same generating $1$-cells
$r_1,r_2,r_{12}$ and the five generating $2$-cells
\[
r_1r_2\odfl{\gamma_{12}}r_{12},
\qquad
r_1r_1\odfl{\eta_1^r}r_1,
\qquad
r_2r_2\odfl{\eta_2^r}r_2,
\qquad
r_2r_1r_1
\odfl{\kappa_{1,1,2}^r}
r_1r_2r_1,
\qquad
r_2r_2r_1
\odfl{\lambda_{1,2,2}^r}
r_2r_1r_2.
\]
We apply the Tietze sequence
$\mathcal T_2^{\mathrm{red}}$ constructed in
Subsection~\ref{SS:CoherentPresentationLetterGenerators}.
The five pre-row cells are first adjoined, and the five surviving row
cells are then eliminated using their canonical pre-row paths. Three
of these paths are immediate:
\[
P_{1,1}=\eta_1^r,
\qquad
P_{1,2}=\gamma_{12},
\qquad
P_{2,2}=\eta_2^r.
\]
The other two are
\[
\begin{aligned}
P_{2,1}\colon\quad
r_2r_1
&\odfl{(\eta_2^r1_{r_1})^{-1}}
r_2r_2r_1
\odfl{\lambda_{1,2,2}^r}
r_2r_1r_2
\odfl{1_{r_2}\gamma_{12}}
r_2r_{12},
\end{aligned}
\]
and
\[
P_{12,12}\colon\quad
r_{12}r_{12}\odfl{(\gamma_{12}1_{r_{12}})^{-1}}
r_1r_2r_{12}
\odfl{1_{r_1}1_{r_2}\gamma_{12}^{-1}}
r_1r_2r_1r_2
\odfl{(\kappa_{1,1,2}^r1_{r_2})^{-1}}
r_2r_1r_1r_2
\odfl{1_{r_2}\eta_1^r1_{r_2}}
r_2r_1r_2
\odfl{1_{r_2}\gamma_{12}}
r_2r_{12}.
\]
At the coherent level, adjoining each pre-row cell~$\delta$ introduces
one comparison $3$-cell
$B_\delta\colon Q_\delta\Rrightarrow\delta$. Thus five comparison
$3$-cells are added. To eliminate the five surviving row cells, one
adjoins the auxiliary cells~$D_{u,v}$ and then eliminates each
$D_{u,v}$ together with the corresponding row cell. Hence the
auxiliary cells~$D_{u,v}$ do not survive, whereas the five comparison
cells~$B_\delta$ do. Consequently,
\[
\bigl|\Sigma^{\mathrm{prerow,coh}}_3(2)\bigr|
=
\bigl|\Sigma^{\mathrm{row,red}}_3(2)\bigr|+5
=
13+5
=
18.
\]

\subsection{The coherent letter presentation}

The compression cell~$
\gamma_{12}\colon
r_1r_2\dfl r_{12}$
eliminates the non-letter generator~$r_{12}$. After replacing every
occurrence of~$r_{12}$ by~$r_1r_2$ and identifying $r_1$ and $r_2$
with the letters $1$ and $2$, respectively, the surviving generating
$2$-cells are
\[
\Sigma^{\mathrm{styl}}_2(2)
=
\left\{
\eta_1\colon11\dfl1,
\quad
\eta_2\colon22\dfl2,
\quad
\kappa_{1,1,2}\colon211\dfl121,
\quad
\lambda_{1,2,2}\colon221\dfl212
\right\}.
\]
Together with the generating $1$-cells $1$ and $2$, these cells form
the letter presentation~$\Sigma^{\mathrm{styl}}(2)$.

The final reduction eliminates only the compression
cell~$\gamma_{12}$ and its redundant target~$r_{12}$. In particular,
it eliminates no generating $3$-cell and only transforms the
boundaries of the eighteen generating $3$-cells. Hence~$\bigl|\Sigma^{\mathrm{styl,coh}}_3(2)\bigr|=18$.
By Theorem~\ref{T:LetterCoherentPresentation}, these generating
$3$-cells form a finite homotopy basis of~$\Sigma^{\mathrm{styl}}(2)_2^\top$. Therefore
$\Sigma^{\mathrm{styl,coh}}(2)$ is a finite coherent presentation of
$\Styl_2$.

\subsection{The resulting homotopy basis}

The generating $3$-cells surviving the second row reduction are indexed by
\[
\mathfrak T_2^{\mathrm{red}}
:=
\mathfrak T_2^{(1)}
\setminus
\left\{
(1,1,2),(1,2,1),(1,2,2)
\right\}
=
\left\{
\begin{aligned}
&(1,1,1),(1,1,12),(1,12,1),(1,12,2),(1,12,12),
\\
&(2,1,1),(2,1,2),(2,1,12),(2,2,1),(2,2,2),
\\
&(12,12,1),(12,12,2),(12,12,12)
\end{aligned}
\right\}.
\]
For~$\tau$ in~$\mathfrak T_2^{\mathrm{red}}$, denote by
$\mathcal X_\tau^{\mathrm{prerow}}$ the transported image of
$\mathcal X_\tau$, and set~$
\mathcal X_\tau^{\mathrm{let}}
:=
R_{\Gamma^{\mathrm{let}}(2)}
\bigl(
\mathcal X_\tau^{\mathrm{prerow}}
\bigr)$.
Similarly, for every~$\delta$ in
$\Sigma^{\mathrm{prerow}}_2(2)$, denote by
$B_\delta^{\mathrm{prerow}}$ the transported image of~$B_\delta$ and
set~$
B_\delta^{\mathrm{let}}
:=
R_{\Gamma^{\mathrm{let}}(2)}
\bigl(
B_\delta^{\mathrm{prerow}}
\bigr)$.

The resulting homotopy basis is
\[
\mathcal H_2
={}
\left\{
\mathcal X_\tau^{\mathrm{let}}
\ \middle|\
\tau\in\mathfrak T_2^{\mathrm{red}}
\right\}
\sqcup
\left\{
B_{\gamma_{12}}^{\mathrm{let}},
B_{\eta_1^r}^{\mathrm{let}},
B_{\eta_2^r}^{\mathrm{let}},
B_{\kappa_{1,1,2}^r}^{\mathrm{let}},
B_{\lambda_{1,2,2}^r}^{\mathrm{let}}
\right\}.
\]
It contains thirteen transported row confluences and five comparison
$3$-cells, hence eighteen generating $3$-cells.
For instance,
\[
\mathcal X_{1,1,1}^{\mathrm{let}}
\colon
(\eta_1 1_1)\star_1\eta_1
\Rrightarrow
(1_1\eta_1)\star_1\eta_1.
\]
The other boundary paths are obtained by the transport and elimination
procedures described above. 

The successive cell counts are
\[
\begin{array}{|c|c|c|c|}
\hline
\text{presentation}
& \#1\text{-cells}
& \#2\text{-cells}
& \#3\text{-cells}
\\ \hline
\Sigma^{\mathrm{row,coh}}(2)
&3&8&21\\
\Sigma^{\mathrm{row,scoh}}(2)
&3&8&16\\
\Sigma^{\mathrm{row,red}}(2)
&3&5&13\\
\Sigma^{\mathrm{prerow,coh}}(2)
&3&5&18\\
\Sigma^{\mathrm{styl,coh}}(2)
&2&4&18\\ \hline
\end{array}
\]

\appendix

\section{Verification of the explicit generating confluences}
\label{A:ExplicitGeneratingConfluences}

We verify the confluence diagrams displayed in Subsubsection~\ref{SSS:ExplicitGeneratingConfluences}. Fix a critical triple~$(u,v,t)\in\mathfrak C_n$. In each case, we show that every displayed rewriting step is well defined and that both branches reach the common normal form~$\operatorname{NF}(u,v,t)$.

We repeatedly use the following facts. The row-generator reading of an
$N$-tableau is in normal form, and the $N$-tableau~$N(uv)$ associated
with two rows~$u$ and~$v$ has at most two rows. Moreover, by convergence
of the row presentation, two normal forms representing the same stylic
element are equal.
\medskip

\noindent
\textbf{Case $1$.}
Assume that $N(uv)$ and $N(vt)$ both have one row. Write
$\R(N(uv))=b$ and $\R(N(vt))=d$. By Lemma~\ref{L:OneRowCriterion},
one has $b=u\cup v$, $d=v\cup t$, $\min(v)\ge\max(u)$, and
$\min(t)\ge\max(v)$. Thus $m=u\cup v\cup t$ is the one-row normal form.
Moreover, $\max(b)=\max(v)$ and $\min(t)\ge\max(v)=\max(b)$, so
Lemma~\ref{L:OneRowCriterion} gives $\R(N(bt))=m$. Similarly,
$\min(d)=\min(v)\ge\max(u)$, so $\R(N(ud))=m$. Hence the two branches are
exactly those displayed in \textbf{Case~$1$}.

\medskip

\noindent
\textbf{Case $2$.}
Assume that $N(uv)$ has one row and $N(vt)$ has two rows. Write
$\R(N(uv))=b$ and $\R(N(vt))=~cd$ with $c\neq\varepsilon$. By
Lemma~\ref{L:UsefulReduciblePairsCriticalTriple}, the $2$-cells~$\alpha_{b,t}$
and $\alpha_{u,c}$ are applicable. Since $N(uv)$ has one row, one has
$b=u\cup v$ and $\min(v)\ge\max(u)$. Since $c$ is the top row of $N(vt)$,
Lemma~\ref{L:TopRowInclusion} gives $\supp(c)\subseteq\supp(v)$, hence
$\min(c)\ge\min(v)\ge\max(u)$. Therefore $N(uc)$ has one row by
Lemma~\ref{L:OneRowCriterion}. Write $\R(N(uc))=g$.
If $g\triangleleft d$, then $r_g r_d$ is already in normal form. On the
left branch, the cell~$\alpha_{b,t}$ has target~$\Rr(N(bt))$, which is also
in normal form.
Uniqueness of normal forms therefore gives~$\R(N(bt))=gd$.
Thus both branches end at
$\operatorname{NF}(u,v,t)=r_g r_d$, giving \textbf{Case~$2a$}.
If $g\not\triangleleft d$, the right branch takes the additional step
$\alpha_{g,d}$. Write~$\R(N(gd))=pq$,
with $p\in\Rr_n\cup\{\varepsilon\}$ and $q\in\Rr_n$. The right branch then
reaches the normal form~$r_p r_q$, with~$r_p$ omitted when
$p=\varepsilon$. Since~$\Rr(N(bt))$ is also a normal form, uniqueness gives~$\R(N(bt))=pq$.
This is \textbf{Case~$2b$}.

\medskip

\noindent
\textbf{Case $3$.}
Assume that $N(uv)$ has two rows and $N(vt)$ has one row. Write
$\R(N(uv))=ab$ with $a\neq\varepsilon$, and write $\R(N(vt))=d$.
By Lemma~\ref{L:UsefulReduciblePairsCriticalTriple}, the $2$-cell~$\alpha_{b,t}$ is applicable. Write $\R(N(bt))=ef$, with
$e\in\Rr_n\cup\{\varepsilon\}$ and $f\in\Rr_n$.
Since $N(vt)$ has one row, Lemma~\ref{L:OneRowCriterion} gives
$d=v\cup t$ and $\min(t)\ge\max(v)$. On the right branch, the word
$r_u r_d$ is either already in normal form if $u\triangleleft d$, or else it
reduces by $\alpha_{u,d}$.

Suppose first that $e=\varepsilon$, so $\R(N(bt))=f$. Since
$a\triangleleft b$, one has
\[
\supp(b)\subseteq\supp(f)\quad  \text{ and } \quad \min(f)\le\min(b)<\min(a).
\]
Lemma~\ref{L:TriangleleftMonotonicity} gives
$a\triangleleft f$. Hence the left branch is already in normal form at
$r_a r_f$, and $\operatorname{NF}(u,v,t)=r_a r_f$.
If $u\triangleleft d$, then the right branch is already in normal form at
$r_u r_d$, and uniqueness of normal forms gives $r_u r_d=r_a r_f$. This is
\textbf{Case~$3a1$}. 
If $u\not\triangleleft d$, then the $2$-cell~$\alpha_{u,d}$ applies,
and uniqueness gives $\R(N(ud))=af$. This is \textbf{Case~$3a2$}.

Suppose now that $e\neq\varepsilon$. Then the left branch reaches
$r_a r_e r_f$. Since the right branch factors through the product of two rows
$r_u r_d$, its normal form has at most two rows by
Lemma~\ref{L:TwoRowsProductRows}. Hence $r_a r_e r_f$ cannot already be in
normal form. Therefore $a\not\triangleleft e$, and the $2$-cell~$\alpha_{a,e}1_{r_f}$ applies.
Write $\R(N(ae))=pq$. If $p\neq\varepsilon$, then $p\triangleleft q$.
By Lemma~\ref{L:BottomRowUnion}, $\supp(q)=\supp(a)\cup\supp(e)$. Since
$a\triangleleft b$, one has $\supp(a)\subseteq\supp(b)$, and since
$\R(N(bt))=ef$, one has $\supp(b)\subseteq\supp(f)$ and
$\supp(e)\subseteq\supp(f)$. Hence $\supp(q)\subseteq\supp(f)$. Moreover,
$\min(f)\le\min(b)<\min(a)$ and $\min(f)<\min(e)$, so
$\min(f)<\min(q)$. Thus $q\triangleleft f$, and~$r_p r_q r_f$ would be a
three-row normal form. This contradicts the fact that the common normal form
has at most two rows. Hence $p=\varepsilon$. Write $\R(N(ae))=g$. Then
$\operatorname{NF}(u,v,t)=r_g r_f$.
If $u\triangleleft d$, then the right branch is already in normal form and
$r_u r_d=r_g r_f$. This is \textbf{Case~$3b1$}. If $u\not\triangleleft d$, then
$\alpha_{u,d}$ applies and uniqueness gives $\R(N(ud))=gf$. This is
\textbf{Case~$3b2$}.

\medskip

\noindent
\textbf{Case $4$.}
Assume that both $N(uv)$ and $N(vt)$ have two rows. Write
$\R(N(uv))=ab$ with $a\neq\varepsilon$, and write $\R(N(vt))=cd$ with
$c\neq\varepsilon$. By Lemma~\ref{L:BottomRowUnion}, one has
\[
\supp(b)=\supp(u)\cup\supp(v) \quad  \text{ and }\quad \supp(d)=\supp(v)\cup\supp(t). 
\]
By Lemma~\ref{L:TopRowInclusion},
$\supp(c)\subseteq\supp(v)\subseteq\supp(b)$.
Since $N(vt)$ has two rows, Lemma~\ref{L:OneRowCriterion} gives
\[
\min(t)<\max(v)\le\max(b),
\]
 so $N(bt)$ has two rows. Write
$\R(N(bt))=ef$ with $e\neq\varepsilon$. Then
$\supp(f)=\supp(u)\cup\supp(v)\cup\supp(t)$.
By Lemma~\ref{L:UsefulReduciblePairsCriticalTriple}, the $2$-cells
$\alpha_{b,t}$ and $\alpha_{u,c}$ are applicable. Thus the left branch always
reaches $r_a r_e r_f$. On the right branch, write $\R(N(uc))=ij$, with
$i\in\Rr_n\cup\{\varepsilon\}$ and $j\in\Rr_n$.

Suppose first that $N(uc)$ has one row. Write $\R(N(uc))=g$. The right branch
reaches $r_g r_d$.
If $g\triangleleft d$, then $r_g r_d$ is already in normal form, so
$\operatorname{NF}(u,v,t)=r_g r_d$. Since this normal form has two rows,
$r_a r_e r_f$ cannot be normal, hence $a\not\triangleleft e$. Write
$\R(N(ae))=pq$. The same support argument as in \textbf{Case~$3b$} shows that
$p=\varepsilon$. Write $\R(N(ae))=g'$. Then the left branch reaches
$r_{g'}r_f$, and uniqueness gives $g'=g$ and~$f=d$. This is \textbf{Case~$4a1$}.
If $g\not\triangleleft d$, then the right branch reduces the pair $(g,d)$.
Since this is a product of two rows, its normal form has at most two rows.
Hence the common normal form has at most two rows, so again $r_a r_e r_f$
cannot already be in normal form. Thus $a\not\triangleleft e$. Write
$\R(N(ae))=\rho\sigma$. The same support argument gives $\rho=\varepsilon$.
Write $\R(N(ae))=p'$. The left branch reaches $r_{p'}r_f$.
Now write $\R(N(gd))=ph$. By Lemma~\ref{L:BottomRowUnion},
$\supp(h)=\supp(g)\cup\supp(d)$. Since
$\supp(g)=\supp(u)\cup\supp(c)$, $\supp(d)=\supp(v)\cup\supp(t)$, and
$\supp(c)\subseteq\supp(v)$, one gets
\[
\supp(h)=\supp(u)\cup\supp(v)\cup\supp(t)=\supp(f).
\]
Thus $h=f$.
We claim that $N(gd)$ has two rows. Since $c$ is the top row of $N(vt)$,
some letter of $c$ is bumped during the insertion of the letters of $t$ into
$v$. Hence there exist $x\in\supp(t)$ and $y\in\supp(c)$ such that $x<y$.
Since $x\in\supp(d)$ and $y\in\supp(c)\subseteq\supp(g)$, one has~$\min(d)\le x<y\le\max(g)$.
Thus $\min(d)<\max(g)$, and Lemma~\ref{L:OneRowCriterion} implies that
$N(gd)$ has two rows. Hence $p\neq\varepsilon$, and $\R(N(gd))=~pf$.
Uniqueness gives $p'=p$. This is \textbf{Case~$4a2$}.

Suppose now that $N(uc)$ has two rows. Write $\R(N(uc))=ij$ with
$i\neq\varepsilon$. Then the right branch reaches $r_i r_j r_d$ and
$i\triangleleft j$. By Lemma~\ref{L:BottomRowUnion},
$\supp(j)=\supp(u)\cup\supp(c)$. Applying the same lemma to~$N(jd)$, the
bottom row of $N(jd)$ has support
\[
\supp(j)\cup\supp(d)
=
\supp(u)\cup\supp(c)\cup\supp(v)\cup\supp(t).
\]
Since $\supp(c)\subseteq\supp(v)$, this is
$\supp(u)\cup\supp(v)\cup\supp(t)=\supp(f)$. Hence the bottom row of~$N(jd)$ is $f$.
If $a\triangleleft e$ and $j\triangleleft d$, then
$r_a r_e r_f$ and $r_i r_j r_d$ are in normal form. Uniqueness gives
$i=a$, $j=e$, and $d=f$. This is \textbf{Case~$4b1$}.
If $a\triangleleft e$ and $j\not\triangleleft d$, then the left branch is
already in normal form at $r_a r_e r_f$. Since $i\neq\varepsilon$ and
$j\not\triangleleft d$, the right-branch analysis in the proof of
Proposition~\ref{Prop:boundedBranching} shows that~$N(jd)$ has two rows
and, writing $\R(N(jd))=sf$, that $i\triangleleft s$. Thus
$r_i r_s r_f$ is in normal form. Uniqueness gives $i=a$ and $s=e$.
This is \textbf{Case~$4b2$}.
If $a\not\triangleleft e$ and $j\triangleleft d$, then the right branch is
already in normal form at~$r_i r_j r_d$, which has three rows. Hence the
reduction of $(a,e)$ on the left must produce two rows. Write~$\R(N(ae))=pq$ with~$p\neq\varepsilon$. The left-branch analysis in the
proof of Proposition~\ref{Prop:boundedBranching} gives
$q\triangleleft f$, so~$r_p r_q r_f$ is in normal form. Uniqueness gives
$p=i$, $q=j$, and $f=d$. This is \textbf{Case~$4b3$}.
Finally, if $a\not\triangleleft e$ and $j\not\triangleleft d$, then both
branches take one more step. Since $i\neq\varepsilon$ and
$j\not\triangleleft d$, the right-branch analysis in the proof of
Proposition~\ref{Prop:boundedBranching} shows that $N(jd)$ has two rows
and, writing $\R(N(jd))=sf$, that $i\triangleleft s$. Thus~$r_i r_s r_f$ is a three-row normal form. Hence the reduction of
$(a,e)$ on the left must also produce two rows. Write~$\R(N(ae))=pq$ with~$p\neq\varepsilon$. The left-branch analysis in the
proof of Proposition~\ref{Prop:boundedBranching} gives~$q\triangleleft f$, so~$r_p r_q r_f$ is also in normal form.
Uniqueness gives $p=i$ and $q=s$. This is \textbf{Case~$4b4$}.

This proves that all displayed diagrams are well defined and that their common
target is $\operatorname{NF}(u,v,t)$.

\begin{remark}
\label{R:AllCasesRealized}
For $n\geq 3$, each of the thirteen subcases appearing in the case analysis is
realized by at least one critical triple $(u,v,t)\in\Rr_n^3$. For instance, one
may take
\[
\begin{array}{|c|c|c|c|c|c|c|c|}
\hline\text{Case}
& 1 & 2\mathrm a & 2\mathrm b
& 3\mathrm{a1} & 3\mathrm{a2} & 3\mathrm{b1} & 3\mathrm{b2}
\\ \hline
(u,v,t)
& (1,1,1)
& (2,2,1)
& (1,2,1)
& (2,1,2)
& (12,1,2)
& (2,1,12)
& (2,1,1)\\\hline
\end{array}
\]
\[
\begin{array}{|c|c|c|c|c|c|c|}
\hline
\text{Case}
& 4\mathrm{a1} & 4\mathrm{a2}
& 4\mathrm{b1} & 4\mathrm{b2} & 4\mathrm{b3} & 4\mathrm{b4}
\\ \hline
(u,v,t)
& (2,13,2)
& (12,12,1)
& (3,12,123)
& (3,12,12)
& (3,2,13)
& (3,2,1)\\\hline
\end{array}
\]
where a word such as $12$ or $123$ denotes the corresponding row.
\end{remark}

\quad

\begin{small}
\renewcommand{\refname}{\Large\textsc{References}}
\bibliographystyle{plain}
\bibliography{biblioCURRENT}
\end{small}

\vfill
\begin{flushright}
\begin{small}
\noindent\textsc{Nohra Hage}\\
\href{mailto:nohra.hage@univ-catholille.fr}
     {\nolinkurl{nohra.hage@univ-catholille.fr}}\\
ICL, Junia, Université Catholique de Lille,\\
LITL, F-59000 Lille, France
\end{small}
\end{flushright}

\vspace{0.25cm}

\begin{small}---\;\;\today\;\;-\;\;\hhmm\;\;---\end{small} \hfill

\end{document}